\documentclass[11pt,a4paper]{amsart}

\usepackage{fullpage}
\usepackage{amsmath}
\usepackage{amsfonts}
\usepackage{amssymb}
\usepackage{mathrsfs}
\usepackage{mathtools}
\usepackage{bbm}
\usepackage{tikz-cd}
\usepackage{url}
\usepackage[english]{babel}
\usepackage[utf8]{inputenc}

\usepackage{csquotes}
\usepackage{enumitem}
\usepackage{multicol}
\usepackage{calc}

\setlist[itemize,1]{label={--\,}}
\setlist[enumerate,1]{label=(\roman*)}

\usepackage[backend=biber, maxbibnames=50, style=alphabetic]{biblatex}
\renewbibmacro{in:}{}

\usepackage{hyperref}

\usepackage{crossreftools}

\makeatletter
\newcommand{\optionaldesc}[2]{%
	\phantomsection
	#1\protected@edef\@currentlabel{#1}\label{#2}%
}
\makeatother

\makeatletter
\newcommand{\doublewidetilde}[1]{{%
		\mathpalette\double@widetilde{#1}%
}}
\newcommand{\double@widetilde}[2]{%
	\sbox\z@{$\m@th#1\widetilde{#2}$}%
	\ht\z@=.9\ht\z@
	\widetilde{\box\z@}%
}

\makeatother

\DeclareRobustCommand{\gobblefour}[4]{}

\makeatletter
\renewcommand{\tocsection}[3]{%
  \indentlabel{\@ifnotempty{#2}{\ignorespaces#1 \makebox[\widthof{00.}][l]{#2.}\quad}}#3}
\renewcommand{\tocsubsection}[3]{%
  \indentlabel{\@ifnotempty{#2}{\ignorespaces#1 \makebox[\widthof{00.0.}][l]{#2.}\quad}}#3}
\makeatother

\setlist[itemize]{leftmargin=*}
\setlist[enumerate]{leftmargin=*}

\newtheorem{thm}{Theorem}[section]
\newtheorem{defin}[thm]{Definition}
\newtheorem{prop}[thm]{Proposition}
\newtheorem{cor}[thm]{Corollary}
\newtheorem{cordef}[thm]{Corollary/Definition}

\newtheorem{lemma}[thm]{Lemma}

\theoremstyle{definition}
\newtheorem{ex}[thm]{Example}
\newtheorem{rem}[thm]{Remark}

\newtheorem{innercustomthm}{Theorem}
\newenvironment{customthm}[1]
{\renewcommand\theinnercustomthm{#1}\innercustomthm}
{\endinnercustomthm}

\newtheorem{innercustomdef}{Definition}
\newenvironment{customdef}[1]
{\renewcommand\theinnercustomdef{#1}\innercustomdef}
{\endinnercustomdef}

\newcommand{\A}{\mathbb A}

\newcommand{\C}{\mathbb C}

\newcommand{\F}{\mathbb F}
\newcommand{\G}{\mathbb G}

\newcommand{\I}{\mathbb I}

\newcommand{\N}{\mathbb N}
\newcommand{\Q}{\mathbb Q}
\newcommand{\R}{\mathbb R}
\newcommand{\T}{\mathbb T}

\newcommand{\Z}{\mathbb Z}

\newcommand{\ab}{\mathrm{ab}}

\newcommand{\alg}{\mathrm{alg}}

\newcommand{\Art}{\mathrm{Art}}

\newcommand{\BC}{\mathrm{BC}}

\newcommand{\cl}{\mathrm{cl}}
\newcommand{\Cl}{\mathrm{Cl}}

\newcommand{\CNL}{{\mathrm{CNL}}}
\newcommand{\cont}{\mathrm{cont}}

\newcommand{\cyc}{\mathrm{cyc}}

\newcommand{\PsDef}{{\mathrm{PsDef}}}

\newcommand{\diag}{\mathrm{diag}\,}
\newcommand{\dist}{\mathrm{dist}}
\newcommand{\dR}{{\mathrm{dR}}}
\newcommand{\End}{\mathrm{End}}
\newcommand{\ev}{\mathrm{ev}}

\newcommand{\Fil}{\mathrm{Fil}}
\newcommand{\fin}{\mathrm{fin}}
\DeclareMathOperator{\Frac}{Frac}
\newcommand{\Frob}{\mathrm{Frob}}

\newcommand{\full}{\mathrm{full}}
\newcommand{\Gal}{{\mathrm{Gal}}}
\newcommand{\GL}{\mathrm{GL}}

\newcommand{\Hom}{\mathrm{Hom}}

\newcommand{\id}{\mathrm{id}}
\newcommand{\Id}{\mathrm{Id}}

\newcommand{\ind}{\mathrm{ind}}
\newcommand{\Ind}{\mathrm{Ind}}

\newcommand{\Nm}{\mathrm{Norm}}
\newcommand{\nord}{\mathrm{n.ord}}

\newcommand{\Norm}{\mathrm{Norm}}
\newcommand{\nr}{\mathrm{nr}}

\newcommand{\ord}{\mathrm{ord}}

\newcommand{\PGL}{\mathrm{PGL}}

\newcommand{\Ps}{{\mathrm{Ps}}}

\newcommand{\ptnt}{\mathrm{ptnt}}

\DeclareMathOperator{\Proj}{Proj}

\newcommand{\ps}{\mathrm{ps}}
\newcommand{\psc}{{\mathrm{psc}}}
\newcommand{\pst}{\mathrm{pst}}
\newcommand{\ptri}{{\mathrm{ptri}}}

\newcommand{\rec}{\mathrm{rec}}

\DeclareMathOperator{\res}{res}
\DeclareMathOperator{\Res}{Res}
\newcommand{\rig}{\mathrm{rig}}
\newcommand{\Rig}{\mathrm{Rig}}

\newcommand{\Sets}{\mathrm{Sets}}

\DeclareMathOperator{\Spec}{Spec}
\DeclareMathOperator{\Spf}{Spf}
\DeclareMathOperator{\Spm}{Spm}

\newcommand{\sms}{{\mathrm{ss}}}
\newcommand{\st}{{\mathrm{st}}}

\newcommand{\tdR}{\mathrm{tdR}}

\newcommand{\tri}{{\mathrm{tri}}}
\newcommand{\tw}{{\mathrm{tw}}}
\newcommand{\triv}{{\mathrm{triv}}}
\newcommand{\univ}{\mathrm{univ}}

\newcommand{\xto}{\xrightarrow}
\newcommand{\into}{\hookrightarrow}
\newcommand{\onto}{\twoheadrightarrow}
\newcommand{\ovl}{\overline}
\newcommand{\wtl}{\widetilde}

\newcommand{\ccirc}{\kern0.5ex\vcenter{\hbox{$\scriptstyle\circ$}}\kern0.5ex}

\newcommand{\calC}{\mathcal{C}}

\newcommand{\cE}{\mathscr{E}}
\newcommand{\cF}{\mathscr{F}}
\newcommand{\cG}{\mathcal{G}}

\newcommand{\cK}{\mathcal{K}}
\newcommand{\cL}{\mathcal{L}}

\newcommand{\cO}{\mathcal{O}}

\newcommand{\cR}{\mathcal{R}}
\newcommand{\calR}{\mathcal{R}}

\newcommand{\cS}{\mathscr{S}}

\newcommand{\cU}{\mathcal{U}}

\newcommand{\cW}{\mathcal{W}}
\newcommand{\cX}{\mathcal{X}}
\newcommand{\cZ}{\mathcal{Z}}
\newcommand{\calH}{\mathcal{H}}

\newcommand{\bB}{{\mathbf{B}}}

\newcommand{\bk}{{\mathbf{k}}}

\newcommand{\bP}{{\mathbf{P}}}

\newcommand{\bx}{{\mathbf{x}}}

\newcommand{\Fp}{\overline{\F}_p}
\newcommand{\Qp}{\overline{\Q}_p}
\newcommand{\Zp}{\overline{\Z}_p}

\newcommand{\fa}{{\mathfrak{a}}}

\newcommand{\fb}{{\mathfrak{b}}}
\newcommand{\fc}{{\mathfrak{c}}}
\newcommand{\fd}{{\mathfrak{d}}}

\newcommand{\fh}{{\mathfrak{h}}}

\newcommand{\fl}{{\mathfrak{l}}}

\newcommand{\fm}{{\mathfrak{m}}}

\newcommand{\fp}{{\mathfrak{p}}}

\newcommand{\fR}{{\mathfrak{R}}}

\newcommand{\vareps}{{\varepsilon}}

\title{$p$-adic rigidity of eigenforms of infinite slope}
\author{Andrea Conti}

\begin{document}

\begin{abstract}
We give a notion of $p$-adic families of Hecke eigenforms that allows for the slope of the forms be infinite at $p$. We prove that, contrary to the case of finite slope when every eigenform lives in a Hida or Coleman family, the only families of infinite slope are either twists of Hida or Coleman families with Dirichlet characters of $p$-power conductor, or non-ordinary families with complex multiplication. 
Our proof goes via a local study of deformations of potentially trianguline Galois representations, relying on work of Berger and Chenevier, and a global input coming from an analogue of a result of Balasubramanyam, Ghate and Vatsal on a Greenberg-type conjecture for families of Hilbert modular forms.
\end{abstract}

\maketitle

\section*{Introduction}



Let $p$ be a prime number, and $N$ a natural number prime to $p$. We write $\Qp$ for an algebraic closure of $\Q_p$, and $\C_p$ for its completion with respect to a $p$-adic valuation. We denote by $G_K$ the absolute Galois group of a field $K$. In this introduction, all eigenforms for $\GL_{2/\Q}$ are assumed to be cuspidal even when it is not specified.

Consider a cuspidal modular form $f$ of weight $k$ and level $\Gamma_1(Np^r)$ for some $r\ge 1$, and assume that $f$ is an eigenform for the action of a Hecke algebra $\calH$ generated by operators $T_\ell$, $\ell\nmid Np$, and $U_p$; we will write $\calH^{Np}$ for the spherical factor of $\calH$. We say that $f$ is of \emph{finite slope} if its $U_p$-eigenvalue is non-zero.
When this is the case, classical constructions of Hida, Coleman, Mazur, Buzzard and Chenevier allow one to insert $f$ in a \emph{$p$-adic family}. In a weak sense, this means that one can find a ``Kummer family'' around $f$: for every $k^\prime$ $p$-adically sufficiently close to $k$, there exists an eigenform $f_{k^\prime}$ of tame level $N$ and finite slope whose Hecke eigensystem is $p$-adically close to that of $f$; in particular, as $k^\prime$ approaches $k$, the Hecke eigensystems of the forms $f_{k^\prime}$ converge to that of $f$. In a stronger sense, one can find a ``big'' Hecke eigensystem $\Theta\colon\calH\to\I$, where $\I$ is a $\Z_p[[T]]$-algebra, finite as a $\Z_p[[T]]$-module, or some variation of it, with the property that $\Theta$ specializes to the Hecke eigensystem of $f$ at a certain height 1 prime of $\I$, and to the eigensystems of other eigenforms on a dense set of height 1 primes of $\I$. Such Hecke eigensystems can be glued over a rigid analytic object, the \emph{eigencurve} $\cE$, a rigid analytic curve over $\Q_p$ equipped with a Hecke eigensystem $\Theta\colon\calH\to\cO_\cE^\circ(\cE)$, valued in the ring of power-bounded functions on $\cE$, and a weight map to the weight space $\cW$, defined as the rigid generic fiber of $\Spf\Z_p[[\Z_p^\times]]$. We say that a point $x$ of $\cE$ is \emph{classical} if $\Theta$ specializes at $x$ to the eigensystem of a classical eigenform; non-classical points (that make up most of $\cE$) correspond to Hecke eigensystems of \emph{overconvergent} eigenforms. By cutting out a sufficiently small neighborhood of a classical point of the eigencurve, one recovers $p$-adic families of eigenforms in the previous sense. The locus on $\cE$ where $\Theta(U_p)$ is of valuation zero is the generic fiber of the big ordinary Hecke algebra constructed by Hida.

The above picture is not complete, in that it does not encompass eigenforms of infinite slope, i.e. forms that appear in the kernel of the $U_p$-operator. The construction of the eigencurve revolves around the idea that most of the information on the Hecke action on spaces of overconvergent eigenforms is encoded in the characteristic series of the $U_p$-operator. Such a series is a well-defined object since $U_p$ acts compactly on these spaces, in the sense of Serre, and it cuts a (spectral) rigid curve $\cZ$ inside of the product $\cW\times\A_1$; then $\cE$ is constructed as a finite cover of $\cZ$. This construction would certainly make no sense for eigenforms of infinite slope, since $U_p$ is the zero operator. The question then remains open of whether one can construct $p$-adic families of infinite slope, with one of the meanings sketched above. 

To our knowledge, there has been so far limited work on this question. Diao and Liu \cite{diaoliu} answered a question of Coleman and Mazur by showing that the eigencurve is ``complete'' (i.e. satisfies the valuative criterion for properness), so that no eigensystems of infinite slope appear as a limit of eigensystems along the eigencurve. On the other hand, Coleman and Stein \cite{colste} showed that by moving in a ``transversal'' direction, rather than along the eigencurve, one can construct sequences of eigenforms of finite slope that converge to an eigenform of infinite slope obtained by twisting an eigenform of finite slope with a Dirichlet character of $p$-power conductor.

The kind of eigenform of infinite slope appearing in the work of Coleman and Stein fits in the first of two possible cases: the local Langlands correspondence at $p$ shows that an eigenform $f$ is of infinite slope if and only if the associated admissible representation of $\GL_2(\Q_p)$ is either a ramified twist of a Steinberg, or a principal series attached to two ramified characters (two options that we group together), in which case $f$ is a twist of an eigenform of finite slope with a Dirichlet character of $p$-power conductor, or a supercuspidal representation. In this second case, we say that $f$ is \emph{$p$-supercuspidal}, and there is no way to write it as a twist of an eigenform of finite slope. It is trivial to construct families of eigenforms of the first kind: if, for instance, $\Theta\colon\calH\to\I$ is the Hecke eigensystem of a Hida or Coleman family, specializing at classical points to some set of eigenforms $\{f_x\}_x$ of finite slope, then one can twist $\Theta$ with a Dirichlet character $\delta$ in an obvious way, with the result that the new eigensystem $\delta\Theta$ specializes to the twists of the eigenforms $f_x$ with $\delta$. However, twisting forces one to lose any information at the primes dividing the conductor of $\delta$, since all of the Hecke operators at such primes are mapped to 0 after twisting. 

As the previous example suggests, a na\"ive definition of family, as a $p$-adic interpolation of Hecke eigensystem, would be too general to be meaningful: for eigenforms of infinite slope, interpolating Hecke eigensystems amounts to interpolating Hecke eigenvalues at unramified primes, i.e. interpolating Galois (pseudo)representations (including Hecke operators at primes dividing the tame level, if any, would not lead to a much stronger notion). Then any open subdomain in the (pseudo)deformation space of a mod $p$ modular representation would constitute a family of eigenforms, as long as modular points are dense in it. 
In light of this, we wish to give a more restrictive definition of family, that resembles the classical ones for finite slope eigenforms, including some information at the prime $p$. In the case of finite slope, the $U_p$-eigenvalue of an eigenform $f$ can also be read on the Galois side, as one of the eigenvalues of the Frobenius operator on the $D_\pst$ of the $p$-adic, local-at-$p$ representation $\rho_{f,p}\vert_{G_{\Q_p}}$ attached to $f$. We call such an eigenvalue \emph{admissible} if it gives the action of Frobenius on an $N$-stable line, where $N$ is the monodromy operator (an empty condition if the representation is crystalline). This leads us to the following definitions. 

\begin{customdef}{1}[cf. Definitions \ref{famdef} and \ref{afffamdef}]\label{introdef}
A \emph{$p$-adic family of eigenforms} is a sequence $(f_i,\varphi_{i})_{i\in\N}$ where:
\begin{itemize}
\item the $f_i$ are pairwise distinct eigenforms of tame level $N$, whose Hecke eigensystems $\theta_i\colon\calH^{Np}\to\Qp$ away from $Np$ converge $p$-adically to a limit eigensystem $\alpha_\infty\colon\calH\to\C_p$ as $i\to\infty$;
\item for every $i$, $\varphi_i$ is an admissible eigenvalue of the Frobenius on $D_\pst(\rho_{f,p}\vert_{G_{\Q_p}})$, and the sequence $\varphi_i$ converges to a limit $\varphi_\infty\in\C_p^\times$ as $i\to\infty$.
\end{itemize}

\noindent An \emph{affinoid $p$-adic family of eigenforms} is a quadruple $(U,\Theta,\Phi,S)$ consisting of
	\begin{itemize}
		\item an integral affinoid space $U$ over a $p$-adic field $L$,
		\item a homomorphism $\Theta\colon\calH^{Np}\to\cO_U(U)$, 
		\item a nowhere-vanishing element $\Phi\in\cO_U(U)$, and
		\item a dense subset $S\subset U(\C_p)$,
	\end{itemize}
	such that:
		\begin{itemize}
	\item for every $x\in S$, $\Theta$ specializes to the Hecke eigensystem away from $Np$ attached to an eigenform $f_x$ of level $Np^r$, for some $r\ge 1$, and $\Phi(x)$ is an admissible eigenvalue of the Frobenius on $D_\pst(\rho_{f_x,p}\vert_{G_{\Q_p}})$; 
		\item there exists a continuous pseudorepresentation $t_U\colon G_{\Q,Np}\to\cO_U^\circ(U)$ such that $t_U(\Frob_\ell)=\Theta(T_\ell)$ for every prime $\ell\nmid Np$, 
		\item the map $\pi^\ps_\cF\colon U\to\fR_{\ovl t_U}$ deduced from the above point and the universal property of $\fR_{\ovl t}$ is finite onto its image. 
	\end{itemize}
\end{customdef}

The last two conditions are there to avoid some pathological situations. The choice of an admissible Frobenius eigenvalue amounts, in the crystalline case, to the usual choice of a refinement of an eigenform, so that in this case our definitions are very close to those of refined families given by Bella\"iche and Chenevier \cite[Section 4.2]{bellchen}. The main difference is that we allow our representations to only be potentially semistable, in order to be able to treat eigenforms of infinite slope.

Hida and Coleman families give obvious examples of families of the kind of Definition \ref{introdef}. Twists of such families with Dirichlet characters of $p$-power conductor also give such examples; such twists do not affect Frobenius eigenvalues. We show that the only remaining possibility is that almost all of the forms in the (affinoid) family are $p$-supercuspidal (Theorem \ref{classfam}). It is possible to construct $p$-supercuspidal families in which all forms have complex multiplication (CM), starting with a $p$-adic deformation of a Gr\"ossencharacter of an imaginary quadratic field (see Section \ref{seccm}); such families are either ordinary or of infinite slope, and had been written down already by Hida. Our main result is that this classification is complete. In the next theorem, we assume that the residual representation $\ovl\rho$ attached to the family satisfies the technical assumptions:
\begin{enumerate}[label=(H$\ovl\rho$\arabic*)]
\item\label{hypintro1} $\ovl\rho$ is not $\Q(\sqrt p)$-induced;
\item\label{hypintro2} if $p=5$, $F$ is a quadratic extension of $\Q$, ramified at 5, such that $\ovl\rho\vert_{G_{F_\fp}}$ is decomposable for the unique 5-adic place $\fp$ of $F$, and $\Proj(\ovl\rho(G_{F,5N}))\cong\PGL_2(\F_5)$, then the degree of $F(\zeta_5)/F$ is 4.
\end{enumerate}

\begin{customthm}{2}[cf. Theorem \ref{infslope}]\label{introthm}
Every $p$-supercuspidal (affinoid) family is a CM family.
\end{customthm}

We sketch a proof, most of which takes place on the Galois side, but with essential automorphic inputs. We prove that every family of eigenforms can be essentially built from the specializations of an affinoid family (Proposition \ref{famlim}), so that it is enough to prove Theorem \ref{introthm} for affinoid families. Let $\cF=(U,\Theta,\Phi,S)$ be an affinoid family whose specializations at $x\in S$ are all $p$-supercuspidal. 
An eigenform $f$ is $p$-supercuspidal if and only if its associated $p$-adic Galois representation $\rho_{f,p}$ is not trianguline, in the sense of $(\varphi,\Gamma)$-modules. It is, however, potentially semistable, hence a standard argument shows that it becomes trianguline over an extension $E/\Q_p$ of degree 2. We can assume such an extension to be constant as $f$ varies among the classical specializations of $\cF$, and we pick a real quadratic field $F$, with a unique $p$-adic place $\fp$, such that $F_\fp=E$. 
Every classical specialization $f_x$ of $\cF$ admits a base change to $F$, a Hilbert eigenform of finite slope whose $U_p$ eigenvalue can be written in terms of $\Phi(x)$, and in particular interpolated along $U$: this allows us to define a base change $\cF_F$ of the whole family as a subdomain of the $\GL_{2/F}$-eigenvariety constructed by Andreatta, Iovita, and Pilloni \cite{aiphilbI} (though the parallel weight eigenvariety of Kisin and Lai \cite{kislaihilb} would also suffice to our purpose). 

If $t_U\colon G_\Q\to\cO_U(U)$ is the pseudorepresentation carried by $\cF$, then by our choice of $F$, and by the work of Kedlaya, Pottharst and Xiao on interpolating triangulations along rigid families \cite{kedpotxia}, the local Galois pseudorepresentation $t_U\vert_{G_E}$ carried by the base-changed family $\cF_F$ is everywhere trianguline on $U$, while $t_U\vert_{G_{\Q_p}}$ itself is almost-nowhere trianguline. The fact that Frobenius eigenvalues can be interpolated along $U$ is essential when applying the results of \emph{loc. cit.}. We deduce that, for every $x$ in a Zariski open subdomain of $U$, the specialization $t_{U,x}\vert_{G_{\Q_p}}$ is potentially trianguline, but not trianguline. By a result of Berger and Chenevier (Theorem \ref{berche}), $t_{U,x}\vert_{G_{\Q_p}}$ is either a twist of a de Rham pseudorepresentation, or induced from a character of a quadratic extension $K/\Q_p$. The first case can only occur when the Hodge--Tate--Sen weights are both integers, which is not true over any non-empty Zariski open subspace of $U$. This means that $t_U\vert_{G_{KF}}$ is the direct sum of two characters, hence that the pseudorepresentation attached to the Hilbert family $\cF_F$ is locally decomposable. We then apply the following result, which is a simple analogue of a theorem of Balasubramanyam, Ghate, and Vatsal \cite{balghavathilb} to the case when $p$ does not split in $F$.

\begin{customthm}{3}[cf. Theorem \ref{bgvaff}]\label{introbgv}
Let $\cX$ be an irreducible component of the cuspidal $\GL_{2/F}$-eigenvariety. Assume that the local Galois pseudorepresentation $t_\cX$ attached to $\cX$ is a direct sum of two characters, and that the residual representation $\ovl\rho_\cX$ satisfies conditions \ref{hypintro1} and \ref{hypintro2} above. Then $\cX$ is a CM component.
\end{customthm}

We deduce from Theorem \ref{introbgv} that $t_U$ itself (not just its restriction to a decomposition group at $p$) becomes decomposable over a subgroup of $G_\Q$ of index at most 4, which implies that it is either reducible, or induced from a character of a subgroup of $G_\Q$ of index 2. We exclude the first case since we are only concerned with cuspidal families, so we obtain that $t_U$ has to be CM.

We conclude with a short outline of the structure of the paper. The first four sections contain little in terms of original material; they are merely a reorganization of known facts in view of their successive application. Section \ref{sectop} is only a reminder of some topological facts and conventions that are used throughout the paper. Section \ref{secdef} is also mostly a recollection of standard material on (pseudo)deformations, with the inclusion of a few specific lemmas that we need in the following, concerning the induced locus in the deformation space of a 2-dimensional pseudorepresentations, and a na\"ive notion of convergence for pseudorepresentations. The main goal of Section \ref{sectri} is to recall what the parameters of trianguline representations attached to eigenforms look like, and how triangulations and their parameters are related to refinements of Fontaine's $D_\pst$. The main difference with the literature is that we wish to speak of refinements for not necessarily crystalline representations. Section \ref{sechilb} is a collection of facts concerning Hilbert eigenvarieties. We recall what the maps between the different weight spaces involved look like, in order to bridge between Hida's nearly ordinary theory, that is used in \cite{balghavathilb}, and the language of eigenvarieties, that is better suited to our application.

Section \ref{secbgv} is devoted to the proof of Theorem \ref{introbgv} above. The strategy is the same as in \cite{balghavathilb}, but we replace a modularity result of Sasaki with a more recent one of Pilloni and Stroh \cite{pilstrmod}, that does not require $p$ to split in the real quadratic field. Section \ref{secfam} introduces the notions of families and affinoid families that we presented above in Definition \ref{introdef}. We include a description of families of finite slope and of their twists by Dirichlet characters of $p$-power conductor, as well as some results describing families as specializations of affinoid families, and affinoid families as subdomains of an eigencurve, after base change to a suitable real quadratic field. In Section \ref{seccm} we prove some facts about CM forms and families of possibly infinite slope, and we complete the classification of (affinoid) families, conditional on the main result of the paper that we present later in the final section. Section \ref{secloc} deals with our main result on families of potentially trianguline, non-trianguline local Galois representations (Theorem \ref{locthm}). Finally, Section \ref{secpsc} combines Theorems \ref{bgvaff} and \ref{locthm} to prove that every $p$-supercuspidal family is CM (Theorem \ref{infslope}).

\smallskip

\subsection*{Acknowledgments}

Most of the work on this paper was done while I was a postdoc at the University of Luxembourg in Gabor Wiese's group. I wish to thank the Number Theory group in Luxembourg for listening to various presentations of early versions of the paper. In particular, I wish to thank Alexandre Maksoud, Emiliano Torti, and Gabor Wiese for some interesting exchanges. I also thank my previous colleagues in Heidelberg University for discussions around the ideas that led to this work; in particular, Gebhard B\"ockle, Peter Gr\"af, and Judith Ludwig. 

\smallskip

\subsection*{Notations and conventions}
We fix once and for all a prime number $p>3$. \\ 
Given any field $F$, we denote by $\ovl F$ an algebraic closure of $F$ and by $G_F$ the absolute Galois group $\Gal(\ovl F/F)$. When $F$ is a $p$-adic field, we denote by $W_F$ and $I_F$ the Weil and inertia groups of $F$, respectively. When $F$ is a global field, we denote by $\A_\F$ its ring of adèles. \\
We fix once and for all a field embedding $\ovl\Q\into\Qp$, inducing a continuous injection $G_{\Q_p}\into G_\Q$. \\
Even when we omit to recall it, all representations and pseudorepresentations appearing in the text are continuous. \\
We denote by $\chi_\cyc$ the cyclotomic character of $G_{\Q_p}$, and we use the convention that its Hodge--Tate weight is $1$. \\
For every positive integer $n$, we denote by $\Q_{p^n}$ the unique unramified extension of $\Q_p$ of degree $n$, and by $\C_p$ the completion of an algebraic closure of $\Q_p$ with respect to a fixed $p$-adic valuation $v_p$, normalized so that $v_p(p)=1$. Whenever a choice of uniformizer $\pi_E$ of a $p$-adic field $E$ is required, we choose $\pi_{\Q_p}=p$.




\setcounter{tocdepth}{1}
\tableofcontents


\numberwithin{thm}{section}
\numberwithin{equation}{section}

\section{Topological lemmas}\label{sectop}

We consider rigid analytic spaces in the sense of Tate. In the following, let $L$ be a $p$-adic field. Given a rigid analytic space $X$ over $L$, we write $\cO_X$ for its structure sheaf and $\cO_X^\circ$ for the sheaf of rigid analytic functions bounded in norm by $1$. 
If we do not specify a coefficient field, by ``point of $X$'' we mean a $\C_p$-point of $X$. 
We sometimes say ``subset of $X$'' for a subset of the set of points of $X$. Given a point $x\in X(\C_p)$, we write $\ev_x$ for the ``evaluation at $x$'' map $\cO_X(X)\to\C_p$. We keep the same notation after possibly replacing the source with $\cO_X^\circ(X)$ and the target with a subring of $\C_p$.

We do not consider points (maximal spectra of fields) to be affinoid, so that our affinoids always have positive dimension.

Unless otherwise specified, we always consider our rigid analytic spaces $X$ as equipped with the \emph{analytic Zariski topology}: the closed subspaces of $X$ are the vanishing loci of the global analytic functions on $X$. 
In particular, a subset $S$ of $\C_p$-points of a rigid analytic space $X$ is \emph{(analytic Zariski-)dense} in $X$ if and only if the only function in $\cO_X(X)$ vanishing on all of $S$ is the zero function, 
if and only if there exists an admissible covering $X=\bigcup_{i\in I}U_i$ by affinoid subspaces such that $S\cap U_i$ is dense in $U_i$ for every $i$. This is the same notion of density as in \cite[Definition 6.3.2]{kedpotxia}. 

\begin{rem}
	A subset $S$ of $X$ that is dense in $X$ is not in general dense in a subspace of $X$ containing $S$: for instance, if $X$ is the one-dimensional affinoid unit disc, so that $X^\circ$ is the wide open unit disc, there exist infinite subsets $S$ of $X^\circ$ that are Zariski-dense in $X$ but not in $X^\circ$ (take any sequence of points converging to the boundary, for instance the zeroes of $\log{(1+x)}$ if $x$ is the variable on the disc).
\end{rem}


When we speak of the \emph{$p$-adic topology} on a rigid analytic space $X$, we refer to the topology obtained by considering $X$ as a $p$-adic manifold. Even if this topology makes $X$ into a totally topological space, it gives a reasonable notion of convergence.

We refer to \cite{conradirr} for the definition of irreducible component of a rigid analytic space. We say that a rigid analytic space is \emph{integral} if it is reduced and irreducible. In practice, we will almost only care about integrality for affinoid spaces: in this case, Conrad's notion of irreducibility is analogue to the usual one for schemes, so that if $U=\Spm A$ for an affinoid $L$-algebra $A$, then $X$ is irreducible if and only if the nilradical of $A$ is a prime ideal, and $X$ is integral if and only if $A$ is an integral domain. 

By \emph{affinoid neighborhood} of a point of a rigid space we will always implicitly mean \emph{open} affinoid neighborhood, but we do not put such a restriction when speaking of an affinoid subspace/subdomain (e.g. the affinoid subspace $\Spm\langle T,U\rangle/(U)$ of $\Spm\langle T,U\rangle$ is not an affinoid neighborhood of 0).

%
%

Let $L$ be a $p$-adic field. 
Let $X$ be a 1-dimensional, integral affinoid space over $L$, and 
let $S\cup\{x_0\}$ be a subset of $X(\C_p)$.
	
\begin{defin}
	We say that $S$ is \emph{discrete} if every $x\in S$ admits an affinoid neighborhood $U$ such that $U\cap S=\{x\}$. 
	We say that $S$ \emph{accumulates at $x_0$} if, for every open affinoid neighborhood $U$ of $x_0$, $S\cap U(\C_p)\ne\varnothing$.
\end{defin}


\begin{lemma}\label{dense}
	Let $X$ be a 1-dimensional, reduced, irreducible affinoid space over $L$, and let $S$ be a subset of $X(\C_p)$.
	The following are equivalent:
	\begin{enumerate}[label=(\roman*)]
		\item $S$ is infinite;
		\item $S$ is dense in $U$.
	\end{enumerate}
If $S$ is a subset of $X(L^\prime)$ for some finite extension $L^\prime/L$, then (i) and (ii) above are equivalent to:
\begin{enumerate}[resume]
\item $S$ has an accumulation point in $X(L^\prime)$.
\end{enumerate}
\end{lemma}

\begin{proof}
The implication (i)$\implies$(ii) is obvious.

By the Weierstrass preparation theorem for affinoid Tate algebras, the zero locus of an analytic function on X consists of a finite set of points. Therefore, if $S$ is infinite then the only global function that can vanish on $S$ is the zero function. This gives (i)$\implies$ (ii). 

If $S$ is a subset of $X(L^\prime)$ for some finite extension $L^\prime/L$, then (i)$\implies$(iii) follows from the compactness of $X(L^\prime)$, and (iii)$\implies$(i) is obvious. 
\end{proof}


Even though we generally try to work over affinoid spaces, we will need to consider at various points what we call ``wide open'' rigid analytic spaces (or ``nested'' in the terminology of  \cite[Definition 7.2.10]{bellchen}).

\begin{defin}\label{wodef}
A rigid analytic space $X$ over $L$ is \emph{wide open} if there exists an admissible covering $(X_i)_{i\in\N}$ of $X$ by affinoid subdomains such that, for every $i$, $X_i\subset X_{i+1}$ and the restriction map $\cO_X(X_{i+1})\to\cO_X(X_i)$ is compact.
\end{defin}

Our main interest in wide open spaces stems from the following property.

\begin{lemma}[{\cite[Lemma 7.2.11]{bellchen}}]\label{bccomp}
If $X$ is a wide open rigid analytic space, then the ring $\cO_X^\circ(X)$ is compact, hence profinite.
\end{lemma}

We use the notion of wide open spaces to single out some ``well-behaved'' inclusions of affinoid spaces.

\begin{defin}\label{nested}
	Let $U_0\subset U_1$ be an inclusion of $L$-affinoids. We say that $U_0$ is \emph{nested in $U_1$}, or a \emph{nested affinoid subdomain} if there exists a wide open rigid analytic $L$-space $V$ such that $U_0\subset V\subset U_1$.
\end{defin}

Note that a nested affinoid subdomain of an affinoid is in particular affinoid, hence not nested (as a rigid analytic space) in the sense of \cite[Definition 7.2.10]{bellchen}.

\begin{lemma}\label{nestedlemma}
	For an inclusion of $L$-affinoids $U_0\subset U_1$, the following are equivalent:
	\begin{enumerate}[label=(\roman*)]
		\item $U_0$ is nested in $U_1$;
		\item the restriction map of analytic functions $\cO_{U_1}(U_1)\to\cO_{U_0}(U_0)$ is compact. 
	\end{enumerate}
\end{lemma}

\begin{proof}
	Write $U_0=\Spm A_0$ and $U_1=\Spm A_1$ and let $f_1,\ldots,f_m\in A_1\setminus L$ be generators of $A_1$ as a Tate algebra over $L$. The restriction map $\res\colon\cO_{U_1}(U_1)\to\cO_{U_0}(U_0)$ is compact if and only if $\lVert\res(f_i)\rVert<\lVert f_i\rVert$ for every $i$, which obviously holds if and only if there exist real numbers $M_1,\ldots,M_m$ such that $\lVert\res(f_i)\rVert<M_i\le\lVert f_i\rVert$ for every $i$. If such $M_i$ exist, then $Y\coloneqq\{x\in U_1\,\vert\, \lVert f_i(x)\rVert<M_i\,\forall i\}$ is a wide open subspace of $U_1$ that contains $U_0$. On the other hand, if there exists a wide open subdomain $Y$ of $U_1$ containing $U_0$, then, for every $i$, $M_i=\sup_{x\in Y}\lVert f_i(x)\rVert$ will have the desired property. 
\end{proof}


\section{Pseudorepresentations and their deformations}\label{secdef}

Let $p$ be an odd prime, $A$ a commutative ring in which 2 is invertible, and $G$ a profinite group satisfying Mazur's ``$p$-finiteness'' assumption: for every open subgroup $H$ of $G$, the set of continuous homomorphisms $H\to\F_p$ is finite. 
Examples of $p$-finite profinite groups are absolute Galois groups of $p$-adic fields, as well as maximal quotients of absolute Galois groups of number fields unramified outside a finite set of places.

\subsection{Preliminaries on pseudorepresentations}

A \emph{1-dimensional pseudorepresentation} is a continuous map $t\colon G\to A$ such that $t(G)\subset A^\times$ and $t\vert_G\colon G\to A^\times$ is a continuous character. We sometimes abuse terminology and refer to 1-dimensional pseudorepresentations as characters. 

The following definition of 2-dimensional pseudorepresentation is essentially Chenevier's. It is not strictly speaking what he calls a 2-dimensional determinant in \cite{chendet}, but there is a natural bijection between the two kinds of objects. One can of course define pseudorepresentations (or determinants) of higher dimension, and some of the results we prove below have analogues that hold in that generality, but we stick to the 2-dimensional case since it is the only one we need in the following.


\begin{defin}\label{defps}
An $A$-valued, 2-dimensional pseudorepresentation $t$ of $G$ is a map $t\colon G\to A$ satisfying:
\begin{itemize}
\item $t(1)=2$,
\item $t(gh)=t(hg)$ for every $g,h\in G$.
\end{itemize}
Given a pseudorepresentation $t$, we attach to it a \emph{determinant} map $d_t\colon G\to A$ defined by
\[ d_t(g)=\frac{t(g)^2-t(g^2)}{2}. \]
\end{defin}

The word determinant in the above definition is not used in the sense of Chenevier: it is simply motivated by the fact that if $t$ is the trace of a representation $\rho$, then $d_t$ is the determinant of $\rho$.


We recall a result that will allow us to identify some pseudorepresentations with traces of actual representations. Recall that $2\in A^\times$ by assumption.

\begin{thm}\label{liftings}
Let $t\colon G\to A$ be a 2-dimensional continuous pseudorepresentation. If one of the following condition holds, then there exists a continuous representation $G\to\GL_2(A)$ of trace $t$.
\begin{enumerate}[label=(\roman*)]
	\item $A$ is an algebraically closed field \cite{taylorgal}.
	\item $A$ is a complete local ring (of maximal ideal $\fm_A$) and the pseudorepresentation $t\otimes_AA/\fm_A\colon G\to A/\fm_A$ is absolutely irreducible \cite{nyssen,rouquier,chendet}.
\end{enumerate}
Moreover, if in any of the two cases above two continuous representations $\rho_1,\rho_2\colon G\to\GL_2(A)$ of trace $t$ are given, then the semisimplifications of $\rho_1$ and $\rho_2$ are isomorphic (over $A$).
\end{thm}

For later use, we give a simple definition.

\begin{defin}
We say that a representation $\rho\colon G\to\GL_2(A)$ is: 
\begin{itemize}
\item \emph{abelian} if the image of $\rho$ is an abelian group;
\item \emph{abelian up to index $n$}, for a positive integer $n$, if there exists a subgroup $H$ of $G$ of index $n$ such that $\rho\vert_H$ is abelian.
\end{itemize}
We say that a 2-dimensional pseudorepresentation $t\colon G\to A$ is 
\begin{itemize}
\item \emph{trivial} if $t(g)=2$ for every $g\in G$;
\item \emph{decomposable} if it is a sum of two characters;
\item \emph{decomposable up to index $n$}, for a positive integer $n$, if there exists a subgroup $H$ of $G$ of index $n$ such that $t\vert_H$ is abelian.
\end{itemize}
\end{defin}


A 2-dimensional pseudorepresentation $t$ is trivial if and only if it is the trace of the trivial representation $G\to\GL_2(A)$. It is decomposable if and only if it is the trace of a representation with abelian image, namely the direct sum of the two characters appearing in the decomposition of $t$. In this case, it can also be the trace of a representation that does not have abelian image: simply take a non-trivial extension of the two characters, if one exists.



\subsubsection{Residual pseudorepresentations}\label{residual}

Let $t\colon G\to\Qp$ be a continuous pseudorepresentation. 
Since $G$ is profinite, hence compact, the image of $t$ is contained in the ring of integers $\cO_L$ of a finite extension $L$ of $\Q_p$. 
We can thus consider $t$ as a pseudorepresentation $G\to\cO_L$ and reduce it modulo the maximal ideal $\fm_L$ of $\cO_L$, obtaining a pseudorepresentation $\ovl t\colon G\to\cO_L/\fm_L$. We simply write $\ovl t\colon G\to\Fp$ when we do not wish to specify the coefficient field. We call $\ovl t_i$ the \emph{residual pseudorepresentation attached to $t$}.


\subsubsection{The Hodge--Tate--Sen weights of a pseudorepresentation}


Let $E$ and $L$ be two $p$-adic fields and let $\sigma\colon L\into\Qp$ be any field embedding. Let $d$ be a positive integer. By \emph{tuple of $\sigma$-Hodge--Tate--Sen weights} of a representation $G_E\to\GL_d(L)$ we mean the collection of its Hodge--Tate--Sen weights with respect to $\sigma$, with multiplicities but without an ordering. Note that if the weights are not all integers, there is no obvious way to order them. 

We prove a simple lemma.

\begin{lemma}
The tuple of $\sigma$-Hodge--Tate--Sen weights of a continuous representation $\rho\colon G_E\to\GL_d(L)$ only depends on the semisimplification of $\rho$.
\end{lemma}

\begin{proof}
Let $\rho_i$, $i=1,\ldots,k$ be the Jordan--H\"older factors of $\rho$. The characteristic polynomial of the $\sigma$-Sen operator of $\rho$ is the product of the characteristic polynomials of the $\sigma$-Sen operators of the $\rho_i$, hence the tuple of $\sigma$-Hodge--Tate--Sen weights of $\rho$ is the union of the tuples of the $\rho_i$. This immediately gives the statement. 
\end{proof}

Let $t\colon G_{E}\to L$ be a pseudorepresentation of dimension $d$. By Theorem \ref{liftings}(i) there exists a continuous representation $\rho_{t,\sigma}\colon G_{E}\to\GL_d(\Qp)$ with trace $\sigma\ccirc t$, and this property determines $\rho_{t,\sigma}$ up to semisimplification.

\begin{defin}
The \emph{$\sigma$-Hodge--Tate--Sen weights} of $t$ are the Hodge--Tate--Sen weights of any semisimple representation $\rho_{t,\sigma}$ with trace $\sigma\ccirc t$. 
\end{defin}


\subsubsection{Hodge--Tate--Sen weights in families}\label{htsfam}

Let $L$ be a $p$-adic field and $X$ a rigid analytic space over $L$, and $G$ a profinite group. By a family $M$ of $G$-representations of rank $n$ over $X$ we mean a locally free $\cO_X$-module of rank $n$ equipped with an $\cO_X$-linear action of $G$.
We define the \emph{trace} of the action of $G$ as the composition of $G\to\End_{\cO_X}(M)$ with the $\cO_X$-linear morphism $\End_{\cO_X}(M)\to\cO_X$ that is defined as the usual trace on every open subspace $U\subset X$ such that $M\vert_U$ is free. 

Let $X$ be a reduced rigid analytic space, $E$ a field and $t\colon G_{E}\to\cO_X(X)$ a continuous, 2-dimensional pseudorepresentation. 
For every $x\in X(\Qp)$ and embedding $\sigma\colon E\into\C_p$, let $s_{x,\sigma}$ and $p_{x,\sigma}$ be the sum and the product of the $\sigma$--Hodge--Tate--Sen weights of $t_x\coloneqq\ev_x\ccirc t$. 


\begin{lemma}\label{weightan}
For every embedding $\sigma\colon E\into\C_p$, there exist rigid analytic functions $s_\sigma$ and $p_\sigma$ on $X$ such that $s_{x,\sigma}=s_\sigma(x)$ and $p_{x,\sigma}=p_\sigma(x)$ for every $x\in X(\Qp)$.

If $E$ does not contain any $p$-th root of 1, then the above holds for $X=\fR_{\ovl t}$, $\ovl t$ a continuous, residual pseudorepresentation of $G_E$, and $t$ the universal pseudodeformation of $\ovl t$.
\end{lemma}

\begin{proof}
If $X$ is affinoid, then by \cite[Lemma 7.8.11(i,ii)]{bellchen}, there exists a rigid analytic space $Y$ equipped with a surjective map $g\colon Y\to X$ and with a family $M$ of $G_E$-representations with trace $g^\ast\ccirc t$. 
By a theorem of Sen \cite{senan}, the sum and product of the $\sigma$-Hodge--Tate--Sen weights of the specializations of $M$ at points of $Y$ are interpolated by rigid analytic functions $s_{Y,\sigma},p_{Y,\sigma}$ on $Y$. Since the Hodge--Tate--Sen weights of $M_y, y\in Y(\Qp)$, coincide with those of the pseudorepresentation $t(g(y))$, the functions $s_{Y,\sigma},p_{Y,\sigma}$ factor through $g$ and rigid analytic functions $s_\sigma, p_\sigma$ with the desired properties. 

If $X$ is not affinoid, we obtained the statement by covering it with reduced affinoids and gluing the functions obtained in the previous paragraph for the restriction of $t$ to each element of the cover.

The second statement follows from the first one, using the fact that if $E$ does not contain any $p$-th root of 1 the space $\fR_{\ovl t}$ is reduced for any $\ovl t$ by \cite{boejuseq}.
\end{proof}

Let $E$ be a field not containing any $p$-th root of 1, and $\ovl t$ a continuous, residual pseudorepresentation of $G_E$. Let $s_\sigma, p_\sigma$ be the analytic functions on $\fR_{\ovl t}$ produced by Lemma \ref{weightan}. 
As in \cite[Section 7]{mazurfern}, we will be interested in the subspace of $\fR_{\ovl t}$ where at least one of the Hodge--Tate--Sen weights vanishes. 

\begin{defin}
We define the \emph{Sen-null subspace} of $\fR_{\ovl t}$ as the vanishing locus $\fR_{\ovl t}^0$ of the function $p_\sigma$.
\end{defin}

The restriction $s_\sigma\vert_{\fR_{\ovl t}^0}$ describes the non-vanishing Hodge--Tate--Sen weight of the specializations of the universal pseudodeformation restricted to $\fR_{\ovl t}^0$.

\subsection{Universal pseudodeformations}

The following notation will be in place until the end of Section \ref{secdef}.
Let $\F$ be a finite field of characteristic $p$, and let $W$ be the ring of Witt vectors of $\F$. Let $\CNL_W$ be the category of complete, Noetherian local $W$-algebras with residue field $\F$. Let $G$ be a profinite, $p$-finite group and $\ovl t\colon G\to A$ a continuous pseudorepresentation.

We say that a continuous pseudorepresentation \emph{lifts} $\ovl t$ if $t\otimes_AA/\fm_A=\ovl t$, where $\fm_A$ is the maximal ideal of $A$; in that case, we also say that $\ovl t$ is a \emph{pseudodeformation} of $\ovl t$ to $A$. We consider the functor
\[ \PsDef_{\ovl t}\colon \CNL_W\to\Sets \]
attaching to $A\in\CNL_W$ the set of continuous pseudodeformations of $\ovl t$ to $A$. 
By \cite[Proposition 2.3.1]{boeckledef} the functor $\PsDef_{\ovl t}$ is represented by a universal pair $(R_{\ovl t},t^\univ)$ of an object $R_{\ovl t}$ of $\CNL_W$ and a pseudodeformation $t^\univ$ of $\ovl t$ to $R_{\ovl t}$. 






\subsubsection{Pseudodeformations along rigid analytic spaces}\label{psrig}

Let $L=\Frac W$ and let $\Rig_L$ denote the category of rigid analytic spaces over $L$. For $X\in\Rig_L$ pseudodeformation of $\ovl t$ to $X$ is a pseudorepresentation $t\colon G\to X$ such that $\ovl{t_x}\cong\ovl t\otimes_L\Qp$ for every $x\in X(\Qp)$.

We denote by $\fR_{\ovl t}$ the rigid analytic generic fiber of the formal $W$-scheme $R_{\ovl t}$, in the sense of Berthelot \cite[Section 7]{dejong}. It is a rigid analytic space over $L$. 
A standard argument shows that the pair $(\fR_{\ovl t},t^\univ)$ represents the deformation functor
\[ \PsDef^\rig_{\ovl t}\colon \Rig_L\to\Sets \]
attaching to $X\in\Rig_L$ the set of continuous pseudodeformations of $\ovl t$ to $X$. 

%

We recall a result of Chenevier implying that any pseudorepresentation along an affinoid is a pseudodeformation of a residual one. 
Let $U$ be an $L$-affinoid and $t\colon G\to\cO_U(U)$ a continuous pseudorepresentation. 
By \cite[Lemma 2.3(iv)]{chenlect5}, the residual pseudorepresentation $\ovl t_x$ attached to $t_x\coloneqq\ev_x\ccirc t$ is independent of the choice of a point $x\in U(\Qp)$. This justifies the following definition.

\begin{defin}\label{affres}
	We define the residual representation $\ovl t$ attached to $t$ as the residual pseudorepresentation $\ovl t_x$ attached to $t_x$ for an arbitrarily chosen $x\in U(\Qp)$. 
\end{defin}

\subsubsection{The determinant map}\label{detsec}

By associating to each pseudorepresentation its determinant, we obtain a natural transformation from $\PsDef_{\ovl t}$ to $\PsDef_{d_{\ovl t}}$, the latter being the deformation space of an $F$-valued character. Such a transformation induces a homomorphism of universal deformation rings $R_{d_{\ovl t}}\to R_{\ovl t}$ and an associated map of rigid analytic spaces $\det\colon\fR_{\ovl t}\to\fR_{d_{\ovl t}}$. Via the standard description of the universal deformation space of a character, $\fR_{d_{\ovl t}}$ can be identified with a connected component of the weight space $\cW_\Q$ that we introduce below in Section \ref{secwt}.

\subsubsection{Restriction of pseudorepresentations}\label{secres}

Let $H$ be an index 2 subgroup of $G$, and set $\ovl t_H=\ovl t\vert_H$. 
Consider the natural transformation $\PsDef_{\ovl t}\to\PsDef_{\ovl t_{H}}$ defined by
\begin{align*} \res_{H}^{G}\colon\PsDef_{\ovl t}(A)&\to\PsDef_{\ovl t_{H}}(A) \\
t&\mapsto t\vert_{H} \end{align*}
for every object $A$ of $\CNL_W$. By Yoneda's lemma, $\res_{H}^{G}$ is induced by a morphism $R_{\ovl t_{H}}\to R_{\ovl t}$ in $\CNL_W$ that we also denote by $\res_{H}^{G}$. 
\[ \res_{H}^{G}\colon\fR_{\ovl t}\to\fR_{\ovl t_H} \]
attached to it. 
By functoriality of the construction of the rigid generic fiber, the morphism $\res_{H}^{G}\colon R_{\ovl t_H}\to R_{\ovl t}$ induces a morphism $\res_{H}^{G,\rig}\colon\fR_{\ovl t}\to\fR_{\ovl t_{H}}$ of rigid analytic spaces.

\subsection{Induction of pseudorepresentations}

Let $A\in\CNL_W$, $t\colon G\to A$ be a 2-dimensional pseudorepresentation, and $H$ a subgroup of $G$ of index 2. Let $c$ be a representative in $G$ of the nontrivial element of $G/H$.

\begin{defin}
	We say that a 2-dimensional pseudorepresentation $t\colon G\to A$ is \emph{$H$-induced} if $t(c)=0$ and $t\vert_H$ is the sum of two characters. We say that $t$ is \emph{induced} if there exists an index 2 subgroup $H\subset G$ such that $t$ is $H$-induced. 
	
	When $G=G_E$ for a field $E$ and $F$ is a quadratic extension of $E$, we say that $t$ is \emph{$F$-induced} if it is $G_F$-induced. We say that $t$ is \emph{induced} if it is $F$-induced for some finite extension $F$ of $E$.
\end{defin}

It would be more natural to start with a subgroup of finite index of $G$, and a pseudorepresentation on it, and define its induction to $G$, but this requires more work; we refer to \cite[Definition 4.6.11]{boejuseq} for a quite general definition of induction for pseudorepresentations, only requiring certain irreducibility assumptions. 

For every character $\chi\colon H\to A^\times$, we denote by $\chi^c$ the character $H\to A^\times$ defined by $\chi^c(h)=\chi(chc^{-1})$. The proof of the following lemma is straightforward. 

\begin{lemma}\mbox{ }
	\begin{enumerate}[label=(\roman*)]
		\item Let $t\colon G\to A$ be a 2-dimensional, $H$-induced pseudorepresentation. Then there exists a character $\chi\colon H\to A^\times$ such that $t\vert_H=\chi\oplus\chi^c$.
		\item If $t$ is the trace of a representation $\rho\colon G\to\GL_2(A)$, then $t$ is $H$-induced if and only if $\rho\cong\Ind_H^G\chi$ for a character $\chi\colon H\to A^\times$.
	\end{enumerate}
\end{lemma}

%

Consider the functor
\begin{align*} \PsDef_{\ovl t}^{H-\ind}\colon \CNL_W&\to\Sets \\
	A&\mapsto\{\textrm{$H$-induced pseudodeformations $t\colon G\to A$ of }\ovl t \}. \end{align*}

\begin{prop}
	The functor $\PsDef_{\ovl t}^{H-\ind}$ is prorepresented by a quotient $R_{\ovl t}^{H-\ind}$ of $R_{\ovl t}$. 
\end{prop}

\begin{proof}
Consider the restriction map $\res_G^{H}\colon R_{\ovl t_H}\to R_{\ovl t}$ from Section \ref{secres}. The functor of reducible deformations of $\ovl t$ is represented by a quotient $R_{\ovl t_H}^\mathrm{res}$ of $R_{\ovl t_H}$ (see for example \cite[Proposition 1.5.1]{bellchen}). The deformations of $\ovl t$ that become reducible over $H$ are represented by $R_{\ovl t}\otimes_{R_{\ovl t_H}}R_{\ovl t_H}^\mathrm{res}$, which is a quotient of $R_{\ovl t}$. By further imposing the Zariski closed condition $t(c)=0$, we obtain a quotient of $R_{\ovl t}$ that represents the $H$-induced deformations.
\end{proof}

We denote by $\fR_{\ovl t}^{H-\ind}$ the rigid analytic generic fiber of $R_{\ovl t}^{H-\ind}$. The quotient map $R_{\ovl t}\to R_{\ovl t}^{H-\ind}$ induces a closed embedding $\fR_{\ovl t}^{H-\ind}\to\fR_{\ovl t}$, which we use from now on to identify $\fR_{\ovl t}^{H-\ind}$ with a closed subspace of $\fR_{\ovl t}$. 


\begin{cordef}\label{indclo}\mbox{ }
\begin{enumerate}
\item Let $A\in\CNL_W$ and $t\colon G\to A$ be a continuous pseudorepresentation. The locus
\[ (\Spec A)^{H-\ind}=\{x\in\Spec A\,\vert\,t_x\textrm{ is $H$-induced}\} \]
is the spectrum of a quotient $A^{H-\ind}$ of $A$. In particular, it is closed in $\Spec A$. 
\item Let $X$ be a rigid analytic space over a $p$-adic field $L$, and $t\colon G\to\cO_X(X)$ be a residually constant, continuous pseudorepresentation. Then there exists a closed subspace $X^{H-\ind}$ of $X$ such that, for every field extension $L^\prime$ of $L$,
\[ X^{H-\ind}(L^\prime)=\{x\in X(L^\prime)\,\vert\, t_x\textrm{ is $H$-induced}\}. \]
\end{enumerate}
Moreover, if $G$ only admits a finite number of subgroups of index 2, then the above statements hold for the loci $(\Spec A)_{\ind}$ and $X_\ind$ obtained by replacing  ``induced'' with ``$H$-induced''.
\end{cordef}

\begin{proof}
For part (i), let $\ovl t=t\otimes_AA/\fm_A$, where $\fm_A$ is the maximal ideal of $A$. 
By the universal property of $R_{\ovl t}^{H-\ind}$, $(\Spec A)^{H-\ind}=\Spec (A^{H-\ind})$ where $A^{H-\ind}=A\otimes_{R_{\ovl t}}R_{\ovl t}^{H-\ind}$ and the map $R_{\ovl t}\to A$ is given by the universal property of $R_{\ovl t}$. In particular, $A^{H-\ind}$ is a quotient of $A$ and $(\Spec A)^{H-\ind}$ is closed in $\Spec A$. 

Part (ii) is obtained in a similar way from the universal property of $\fR_{\ovl t}^{H-\ind}$: let $\ovl t$ be the residual pseudorepresentation attached to $t$, and define $X^{H-\ind}$ as the product $X\otimes_{\fR_{\ovl t}}\fR_{\ovl t}^{H-\ind}$, where $\fR_{\ovl t}^{H-\ind}\into\fR_{\ovl t}$ is the inclusion and $\fR_{\ovl t}\to X$ is provided by the universal property of $\fR_{\ovl t}$ given in Section \ref{psrig}. 

The last statement can be obtained by letting $H$ vary among the index 2 subgroups of $G$ and taking a (finite) union.
\end{proof}

\subsection{Convergence of pseudorepresentations}

For later use, we introduce a notion definition of convergence for pseudorepresentations and Hecke eigensystems. We only require very basic notions and facts; for a better discussion of this kind of questions, we refer the reader to \cite{bellconv}.


Let $L$ be an extension of $\Q_p$ contained in $\C_p$, equipped with our choice of $p$-adic valuation and with the $p$-adic topology. We set $p^{-\infty}=0$ in the following.

\begin{defin}\label{psdist}
	Let $H$ be a topological group, $d$ be a positive integer and $\Ps_{L}^d(H)$ the set of continuous pseudorepresentations $H\to L$. For $t_1,t_2\in\Ps_{L}^d(H)$, we define the \emph{distance} between $t_1$ and $t_2$ as $\dist(t_1,t_2)=\sup_{h\in H}p^{-v_p(t_1(h)-t_2(h))}$. If $\dist(t_1,t_2)\le p^{-n}$, we say that $t_1$ and $t_2$ are \emph{congruent modulo $p^n$}.
\end{defin}

It is trivial to check that Definition \ref{psdist} actually provides us with an ultrametric on $\Ps_{\cO_L}^d(H)$. If $H$ is compact, this metric takes values in $[0,1]$. Two pseudorepresentations $t_1$ and $t_2$ are congruent modulo $p^n$ in the sense of the above definition if and only if $t_1(h)-t_2(h)\in p^n\cO_L$ for every $h\in H$.

When we speak of \emph{convergence} of pseudorepresentations, it will be with respect to this distance. When $d=1$, this gives a notion of convergence for a sequence of characters.

\begin{rem}\label{unifconv}
	For every $i\in\N\cup\{\infty\}$, let $t_i\colon H\to L$ be a $d$-dimensional pseudorepresentation. By definition, the sequence $(t_i)_{i\in\N}$ converges to $t_\infty$ if, for every $h$, the sequences $(t_i(h))_{i\in\N}$ converges to $t_\infty(h)$, uniformly in $h$. If $H$ is compact, uniformity is automatic, so that the sequence $(t_i)_{i\in\N}$ converges to $t_\infty$ if and only if, for every h, the sequence $(t_i(h))_{i\in\N}$ converges to $t_\infty(h)$.
\end{rem}

We give an analogous definition for the distance between two ring homomorphisms. We use the same notation as this will not give rise to any ambiguity.

\begin{defin}\label{homdist}
	Let $R$ be a ring. We define the \emph{distance} between two homomorphisms $f_1,f_2\colon R\to L$ as $\dist(f_1,f_2)=\sup_{r\in R}p^{-v_p(f_1(r)-f_2(r))}$. If $\dist(f_1,f_2)\le p^{-n}$, we say that $f_1$ and $f_2$ are \emph{congruent modulo $p^n$}.
\end{defin}

Again, it is trivial to check that the above definition gives an ultrametric on $\Hom_\cont(R,L)$. Two homomorphisms $f_1$ and $f_2$ are congruent modulo $p^n$ in the sense of the above definition if and only if $f_1(r)-f_2(r)\in p^n\cO_L$ for every $r\in R$.

\begin{lemma}\label{distcomp}
	If $L$ is complete, then $\Ps_{L}^d(H)$ is complete and $\Hom(R,L)$ are both complete with respect to the metrics introduced in Definitions \ref{psdist} and \ref{homdist}.
\end{lemma}

\begin{proof}
	If $(t_i)_i$ is a Cauchy sequence in $\Ps_L^d(H)$, then for every $h\in H$ we define $t_\infty(h)$ as the limit in $L$ of the Cauchy sequence $t_i(h)$. The fact that $t_\infty$ is a $d$-dimensional pseudorepresentation follows from the fact that, for every $n\ge 1$, $\dist(t_\infty,t_i)\le p^{-n}$ if $i$ is large enough, so that the pseudorepresentation identities can be checked modulo $p^n$ on such an $i$. The result on $\Hom(R,L)$ is analogous.
\end{proof}

In practice, we will compute distances between two pseudorepresentations or ring homomorphisms using the following.

\begin{lemma}\label{distgen}\mbox{ }
	\begin{enumerate}
		\item Let $H,d,L$ be as in Definition \ref{psdist}. Let $\Sigma$ be a dense subset of $H$. Then, for every $t_1,t_2\in\Ps_{L}^d(H)$, $\dist(t_1,t_2)=\sup_{h\in \Sigma}p^{-v_p(t_1(h)-t_2(h))}$.
		\item Let $R$ be a ring and $\Sigma$ a set of generators of $R$. Then, for every $f_1,f_2\in\Hom(R,L)$, $\dist(f_1,f_2)=\sup_{r\in \Sigma}p^{-v_p(f_1(r)-f_2(r))}$.
	\end{enumerate}
\end{lemma}

\begin{proof}
	The inequality $\dist(t_1,t_2)\ge\sup_{h\in \Sigma}p^{-v_p(t_1(h)-t_2(h))}$ is trivial, so it is enough to prove the  reverse one. Let $t_1,t_2\in\Ps_{L}^d(H)$ and $n=\inf_{h\in \Sigma}v_p(t_1(h)-t_2(h))$. Consider the continuous function $t_1-t_2\colon H\to L$. Since $(t_1-t_2)(\Sigma)\subset p^n\cO_L$, $\Sigma$ is dense in $H$ and $p^n\cO_L$ is closed in $L$, we must have $(t_1-t_2)(H)\subset p^n\cO_L$, as desired.
	
	For the second statement, it is enough to observe that if $f_1(r_1)-f_2(r_1)\in p^n\cO_L$ and $f_1(r_2)-f_2(r_2)\in p^n\cO_L$ for some $r_1,r_2\in R$, then $f_1(r_1+r_2)-f_2(r_1+r_2)\in p^n\cO_L$ and $f_1(r_1r_2)-f_2(r_1r_2)\in p^n\cO_L$.
\end{proof}

To simplify the exposition in what follows, we give a simple definition. Here $N$ is a positive integer prime to $p$, and the Hecke algebra $\calH_\Q^{Np}$ is the one we will introduce in Section \ref{triparsec}, generated over $\Z$ by the operators $T_\ell$ where $\ell$ varies over the primes not dividing $Np$.

\begin{defin}\label{hgcorr}
	Let $A$ be a ring and $t\colon G_{\Q,Np}\to A$ a pseudorepresentation. The Hecke eigensystem away from $Np$ \emph{attached to $t$} is the unique ring homomorphism $\alpha_t\colon\calH_\Q^{Np}\to A$ such that, for every prime $\ell\nmid Np$ and lift $\Frob_\ell\in G_{\Q,Np}$ of a Frobenius at $\ell$, $\alpha_t(T_\ell)=t(\Frob_\ell)$. 
	When the above holds, we also say that $t$ is attached to $\alpha_t$.
\end{defin}

\begin{lemma}\label{hglim}
	Let $t_1,t_2$ be $L$-valued pseudorepresentations of $G_{\Q,Np}$ and $\alpha_1,\alpha_2\colon\calH_\Q^{Np}\to L$ be the Hecke eigensystems away from $Np$ attached to $t_1$ and $t_2$, respectively. Then $\dist(t_1,t_2)=\dist(\alpha_1,\alpha_2)$.
	
	In particular, if $(t_i)_i$ is a Cauchy sequence in $\Ps_{L}^d(H)$, then the associated sequence $(\alpha_i)_i$ in $\Hom(\calH_\Q^{Np},L)$ is also Cauchy. If $L$ is complete and $t_\infty=\lim_i t_i, \alpha_\infty=\lim_i\alpha_i$, then $\alpha_\infty$ is attached to $t_\infty$.
\end{lemma}

\begin{proof}
	By Lemma \ref{distgen}(i) and Chebotarev's theorem, the distance $\dist(t_1,t_2)$ can be computed on the lifts $\Frob_\ell$ of the Frobenii at the primes $\ell\nmid Np$, and by Lemma \ref{distgen}(ii) the distance $\dist(\alpha_1,\alpha_2)$ can be computed on generators of $\calH_\Q^{Np}$, such as 1 and the operators $T_\ell$ at the primes $\ell\nmid Np$. Since $\alpha_i(T_\ell)=t_i(\Frob_\ell)$ for $i=1,2$, the first statement follows. 
	The second statement is obvious.
\end{proof}

We prove a natural result on how to read convergence of pseudorepresentations from the corresponding points on a universal deformation space.
Let $H$ be a $p$-finite, profinite group and $\ovl t\colon H\to\F$ a continuous pseudorepresentation, valued in a finite extension $\F/\F_p$. As before, we denote by $(\fR_{\ovl t},t^\univ)$ the universal pseudodeformation of $\ovl t$, and, for $x\in\fR_{\ovl t}(\C_p)$, we write $t_x=\ev_x\ccirc t^\univ$. 

For every $i\in\N\cup\{\infty\}$, let $t_i\colon G\to\C_p$ be a pseudorepresentation lifting $\ovl t$, and let $x_i$ be the corresponding point of the pseudodeformation space $\fR_{\ovl t}$. 

\begin{lemma}\label{xiconv}
	If the sequence $(x_i)_{i\in\N}$ converges to $x_\infty$ in the $p$-adic topology, then the sequence of pseudorepresentations $(t_i)_{i\in\N}$ converges to $t_\infty$. The converse is true under the following assumption: there exists an affinoid $U\subset\fR_{\ovl t}$ containing all of the $x_i$, $i\in\N$. 
\end{lemma}

\begin{proof}
	Assume that the sequence $(x_i)_{i\in\N}$ converges to a point $x_\infty$ of $\fR_{\ovl t}$ in the $p$-adic topology. For every $h$, $t^\univ(h)$ is an analytic function on $\fR_{\ovl t}$, so the sequence $(t_i(h))_{i\in\N}=(\ev_{x_i}\ccirc t^\univ(h))_{i\in\N}$ converges $p$-adically to $t_\infty(h)=(\ev_x\ccirc t_\infty(h))$. Since $H$ is compact, we conclude that $(t_i)_{i\in\N}$ converges to $t_\infty$ thanks to Remark \ref{unifconv}. 
	
	For the converse, assume that $(t_i)_{i\in\N}$ converges to $t_\infty$, that all of the $x_i$, $i\in\N$, are contained in a common affinoid $U\subset\fR_{\ovl t}$, and, by contradiction, that there exists a subsequence $(x_i)_{i\in I}$, $I\subset\N$, that does not converge to $x_\infty$. 
	
	
	For every $n\in\N$, let $U_n=\{x\in U\,\vert\, \lVert t_x-t_\infty\rVert\le p^{-n}\}$: it is an affinoid subdomain of $U$. The intersection $\cap_{n\in\N}U_n$ consists of the points $x$ of $\fR_{\ovl t}$ satisfying $t_x=t_\infty$; since $\fR_{\ovl t}$ is the universal pseudodeformation space for $\ovl t$, the only such point is $x_\infty$. Let $(D_j)_{j\in\N}$ be a collection of wide open discs around $x_\infty$, satisfying $\cap_{j\in\N}D_j=\{x_\infty\}$. Fix $j\in\N$, and for every $n\in\N$ set $U_{n,j}=U_n\setminus D_j$: since $\cap_{n\in\N} U_{n,j}$ is empty and every $U_{n,j}$ is affinoid, there exists $n_j\in\N$ such that $U_{n,j}$ is empty for every $n\ge n_j$. Since $x_i\in U_{n}(\C_p)$ for $i$ large enough, we have $x_i\in D_j(\C_p)$ if $i$ is large enough. We conclude that $(x_i)_{i\in\N}$ converges to $x_\infty$.
	%
	%
\end{proof}

\begin{rem}
	Without the extra assumption on the existence of $U$ in Lemma \ref{xiconv}, we can only conclude that $(x_i)_i$ is the union of a sequence of points converging to $x_\infty$, and a sequence of points that move closer and closer to the boundary of $\fR_{\ovl t}$.
\end{rem}

\medskip

\section{Trianguline Galois representations}\label{sectri}

Let $E$ and $L$ be two $p$-adic fields. Let $n$ be a positive integer.
Let $X$ be a rigid analytic space over $L$. We refer to \cite[Sections 2.1-2.2]{kedpotxia} for the definition of the relative Robba ring $\cR_X(\pi_E)$, of $(\varphi,\Gamma_E)$-modules over it, and of the functor $V\mapsto D_\rig(V)$ attaching to a finite projective $\cO_X$-module $V$, equipped with a continuous $\cO_X$-linear action of $G_E$, a $(\varphi,\Gamma_E)$-module over $\cR_X(\pi_E)$. 

In the following, let $V$ be a finite projective $\cO_X$-module equipped with a continuous $\cO_X$-linear action of $G_E$, and $\delta_1,\ldots,\delta_n\colon E^\times\to\cO_X(X)^\times$ continuous characters. 
We recall the following:

\begin{defin}\label{tridef}\cite[Definition 6.3.1]{kedpotxia}\mbox{}
\begin{enumerate}
\item A $(\varphi,\Gamma_E)$-module $M$ of rank $n$ over $\calR_X(\pi_E)$ is \emph{trianguline with ordered parameters $\delta_1,\ldots,\delta_n$} if, after perhaps enlarging $L$, there exists an increasing filtration $(M_i)_{0\le i\le n}$ of $M$ given by $(\varphi,\Gamma_K)$-submodules over $\calR_X(\pi_E)$ and line bundles $\cL_1,\ldots,\cL_n$ on $X$ such that
\[ M_i/M_{i-1}\cong\calR_X(\pi_E)(\delta_i)\otimes_{\cO_X}\cL_i \]
for all $i$. 
Such a filtration is called a triangulation of $M$ with ordered parameters $\delta_1,\ldots,\delta_n$. 
\item If $X=\Spm(L)$, we say that $M$ is \emph{strictly trianguline with ordered parameters $\delta_1,\ldots,\delta_n$} if, for each $i$, the sub-$(\varphi,\Gamma_E)$-module $M_{i+1}$ is the unique way of enlarging $M_i$ to a sub-$(\varphi,\Gamma_E)$-module of $M$ with quotient isomorphic to $\calR_X(\pi_E)(\delta_i)$ (and in particular the trianguline filtration is the unique one with these ordered parameters). 
\end{enumerate}
We say that $\rho$ is: 
\begin{enumerate}[resume]
\item \emph{trianguline with ordered parameters $\delta_1,\ldots,\delta_n$} if $D_\rig(\rho)$ is trianguline with ordered parameters $\delta_1,\ldots,\delta_n$; 
\item in the case $X=\Spm(L)$, \emph{strictly trianguline with ordered parameters $\delta_1,\ldots,\delta_n$} if $D_\rig(\rho)$ is strictly trianguline with ordered parameters $\delta_1,\ldots,\delta_n$;
\item \emph{potentially trianguline} (if $X=\Spm(L)$, \emph{strictly trianguline, semistable}) if there exists a finite extension $F$ of $E$ such that $\rho\vert_{G_{F}}$ is trianguline. In this case, we also say that $\rho$ is \emph{$F$-trianguline} (\emph{strictly trianguline, semistable}). 
When $F$ can be taken to be abelian over $\Q$, we also say that $\rho$ is \emph{semistabelian}.
\end{enumerate}
\end{defin}


\begin{rem}
As remarked in \cite[Definition 6.3.1]{kedpotxia}, a $(\varphi,\Gamma_E)$-module $M$, trianguline with ordered parameters $\delta_1,\ldots,\delta_n$, is strictly trianguline if and only if $H_{\varphi,\Gamma_E}^0((M/M_i)(\delta_{i+1}^{-1}))$ is one-dimensional for all $i$, or alternatively $H_{\varphi,\Gamma_E}^0(M_i^\vee(\delta_{i}))$ is one-dimensional for all $i$.
\end{rem}

	

\begin{rem}\label{semistri}
A continuous representation $\rho\colon G_K\to\GL_2(L)$ is trianguline if and only if its semisimplification is trianguline: when $\rho$ is irreducible there is nothing to prove; on the other hand, reducible 2-dimensional representations are trivially trianguline. 
\end{rem}


\subsection{$p$-adic Galois types}\label{ptypessec}



%

For an arbitrary $p$-adic field $E$, we denote by $I_E$ the inertia subgroup of $G_E$, and by $W_E$ the Weil group of $E$.

Let $E$ and $L$ be two $p$-adic fields. Let $E_0$ be the largest subextension of $E/\Q_p$ unramified over $\Q_p$, and let $\varphi_0$ be the absolute arithmetic Frobenius in $G_{\F_p}$, that we also identify with an element of $\Gal(E_0/\Q_p)$. 
Following \cite[Definition 1.1]{bcdt}, we call \emph{$n$-dimensional $p$-adic type} a representation $I_{\Q_p}\to\GL_n(\Qp)$ with open kernel which can be extended to a representation of $W_{\Q_p}$. 
Given $d\ge 1$, an $n$-dimensional $\Qp$-vector space $V$ and an $E$-semistable representation $\rho\colon G_{\Q_p}\to\GL(V)$, let
\[ D_{E,\st}(\rho)=(\bB_\st\otimes_{\Q_p}V)^{G_E}, \]
where we are taking invariants with respect to the diagonal action of $G_E\subset G_{\Q_p}$. Now $D_{E,\st}(\rho)$ is a free $E_0\otimes_{\Q_p}\Qp$ module of rank $d$, equipped with a $\varphi_0$-semilinear, $\Qp$-linear (Frobenius) automorphism $\varphi$, an $E_0\otimes_{\Q_p}\Qp$-linear nilpotent endomorphism $N$ such that $N\varphi=p\varphi N$ (both coming from the corresponding objects on $\bB_\st$), and a $\Gal(E_0/\Q_p)$-semilinear, $\Qp$-linear action of $G_{\Q_p}$ induced from the diagonal action of $G_{\Q_p}$ and commuting with $\varphi$ and $N$ (if $E/\Q_p$ is Galois, this action factors through $\Gal(E/\Q_p)$). 


If $E^\prime/E$ is a finite extension, then obviously $\rho$ is $E^\prime$-semistable, and $D_{E,\st}(\rho)$ can be recovered from $D_{E^\prime,\st}(\rho)$ by taking $G_E$-invariants. This provides us with a map $D_{E,\st}(\rho)\into D_{E^\prime,\st}$. We can define the colimit $D_\pst(\rho)=\varinjlim_{E}D_{E,\st}(\rho)$ with respect to these transition maps, where $E$ varies among the fields over which $\rho$ is semistable. Then $D_{\pst}(\rho)$ is a free $\Q_p^\nr\otimes_{\Q_p}\Qp$-module of rank $d$, equipped with a $\varphi_0$-semilinear, $\Qp$-linear Frobenius automorphism that we still denote by $\varphi$, and a $\Q_p^\nr\otimes_{\Q_p}\Qp$-linear nilpotent endomorphism $N$ such that $N\varphi=p\varphi N$. 


As before, let $E$ be any Galois extension of $\Q_p$ such that $\rho$ is $E$-semistable. Following \cite[Section 2.2.1]{bremez}, we show how to construct a Weil--Deligne representation from $D_{E,\st}(\rho)$.
There is an isomorphism
\begin{gather}\begin{aligned}\label{E0Qp}
E_0\otimes_{\Q_p}\Qp&\xto{\sim}\Qp^{[E_0\colon\Q_p]} \\
a\otimes b&\mapsto (\sigma(a)b)_{\sigma\colon E_0\to\Qp}
\end{aligned}\end{gather}
for which the automorphism $\varphi_0\otimes\Id$ on the left-hand side acts by permuting the $\Qp$ factors on the right-hand side. Let $p_i$ be the composition of the isomorphism \eqref{E0Qp} with the projection to the $i$-th factor of the right-hand side. Set $D_{i,\st}(\rho)=D_{E,\st}(\rho)\otimes_{p_i}\Qp$. 
By tensoring \eqref{E0Qp} up to $D_{E,\st}(\rho)$, we obtain
\begin{equation}\label{Dist} D_{E,\st}(\rho)\cong\prod_{i=1,\ldots,[E_0\colon\Q_p]}D_{i,\st}(\rho),
\end{equation}
where $\varphi_0$ now permutes the factors of the right-hand side. Clearly $D_{i,\st}$ is a $\Qp$-vector space of rank $n$ for every $i$. 
We let any $g\in W_{\Q_p}$ act on $D_{E,\st}(\rho)$ as $g\ccirc\varphi_0^{-\alpha(g)}$, where $g$ acts via the action of $G_{\Q_p}$ and $\varphi_0^{\alpha(g)}$ is the image of $g$ in $G_{\F_p}$. Since $\varphi$ is $\varphi_0$-semilinear, the action of $W_{\Q_p}$ leaves each factor in the right-hand side of \eqref{Dist} stable. The nilpotent $E_0\otimes_{\Q_p}\Qp$-linear operator $N$ also induces a nilpotent $\Qp$-linear operator on $D_{i,\st}(\rho)$, for every $i$. This data makes each $D_{i,\st}(\rho)$ into a Weil--Deligne representation, that is independent of the choice of $E$ and $i$ thanks to \cite[Lemme 2.2.1.2]{bremez}. We call it the Weil--Deligne representation associated with $\rho$, and we refer to its restriction to $I_{\Q_p}$ as the \emph{Galois type} of $\rho$.

The following is immediate from the above construction.

\begin{lemma}\label{sttype}
	Let $E$ be a finite extension of $\Q_p$, and let $\tau$ be the Galois type of $\rho$.
	The representation $\rho$ is $E$-semistable if and only if $\tau\vert_{I_{E}}$ is trivial.
\end{lemma}

Under our assumption that $p>2$, there is a classification of 2-dimensional $p$-adic Galois types 
to be found for instance in \cite[Lemme 2.1.1.2]{bremez}: 

\begin{prop}\label{padictypes}
Let $\tau$ be a $p$-adic Galois type of dimension $2$. Then $\tau$ has one of the following forms:
\begin{enumerate}[label=(\roman*)]
\item\label{type1} $\tau=\chi_1\vert_{I_{\Q_p}}\oplus\chi_2\vert_{I_{\Q_p}}$ for two characters $\chi_1,\chi_2\colon W_{\Q_p}\to\Qp^\times$ of finite order on $I_{\Q_p}$;
\item\label{type2} $\tau=\chi\vert_{I_{\Q_p}}\oplus\chi^c\vert_{I_{\Q_p}}$, where $\chi$ is a character $W_{\Q_{p^2}}\to\Qp^\times$, of finite order on $I_{\Q_p}$, that cannot be extended to $W_{\Q_p}$, and $\chi^c$ is its conjugate under $\Gal(\Q_{p^2}/\Q_p)$;
\item\label{type3} $\tau=\Ind_{W_E}^{W_{\Q_p}}(\chi)\vert_{I_{\Q_p}}$ for a ramified quadratic extension $E$ of $\Q_p$ and a character $\chi\colon W_E\to\Qp^\times$, whose restriction to $I_E$ is of finite order and cannot be extended to a character $I_{\Q_p}\to\Qp^\times$. 
\end{enumerate}
\end{prop}



Let 
\[ D_\dR(\rho)=(\bB_\dR\otimes_{\Q_p}V)^{G_{\Q_p}}. \] 
It is an $E$-vector space of dimension $n$ equipped with a filtration coming from that on $\bB_\dR$. 
Every choice of logarithm of a uniformizer of $E$ determines an embedding $E\otimes_{E_0}\bB_\st\into\bB_\dR$, hence
\begin{equation}\label{stdR} E\otimes_{E_0}D_{E,\st}(\rho)\into D_\dR(\rho),
\end{equation}
where we are letting $E_0$ act on the $E_0\otimes_{\Q_p}\Qp$-module $D_{E,\st}$ via the embedding $E_0\into E_0\otimes_{\Q_p}\Qp$, $e\mapsto e\otimes 1$. Via an embedding \eqref{stdR} the filtration on $D_\dR(\rho)$ induces a (Hodge) filtration on $E\otimes_{E_0}D_{E,\st}(\rho)$, that is independent of the chosen embedding. The Hodge--Tate weights are the indices at which the jumps of the Hodge filtration occur.



\subsection{Potentially trianguline representations}
 
In the following, let $\rho\colon G_{\Q_p}\to\GL_2(\Qp)$ a potentially semistable representation and $\tau$ its Galois type. 

We say that a finite Galois extension of fields $F/E$ is \emph{dihedral} if $\Gal(F/E)$ is a dihedral group, or equivalently, if $\Gal(F/E)$ is not abelian but admits an abelian subgroup of index 2.

\begin{cor}\label{pstdih}
There exists a finite extension $E$ of $\Q_p$, either abelian or dihedral, such that $\rho\vert_{G_E}$ is semistable. Moreover, if $\rho$ is not semistabelian and $E$ is any extension such that $\rho\vert_{G_E}$ is semistable, then $\rho\vert_{G_E}$ is crystalline.
\end{cor}

\begin{rem}\label{inertriv}
The abelianization $W_{K}^\ab$ of $W_K$ is isomorphic, as a topological group, to the direct product $\Z\times I_{K}^\ab$, where $I_{K}^\ab$ is the image of $I_{K}$ in the abelianization of $G_{K}$. If $H$ is a closed subgroup of $I_{K}^\ab$, then let $E_H$ be the field fixed in the maximal abelian extension $K^\ab$ of $K$ by the closed subgroup $\Z\times H$ of $W_{K}^\ab$. By construction $E_H$ is a totally ramified extension of $K$. If $\chi\colon W_{K}\to\Qp^\times$ is a continuous character, then $\chi$ factors through a character $\chi^\ab\colon W_K^\ab\to\Qp^\times$, and by choosing $H=\ker\chi\vert_{I_K^\ab}$ we obtain a totally ramified abelian extension $E_\chi\coloneqq E_H$ of $K$ with the property that $\chi\vert_{I_{E_\chi}}$ is trivial.
\end{rem}

\begin{proof}
If the Galois type of $\rho$ is of the form given in (i) of Proposition \ref{padictypes}, then it becomes trivial over the composite of the two finite abelian extensions $E_{\chi_1}, E_{\chi_2}$ of $\Q_p$ provided by Remark \ref{inertriv}. In case (ii), it is the composite of the finite abelian extension $E_\chi$ of $\Q_{p^2}$ and its conjugate under a lift of the nontrivial element of $\Gal(\Q_{p^2}/\Q_p)$, and in case (iii) it is the composite of the finite abelian extension $E_{\chi}$ of $K$ and its conjugate under a lift of the nontrivial element of $\Gal(K/\Q)$.

The second statement follows immediately from the above argument and \cite[Lemme 2.2.2.2]{bremez}.
\end{proof}

\begin{lemma}[{cf. \cite[Proposition 6.1.5]{hansenuni}}]
\label{typetri}
For every finite extension $E$ of $\Q_p$, the following are equivalent: 
\begin{enumerate}
\item $\rho\vert_{G_E}$ is trianguline;
\item $\rho\vert_{G_E}$ is semistabelian.
\end{enumerate}
\end{lemma}

A sketch of the proof of Lemma \ref{typetri} can be found for instance in \cite[Proposition 6.1.5]{hansenuni}. 

By combining Corollary \ref{pstdih} and Lemma \ref{typetri} we obtain:

\begin{cor}\label{resttri}
If $\rho$ is potentially semistable, then there exists a quadratic extension $E$ of $\Q_p$ such that $\rho\vert_{E}$ is trianguline. 
\end{cor}

Finally, we recall the following classification, due to Berger and Chenevier:

\begin{thm}[{\cite[Théorème 3.1]{bergchen}}]\label{berche}
	Let $\rho\colon G_{\Q_p}\to\GL_2(\Qp)$ be a potentially trianguline, non-trianguline representation. Then $\rho$ satisfies one of the following (non-exclusive) conditions:
	\begin{enumerate}[label=(\roman*)]
		\item $\rho\cong\Ind_{G_E}^{G_{\Q_p}}\chi$ for a quadratic extension $E$ of $\Q_p$ and a character $\chi\colon G_E\to\Q_p^\times$;
		\item\label{tdr} $\rho$ is the twist of a de Rham representation with a character.
	\end{enumerate}
\end{thm}

We state a corollary.

\begin{cor}\label{triquad}
If $\rho\colon G_{\Q_p}\to\GL_2(\Qp)$ is any continuous representation, then there exists a quadratic extension $E$ of $\Q_p$ such that $\rho$ is $E$-trianguline.
\end{cor}

\begin{proof}
If $\rho$ is trianguline the statement is trivial. If $\rho$ falls in case (i) of Theorem \ref{berche}, then it becomes trianguline over the quadratic extension of $\Q_p$ from which it is induced. If $\rho$ falls in case (ii), write it as $\delta\otimes\rho_0$ for a character $\delta$ and a de Rham representation $\rho_0$. Then $\rho$ is $E$-trianguline, for some finite extension $E$ over $\Q_p$, if and only if $\rho_0$ is $E$-trianguline. Then Corollary \ref{resttri}, gives the conclusion.
\end{proof}

\subsection{Trianguline and semistable pseudorepresentations}

Let $E$ and $L$ be two $p$-adic fields, and let $\rho\colon G_E\to \GL_d(L)$ be a continuous representation. Let $\rho^\sms$ denote the semisimplification of $\rho$.

\begin{lemma}\label{triss}\mbox{ }
\begin{enumerate}[label=(\roman*)]
\item If $\rho$ is semistable, then $\rho^\sms$ is. The converse is false.
\item The representation $\rho$ is trianguline if and only if $\rho^\sms$ is. Moreover, for every triangulation of $\rho$, there exists a triangulation of $\rho^\sms$ of the same ordered parameter.
\end{enumerate}
\end{lemma}

\begin{proof}
It is a standard fact that the category of semistable representations $G_E\to \GL_d(L)$ is stable under subquotients, which gives the first statement of (i). As for the converse, it is also well-known that the non-trivial extension of the cyclotomic character of $G_{\Q_p}$ by the trivial character is not potentially semi-stable. 

For (ii), let $\rho_i$, $i=1,\ldots,k$ be the Jordan--H\"older factors of $\rho$. It is immediate from the definition that the $(\varphi,\Gamma)$-module $D(\rho)$ can be written via successive extensions of the $(\varphi,\Gamma)$-modules $D(\rho_i)$, and that $D(\rho^\sms)$ is isomorphic to the direct sum of the $D(\rho_i)$. Given for each $i$ a triangulation of $D(\rho_i)$, one obtains by successive extensions a triangulation of $\rho$. On the other hand, if $(M_j)_j$ is a triangulation of $\rho$, then, for every $i$, $(M_j\cap D(\rho_i))_j$ is a triangulation of $D(\rho_i)$ (after removing the redundant terms), and putting these triangulations back together as in the previous sentence gives back the original triangulation of $D(\rho)$ (in particular, the ordered parameters of the triangulation are preserved while going through these two operations). 
Therefore, $D(\rho)$ is triangulable if and only if $D(\rho_i)$ is triangulable for every $i$, and applying once more this statement to $\rho^\sms$ in place of $\rho$ gives the equivalence in (ii). Since parameters are preserved by cutting out and putting back together triangulations, we also obtain the last statement. 
\end{proof}

\begin{defin}
	We say that a continuous representation $\rho\colon G_E\to\GL_d(\Qp)$ is trianguline if there exists a $p$-adic field $L$ and a trianguline representation $\rho_L$ such that $\rho\cong\rho_L\otimes_L\Qp$.
\end{defin}

It is easy to show that if $\rho\colon G_E\to\GL_d(\Qp)$ is trianguline, then for every $p$-adic field $L$ and continuous representation $\rho_L$ satisfying $\rho\cong\rho_L\otimes_L\Qp$, $\rho_L$ is trianguline.

Let $t\colon G_E\to L$ be a pseudorepresentation of dimension $d$.
The first part of the following definition is a special case of \cite[Definition 7.1]{wefam}. 

\begin{defin}\label{defst}
We say that $t$ is \emph{semistable} (respectively, \emph{trianguline}) if there exists a semistable (respectively, trianguline) representation $\rho\colon G_E\to\GL_d(\Qp)$ with trace $t\otimes_L\Qp$.

As usual, we say that $t$ has \emph{potentially} a property if there exists a finite extension $F/E$ such that $t\vert_{G_F}$ has that property. 

For a finite extension $F$ of $E$, we say that $t$ is $F$-trianguline (respectively, $F$-semistable) if $t\vert_{G_F}$ is trianguline (respectively, semistable). Alternatively, we say that $\rho$ becomes trianguline (respectively, semistable) over $F$.
\end{defin}

\begin{rem}\mbox{ }
\begin{enumerate}
\item 
In \emph{loc. cit.} we can simply take $\cK$ to be the category with a single object $\Qp$ and morphisms the field automorphisms of $\Qp$.
\item Thanks to Lemma \ref{triss}, we can assume that the representations appearing in Definition \ref{defst} are semisimple. Moreover, if $t\colon G_E\to L$ is trianguline then every continuous representation $\rho\colon G_E\to\GL_d(\Qp)$ with trace $t$ is trianguline. 
\end{enumerate}
\end{rem}

We check that potential properties behave well with respect to Definition \ref{defst}.

\begin{lemma}\label{pspot}
Let $F$ be a finite extension of $E$. 
A pseudorepresentation $t\colon G_E\to L$ is $F$-semistable (respectively, $F$-trianguline) if and only if there exists an $F$-semistable (respectively, $F$-trianguline) representation $\rho\colon G_E\to\GL_d(\Qp)$ with trace $t\otimes_L\Qp$.
\end{lemma}

\begin{proof}
Let $\rho\colon G_E\to\GL_d(\Qp)$ be any continuous representation with trace $t\otimes_L\Qp$.

Assume first that $t$ is $F$-trianguline, so that there exists a trianguline representation $\rho_F\colon G_F\to\GL_d(\Qp)$ with trace $t\vert_{G_F}\otimes_L\Qp$. Since $\rho\vert_{G_F}$ and $\rho_F$ share the same trace, by the last statement of Theorem \ref{liftings} their semisimplifications coincide. In particular, $\rho\vert_{G_F}$ is trianguline by Lemma \ref{triss}(ii).
The converse implication in the trianguline case is obvious.

Assume now that $t$ is $F$-semistable. Thanks to Lemma \ref{triss}(i), we can assume that $\rho$ is semisimple. Let $\rho_F\colon G_F\to\GL_d(\Qp)$ be a semistable representation with trace $t\vert_{G_F}\otimes_L\Qp$; by Lemma \ref{triss}(i) we can take it to be semisimple. We prove that $\rho\vert_{G_F}$ is semistable. Since $\rho$ is semisimple, it is enough to prove this for its irreducible components; therefore we can and do assume that $\rho$ is irreducible. Then $\rho\vert_{G_F}$ is semisimple by Lemma \ref{Hss}, hence by the last statement of Theorem \ref{liftings} it is isomorphic to $\rho_F$, hence semistable.
\end{proof}

By combining Lemmas \ref{pstdih}, \ref{resttri} and \ref{pspot} we obtain the following.

\begin{cor}\label{pstps}
Let $t\colon G_{\Q_p}\to L$ be a 2-dimensional, potentially semistable pseudorepresentation. 
\begin{enumerate}
\item There exists either an abelian or a dihedral extension $E$ of $\Q_p$ such that $t\vert_{G_E}$ is semistable.
\item There exists a quadratic extension $E$ of $\Q_p$ such that $t\vert_{G_E}$ is trianguline.
\end{enumerate}
\end{cor}

\subsection{Refined trianguline representations}

Let $E$ and $L$ be two $p$-adic fields. Assume that $L$ contains the image of every embedding $E\into\Qp$, and let $\Sigma_E$ be the set of embeddings $E\into L$. 

Let $V$ be a $\Qp$-vector space of finite dimension $d$, $\rho\colon G_\Q\to\GL(V)$ a potentially semistable representation, and $E$ a $p$-adic field such that $\rho\vert_{G_E}$ is semistable. Let $E_0$ be the largest intermediate extension of $E/\Q_p$ unramified over $\Q_p$. By the construction in Section \ref{ptypessec}, $D_{E,\st}(V)$ is a free $E_0\otimes_{\Q_p}\Qp$-module equipped with a $\Gal(E_0/\Q_p)$-semilinear Frobenius automorphism $\varphi$. Then $\varphi^{[E_0\colon\Q_p]}$ is an $E_0\otimes_{\Q_p}\Qp$-linear automorphism of $D_{E,\st}(V)$. We define the \emph{potentially semistable Frobenius eigenvalues of $\rho$ (or $V$)}, or simply \emph{Frobenius eigenvalues of $\rho$ (or $V$)}, as the eigenvalues of $\varphi^{[E_0\colon\Q_p]}$. A simple check shows that they do not depend on the chosen $E$.

Note that the fact that $\rho\vert_{G_E}$ is the restriction of a representation of $G_{\Q_p}$ together with condition \ref{feigendelta} forces the set $\{k_{\sigma,i}\}_{1\le i\le d}$ to be independent of $\sigma\in\Sigma_E$; it coincides with the set $\{k_{i}\}_{1\le i\le d}$ of Hodge--Tate weights of $\rho$. 

We extend some definitions of \cite[Section 2.4]{bellchen} from the crystalline to the potentially semistable case.
We assume for the rest of this subsection that both the Frobenius eigenvalues and the Hodge--Tate weights of $\rho$ are pairwise distinct. We write the Hodge--Tate weights of $\rho$ as $(k_1,\ldots,k_d)$ with $k_1<k_2<\ldots<k_d$.

\begin{defin}
A \emph{refinement} of $\rho$ is an ordering $(\varphi_1,\ldots,\varphi_d)$ of the Frobenius eigenvalues of $\rho$. To such a refinement we attach an increasing filtration $(\cF_i)_i$ of $D_{E,\st}(\rho)$, where $\cF_i$ is spanned by the Frobenius eigenspaces of eigenvalues $\varphi_j, 1\le j\le i$. We use the same notation for the filtration of $D_\pst(\rho)$ induced by that on $D_{E,\st}(\rho)$.

We say that a refinement $(\varphi_1,\ldots,\varphi_d)$ of $\rho$ is 
\begin{itemize}
\item \emph{$N$-stable} if every step of the filtration $(\cF_i)_i$ is stable under the monodromy operator of $D_{E,\st}(\rho)$ (an empty condition if $\rho\vert_{G_E}$ is crystalline);
\item \emph{non-critical} if the filtration $(\cF_i)_i$ is in general position with respect to the Hodge filtration of $D_{E,\st}(\rho)$, meaning that, for every $i\in\{1,\ldots,d\}, D_{E,\st}(\rho)=\cF_i\oplus\Fil^{k_i+1}D_{E,\st}(\rho)$; 
\item \emph{numerically non-critical} if $v_p(\varphi_1)\le k_2$ and, for $2\le i\le d-1$, $v_p(\varphi_1)+\ldots+v_p(\varphi_i)\le k_1+k_2+\ldots+k_{i-1}+k_{i+1}$.
\end{itemize}
\end{defin}

A simple check shows that if $\rho$ is numerically non-critical then it is non-critical. 

We fix a uniformizer $\pi_E$ of $E$, and we normalize the local reciprocity map from $E^\times$ to the abelianization of the Weil group of $E$ by sending $\pi_E$ to a geometric Frobenius. We write $e$ for the degree of ramification of $E/\Q_p$. 
Let $\delta_1,\ldots,\delta_d\colon E^\times\to L^\times$ be potentially semi-stable characters, such that for every $i$ all of the Hodge--Tate weights of $\delta_i$ coincide with $k_i$. 

The next definition is inspired by \cite[Definition 6.4.1]{kedpotxia}, but we look at the actual $D_{E,\st}$ rather than the $D_\pst$, so that our parameters are defined over $E^\times$ and we do not specify the descent data to $\Q_p$.

\begin{defin}\label{reftridef}
We say that $\rho\vert_{G_E}$ is \emph{refined trianguline} with parameters $\delta_1,\ldots,\delta_d\colon E^\times\to L^\times$, if $\rho\vert_{G_E}$ is semistabelian and:
\begin{enumerate}[label=(\roman*)]
	\item\label{feigendelta} the eigenvalues of the Frobenius operator on $D_{E,\st}(\rho)$ are pairwise distinct and given by $\pi_E^{k_{1}}\delta_1(\pi_E),\ldots,\pi_E^{k_{n}}\delta_d(\pi_E)$; 
	\item\label{psicond} for every $i$, $1\le i\le n$, the action of $G_{\Q_p}$ on $D_{E,\st}(\rho)$ stabilizes the $p^{k_{i}}\delta_i(\pi_E)$-eigenspace and acts on it via $\psi_i$;
	\item\label{flag} the refinement $(\pi_E^{k_{1}}\delta_1(\pi_E),\ldots,\pi_E^{k_{d}}\delta_d(\pi_E))$ is $N$-stable and non-critical. 
\end{enumerate}
We say that $\rho$ is \emph{potentially refined trianguline} if $\rho\vert_{G_E}$ is refined trianguline for some finite extension $E/\Q_p$ and some choice of parameters.
\end{defin}

\begin{rem}
The data in Definition \ref{reftridef} depends on the choice of the uniformizer $\pi_E$, which is not surprising given that such a choice is also relevant when attaching a $(\varphi,\Gamma_E)$-module $\cR_L(\pi_E)(\delta)$ to a character $\delta\colon E^\times\to L^\times$.
\end{rem}

\begin{rem}\label{uniquepar}
If $\rho\vert_{G_E}$ is refined trianguline with parameters $\delta_1,\ldots,\delta_d$, then the characters $\delta_1,\ldots,\delta_d$ are uniquely determined by the $N$-stable refinement $(p^{k_{1}}\delta_1(\pi_E),\ldots,p^{k_{d}}\delta_d(\pi_E))$: indeed, the finite order characters $\psi_1,\ldots,\psi_d$ are determined by condition \ref{psicond} in Definition \ref{reftridef}, and together with the ordered Hodge--Tate weights of $\rho$ they provide us with all of the $\delta_i\vert_{\cO_E^\times}$, while the values of the $\delta_i$ at $\pi_E$ are obviously determined by the refinement. For this reason, we sometimes say that $\rho\vert_{G_E}$ is refined trianguline \emph{with respect to a refinement}.
\end{rem}

Roughly speaking, if a $p$-adic family of Galois representations containing a dense subset of potentially semistable points is given (say, the family carried by a neighborhood of a classical point of an eigencurve), it is not possible to interpolate directly the Frobenius eigenvalues of the potentially semistable points, since their $p$-adic valuation is not locally constant along the family. This is our motivation for the following definition.

\begin{defin}\label{normfrob}
If $\rho$ is potentially refined trianguline, $E$ is a finite extension of $\Q_p$ such that $\rho\vert_{G_E}$ is refined trianguline with parameters $\delta_1,\ldots,\delta_n$, and $\pi_E$ is a uniformizer of $E$, we define the \emph{normalized potentially semistable Frobenius eigenvalues} of $\rho$ as the elements $\delta_1(\pi_E),\ldots,\delta_n(\pi_E)$ of $L^\times$. When this causes no ambiguity, we refer to them simply as the \emph{normalized Frobenius eigenvalues} of $\rho$.
\end{defin}

The normalized Frobenius eigenvalues of $\rho$ are independent of the choice of $E$, since both the Hodge--Tate weights and the Frobenius eigenvalues of $\rho$ are.

The following lemma and its proof are analogous to \cite[Lemma 6.4.2]{kedpotxia}. Recall that $\rho$ is assumed to be potentially semistable throughout this section.

\begin{lemma}\label{refstr}
The following are equivalent:
\begin{enumerate}
\item $\rho\vert_{G_E}$ is refined trianguline with parameters $\delta_1,\ldots,\delta_n$;
\item $\rho\vert_{G_E}$ is trianguline with parameters $\delta_1,\ldots,\delta_n$ and the corresponding refinement is non-critical;
\item $\rho\vert_{G_E}$ is strictly trianguline with parameters $\delta_1,\ldots,\delta_n$ and the corresponding refinement is non-critical.
\end{enumerate}
Moreover, the tuples of parameters $(\delta_1,\ldots,\delta_n)$ with respect to which $\rho\vert_{G_E}$ is refined trianguline are in bijection with the $N$-stable, non-critical refinements of $D_{E,\st}(\rho)$.
\end{lemma}

\begin{proof}
By the equivalence of categories from \cite[Théorème A]{bergeqphiN}, the complete flags of the type given in condition \eqref{flag} of Definition \ref{reftridef} are in bijection with the triangulations $(M_i)_{0\leq i\leq n}$ of the $(\varphi,\Gamma_E)$-module $M\coloneqq D_\rig(\rho\vert_{E})$ with ordered parameters $\delta_1,\ldots,\delta_n$ (in the crystalline case, one can also use \cite[Proposition 2.4.1]{bellchen}). 
All such triangulations are automatically strict, since the Frobenius eigenvalues of $\rho$ are assumed to be pairwise distinct: Indeed, for every $i\in\{0,\ldots,n\}$, the $(\varphi,\Gamma_E)$-module $(M/M_i)(\delta_{i+1}^{-1})$ corresponds via the above equivalence of categories to a $(\varphi,N,G_E)$-module with a unique $\varphi$-eigenvalue equal to $1$, hence the cohomology group $H^0_{\varphi,\Gamma_E}((M/M_i)(\delta_{i+1}^{-1}))$ is one-dimensional. 

The last statement follows from the above bijection and Remark \ref{uniquepar}.
\end{proof}

\subsection{Trianguline parameters of eigenforms}\label{triparsec}

We will focus on 2-dimensional representations in all that follows.	Let $\rho\colon G_{\Q_p}\to\GL_2(\Qp)$ be a potentially semistable representation. 

\begin{defin}\label{dim2ref}
We define a bijection between Frobenius eigenvalues and refinements of $\rho$: if $\varphi_1$ is a Frobenius eigenvalue of $\rho$, we attach to it the unique refinement of $\rho$ whose first element is $\varphi_1$.
\end{defin}

\begin{rem}\label{admfrob}
The bijection of Definition \ref{dim2ref} induces a bijection between admissible (normalized) Frobenius eigenvalues of $\rho$ and $N$-stable refinements of $\rho$. \\
If $\rho$ is potentially crystalline, every (normalized) Frobenius eigenvalue of $\rho$ is admissible. If $\rho$ is not potentially crystalline, it will admit a unique admissible (normalized) Frobenius eigenvalue.
\end{rem}

We are mostly interested in 2-dimensional, refined trianguline representations arising from eigenforms. We start by introducing the Hecke algebra for $\GL_{2/\Q}$ that we will be working with for the rest of the paper. Let $N$ be a positive integer, prime to $p$. We define in the usual way a Hecke operator $T_\ell$ for every prime $\ell$ not dividing $Np$, and a Hecke operator $U_p$. 
We will work with the Hecke algebra
\[ \calH_\Q=\Z[\{T_\ell\}_{\ell\nmid Np},U_p], \]
and let it act on the spaces of modular forms of level $\Gamma_1(Np^r)$, $r\ge 1$, with either complex or $p$-adic coefficients. We also write
\[ \calH_\Q^{Np}=\Z[\{T_\ell\}_{\ell\nmid Np}],\quad \calH_{\Q,p}=\Z[U_p], \]
so that $\calH_\Q=\calH_\Q^{Np}\otimes_\Z\calH_{\Q,p}$.

Let $f$ be an eigenform of level $\Gamma_1(Np^r)$ for some $r\ge 0$, and let $\rho_{f,p}\colon G_\Q\to\GL_2(\Qp)$ be its associated Galois representation.

\begin{defin}
An \emph{admissible} Frobenius eigenvalue of $f$ is a Frobenius eigenvalue $\varphi$ of $\rho_{f,p}\vert_{G_{\Q_p}}$ such that:
\begin{itemize}
\item the refinement associated to $\varphi$ is $N$-stable;
\item if $r\ge 1$ and $f$ is of finite slope, $\varphi$ is the $U_p$-eigenvalue of $f$. 
\end{itemize}
	
A \emph{refinement} of $f$ is a the choice of an admissible Frobenius eigenvalue $\varphi$ of $f$. We also say that $(f,\varphi)$ is a \emph{refined eigenform}. 

If $r\ge 1$, we define the \emph{slope} of $f$ as the $p$-adic valuation of the $U_p$-eigenvalue of $f$. 
\end{defin}

Note that we only speak of ``slope'' for eigenforms that have a nontrivial level at $p$.

\begin{rem}
If $r\ge 1$ and $f$ is of finite slope, then by definition $f$ admits a unique refinement. If $f$ is of level $\Gamma_1(N)$ and finite slope, then it admits two refinements, i.e. the $U_p$-eigenvalues of its two $p$-stabilizations (but in the following we will rarely work with form that have no level at $p$). If $f$ is of infinite slope, then:
\begin{itemize}
	\item if $f$ is a twist of an eigenform $g$ of finite slope with a Dirichlet character $\delta$ of $p$-power order, then the refinements of $f$ coincide with the refinements of the $p$-depletion of $g$: indeed, twisting with $\delta$ has no effect on the eigenvalues of the semistable Frobenius, or the monodromy operator, but forces the $U_p$-eigenvalue to be 0, so that there is no restriction on the choice of an $N$-stable refinement of $\rho_{f,p}\vert_{G_{\Q_p}}$.
	\item If $f$ is $p$-supercuspidal, then its $U_p$-eigenvalue is 0 and $\rho_{f,p}\vert_{G_{\Q_p}}$ is potentially crystalline, hence admits two (trivially, $N$-stable) refinements.
\end{itemize}
We describe the situation more precisely in the next example.
\end{rem}

Recall that Coleman defined in \cite{colclass} a $\theta$-operator on $p$-adic modular forms, with the property that, for every $k\in\Z$, $\theta^{k-1}$ maps overconvergent eigenforms of weight $2-k\in\Z$ and some level to overconvergent eigenforms of weight $k$ and the same level. In the example that follows we describe which eigenforms give rise to potentially refined trianguline representations of $G_{\Q_p}$, and what the corresponding parameters are. 

\begin{ex}\label{eigenpar}
Let $f$ be a $\GL_{2/\Q}$-eigenform of weight $k\ge 2$ and level $\Gamma_1(Np^r)$ for some $r\ge 1$. Let $L$ be the field generated by the Hecke eigenvalues of $f$, and let $\rho_f\colon G_{\Q_p}\to\GL(V_f)$ be the local at $p$ representation attached to $f$, with $V_f$ a 2-dimensional $L$-vector space. With the notation of Definition \ref{reftridef}, $d=2$ and, if $V\vert_{G_E}$ is semistabelian for some finite extension $E/\Q_p$, then the characters $\delta_1,\delta_2\colon E^\times\to L^\times$ of a triangulation of $V_f\vert_{G_E}$ are determined on $\cO_E^\times$ by the Hodge--Tate weights $(0,k-1)$ of $V_f$ up to finite order characters $\psi_1,\psi_2$.
\begin{enumerate}[label=(\roman*)]
	\item If $f$ is of finite slope and not in the image of $\theta^{k-1}$, then $V_f$ is refined trianguline with respect to every $N$-stable refinement of $D_\st(V_f)$ (two in the crystalline case, one in the semistable, non-crystalline case). If $f$ is in the image of $\theta^{k-1}$, only the ordinary refinement is non-critical, hence gives a refined triangulation. In every case, the parameters $\delta_1,\delta_2\colon\Q_p^\times\to L^\times$ satisfy: 
\begin{itemize}
	\item $\delta_1\vert_{\Z_p^\times}$ is trivial, and $\delta_2\vert_{\Z_p^\times}$ is the weight-nebentypus character;
	\item $\delta_1(p)=a_p$, the $U_p$-eigenvalue of $f$, and $\delta_2(p)=a_p^{-1}$ (so that $a_p(f)$ and $p^{k-1}a_p^{-1}$ are the eigenvalues of the Frobenius operator of $D_\pst(V_f)$).
\end{itemize}
Indeed, $V_f$ is semistabelian by Theorem \ref{infpottri} and Lemma \ref{typetri}, and conditions (i,ii) of Definition \ref{reftridef} follow from the well-known description of $D_\pst(V_f)$. 
Condition (iii) is satisfied either by numerical non-criticality, or by \cite[Remark 2.4.6(ii)]{bellchen} for eigenforms of critical slope that are not in the image of $\theta^{k-1}$. 

\item Let $f$ be of the form $f_0\otimes\chi$ for an eigenform $f_0$ of finite slope satisfying the conditions of point (i), and a non-trivial Dirichlet character $\chi$ of $p$-power conductor. 
Let $\wtl\chi\colon\Q_p^\times\to L^\times$ be the character whose restriction to $\Z_p^\times$ is the character factoring through $\chi$, and whose value at $p$ is 1. 
If $V_{f_0}$ is refined trianguline of parameters $\delta_{0,1},\delta_{0,2}$, then $V$ is refined trianguline of parameters $\delta_1,\delta_2\colon\Q_p^\times\to L^\times$ satisfying $\delta_1=\wtl\chi\delta_{0,1}$ and $\delta_2=\wtl\chi\delta_{0,2}$. In particular:
\begin{itemize}
	\item $\delta_1\vert_{\Z_p^\times}=\wtl\chi$, and $\delta_2\vert_{\Z_p^\times}$ is the weight-nebentypus character of $f_0$ multiplied with $\wtl\chi$;
	\item $\delta_1(p)$ is the $U_p$-eigenvalue of $f_0$, and $\delta_2(p)$ is its inverse (note that the $U_p$-eigenvalue of $f$ is $a_p\chi(p)=0$, hence gives no information on the Frobenius eigenvalues of $V_f$). 
\end{itemize}
Indeed, $V_f$ is semistabelian since both $V_{f_0}$ and $\chi$ become semistable (in the case of $\chi$, even trivial) over abelian extensions of $\Q_p$. Conditions (i,ii,iii) of Definition \ref{reftridef} follow from the corresponding conditions for $V_{f_0}$.

\item\label{eigenparsc} If $f$ is $p$-supercuspidal, then by Corollary \ref{resttri} $V_f$ is not trianguline, but by Corollary \ref{pstdih} $V_f\vert_{G_E}$ is semistabelian for a quadratic extension $E/\Q_p$. By the same corollary, the monodromy operator on $D_{\pst}(V_f)$ is trivial, so that $V_f\vert_{G_E}$ is refined trianguline with respect to every non-critical refinement of it. If the normalized Frobenius eigenvalues $(\varphi_1,\varphi_2)$ of $V_f\vert_{G_E}$ satisfy $v_p(\varphi_1)v_p(\varphi_2)\ne 0$, then every refinement of $V_f\vert_{G_E}$ is numerically non-critical. If one of the two valuations is 0, say $v_p(\varphi_1)$, then the ordinary refinement $(\varphi_1,\varphi_2)$ is numerically non-critical. We do not discuss when the non-ordinary refinement is non-critical, since this can only occur on a discrete subset of the affinoid families we consider in the following. 
\end{enumerate}
\end{ex}

\subsubsection{Refined Hilbert eigenforms}

Let $F$ be a real quadratic field in which $p$ does not split, and let $\fp$ be the only $p$-adic place of $F$. 
In some of our arguments, we will need to base change some $\GL_{2/\Q}$-eigenforms to a real quadratic field. In order to lighten the exposition, we give here the analogue over $F$ of some of the terminology we introduced over $\Q$.

Let $N$ be a positive integer prime to $p$, and $f$ be a $\GL_{2/F}$-eigenform of level $N\fp^r$ for some $r\in\Z_{\ge 0}$ and parallel weight $k\ge 2$. Let $\rho_{f,p}\colon G_F\to\GL_2(\Qp)$ be the continuous representation attached to $f$ by work of Carayol, Blasius, Rogawski, Taylor among others. For every $p$-adic place $\fp$ of $F$, the restriction $\rho_{f,p}\vert_{G_{F_\fp}}$ is potentially semistable, so that we can attach to it a $(\varphi,N,G_{F_\fp})$-module $D_{\pst}(\rho_{f,p}\vert_{G_{F_\fp}})$. By the usual dictionary from $(\varphi,N)$ to $(\varphi,\Gamma)$-modules, $\rho_{f,p}\vert_{G_{F_\fp}}$ becomes trianguline over a $p$-adic field $E$, and its restriction to $G_E$ is even refined trianguline in the same sense of Definition \ref{reftridef}. 

\begin{thm}\label{finslopetri}
	The following are equivalent:
	\begin{enumerate}[label=(\roman*)]
		\item $f$ is a twist of a $\GL_{2/F}$-eigenform of finite slope with a Hecke character of finite order and conductor dividing $p^\infty$;
		\item $\rho_{f,p}\vert_{G_{F_\fp}}$ is trianguline. 
	\end{enumerate}
\end{thm}

\begin{proof}
If $f$ is of finite slope and old at $\fp$, then $\rho_{f,p}$ is crystalline at $\fp$, hence trianguline. The fact that $\rho_{f,p}$ is trianguline at $\fp$ for every $f$ of finite slope follows from the density of $\fp$-old forms on the eigenvariety $\wtl\cE_F$ and the interpolation result \cite[Theorem 6.3.13]{kedpotxia}. Since a twist of a trianguline representation with a character is still trianguline, we conclude that (i)$\implies$(ii).

Conversely, if (ii) holds, then we can always find a Dirichlet character $\delta$ of $\fp$-power conductor such that $\rho_{\delta f,p}\vert_{G_{F_\fp}}=\delta\rho_{f,p}\vert_{G_{F_\fp}}$ has a crystalline period, i.e. $\delta f$ is a principal series at $\fp$, attached to two characters at least one of which is unramified. Therefore, $\delta f$ is the $p$-depletion of an eigenform of finite slope.
\end{proof}

\begin{defin}
An \emph{admissible Frobenius eigenvalue} of $f$ is an eigenvalue $\varphi$ of the Frobenius operator on $D_\pst(\rho\vert_{G_{F_\fp}})$ such that:
\begin{itemize}
\item $N$ vanishes on the line generated in $D_{\pst}(\rho_{f,p}\vert_{G_{F_\fp}})$ by an Frobenius eigenvector of eigenvalue $\varphi$ (i.e. the corresponding refinement of $D_{\pst}(\rho_{f,p}\vert_{G_{F_\fp}})$ is $N$-stable);
\item if $f$ is of finite slope, $\varphi$ coincides with the $U_\fp$-eigenvalue of $f$.
\end{itemize}
A \emph{refinement} of $f$ is a pair $(f,\varphi)$ where $\varphi$ is an admissible Frobenius eigenvalue of $f$.
\end{defin}



\medskip

\section{Components of Hilbert eigenvarieties}\label{sechilb}

In this section we summarize some results from the literature on eigenvarieties parameterizing Hilbert modular eigenforms, presenting them in the form that we will need for our applications.

\subsection{The weight spaces}\label{secwt}

There are two natural choices for a weight space for (not necessarily parallel weight) $p$-adic families of Hilbert eigenforms for $\GL_2$ of a totally real number field: the weight space that appears in Hida's theory of big nearly ordinary Hecke algebras \cite{hidanord}, whose dimension depends on the Leopoldt defect of the number field at $p$, and the weight space over which the eigenvariety of \cite{aiphilbI} lives, whose dimension is simply the degree of the field plus 1. Since we use constructions over both weight spaces, we introduce both in this section. We found the presentation in \cite[Section 2]{loedens} helpful in this regard, especially considering that it makes it straightforward to introduce Buzzard's eigenvariety for $\GL_1$ of a number field. 

For the moment, let $F$ be an arbitrary number field, and let $F_\infty^\circ$ be the identity component of $(F\times_\Q\R)^\times$. 

\subsubsection{The weight spaces for $\Res_{F/\Q}\GL_{1/F}$}\label{GL1}
Let 
\[ \cO_{F,p}=\cO_F\otimes_\Z\Z_p=\prod_{\fp\mid p}\cO_{F_{\fp}}, \]
which is an abelian $p$-adic analytic group of rank $[F\colon\Q]$. 
Let $\cU$ be a compact open subgroup of $(\A_F^{p,\infty})^\times$, and 
%
$H(\cU)=\A_F^\times/\ovl{F^\times\cdot\cU\cdot F_\infty^\circ}$. The natural inclusion $\cO_{F,p}\into\A_F^\times$ induces a continuous map $\cO_{F,p}^\times\to H(\cU)$, fitting into an exact sequence of abelian groups
\begin{equation}\label{HU}
	0\to\ovl{\Gamma(\cU)}\to\cO_{F,p}^\times\into H(\cU)\to\Cl^{\cU F^\circ_\infty}_F\to 0
\end{equation}
whose kernel $\cU_0$ is the closure in $\cO_{F,p}^\times$ of the abelian group $\Gamma(\cU)=F^\times\cap(\cU\cdot\cO_{F,p}^\times\cdot F_\infty^\circ)$, and cokernel $\Cl^{\cU F^\circ_\infty}_F$ is the ray class group modulo $\cU F^\circ_\infty$. 
In particular, $H(\cU)$ is a compact abelian $p$-adic analytic group of dimension $1+r_2(F)+\delta_F$, where $r_2(F)$ is the number of complex places of $F$ and $\delta_F$ is the Leopoldt defect of $F$ at $p$.

Set $Q(\cU)=\cO_{F,p}^\times/\ovl{\Gamma(\cU)}$, and identify $Q(\cU)$ with a finite index subgroup of $H(\cU)$ via \eqref{HU}. 


\begin{defin}
	The \emph{weight space of tame level $\cU$} for $\Res_{F/\Q}\GL_{1/F}$ is the rigid generic fiber $\wtl\cW_{F,1}^\prime$ of $\Z_p[[Q(\cU)]]$. \\
	The \emph{weight space} for $\Res_{F/\Q}\GL_{1/F}$ is the rigid generic fiber $\wtl\cW_{F,1}$ of $\Z_p[[\cO_{F,p}^\times]]$.
\end{defin}

Via \cite[Theorem 2.1]{loedens}, we identify the $\C_p$-points of $\wtl\cW_{F,1}^\prime$ and $\wtl\cW_{F,1}$ with the continuous characters $Q(\cU)\to\C_p^\times$ and $\cO_{F,p}^\times\to\C_p^\times$, respectively.

We are now in a position to introduce Buzzard's eigenvariety for $\Res_{F/\Q}\GL_{1/F}$, whose points parameterize the $\C_p^\times$-valued Gr\"ossencharacters of $K$ of a fixed tame level. We follow \cite[Section 2]{buzzaut} and \cite[Section 2]{loedens}. 

\begin{defin}
Let $\kappa\in\wtl\cW_{F,1}(\C_p)$. 
An \emph{overconvergent $p$-adic automorphic form for $\Res_{F/\Q}\GL_{1/F}$ 
of weight $\kappa$ and tame level $\cU$} is a continuous group homomorphism $\psi\colon H(\cU)\to\C_p^\times$ such that the restriction of $\psi$ to $Q(\cU)$ coincides with $\kappa$. \\ 
The \emph{eigenvariety of tame level $\cU$ for $\Res_{F/\Q}\GL_{1/F}$} is the rigid analytic generic fiber $\cE_{F,1}$ of the formal $\Z_p$-scheme $\Z_p[[H(\cU)]]$.
\end{defin}

The $\C_p$-points of the eigenvariety $\cE_{F,1}$ are in bijection with overconvergent $p$-adic automorphic forms of tame level $\cU$ for $\Res_{F/\Q}\GL_{1/F}$, which, via the usual identifications, are in bijection with $\C_p$-valued Gr\"ossencharacters of $F$ of modulus $\fm\fc$, where $\fm$ is the prime-to-$p$ ideal of $\cO_K$ attached to $\cU$, and $\fc\mid p^\infty\cO_K$.

The maps $\cO_{F,p}^\times\to Q(\cU)\into H(\cU)$ induce maps $\Z_p[[\cO_{F,p}^\times]]\to\Z_p[[Q(\cU)]]\to\Z_p[[H(\cU)]]$. Passing to rigid analytic generic fibers, we obtain maps 
\begin{equation}\label{EW} \cE_{F,1}\xto{\omega_{\cE_{F,1}}}\wtl\cW_{F,1}^\prime\to\wtl\cW_{F,1} \end{equation}
of rigid analytic spaces over $\Q_p$.

The above discussion, together with the results of \cite[Section 2]{buzzaut}, gives the following.

\begin{prop}
The weight space $\wtl\cW_{F,1}$ is a disjoint union of a finite number of $[F\colon\Q_p]$-dimensional unit discs.
	
The weight space $\wtl\cW_{F,1}^\prime$ is a disjoint union of $(1+r_2(F))$-dimensional unit discs, and the map $\wtl\cW_{F,1}^\prime\to\wtl\cW_F$ from \eqref{EW} is a closed embedding.
	
The weight map $\omega_{\cE_{F,1}}$ is finite and flat, so that $\cE_{F,1}$ is equidimensional of dimension $1+r_2(F)$.
\end{prop}

\subsubsection{The weight spaces for $\Res_{F/\Q}\GL_{2/F}$} 
We now give the analogous picture for $\Res_{F/\Q}\GL_{2/F}$; in order to simplify the notation, we refer to this group as $\GL_{2/F}$ in the following. Here $F$ is still an arbitrary number field, though we will only apply the constructions to the case when $F$ is totally real, and, in the application to families of eigenforms, quadratic. There are various weight spaces that one can consider: the full weight space in Hida's nearly ordinary theory, and its parallel weight subspace, have dimensions that depend on Leopoldt's defect of $F$, contrary to the weight spaces appearing in \cite{aiphilbI}. We introduce these objects and explain how they are related to one another.


Let $\Sigma_F$ be the set of embeddings $F\into\Qp$, $d$ the degree of $F/\Q$, and $\cO_{F,p}=\cO_F\otimes_\Z\Z_p$. Let $K$ a $p$-adic field containing the images of all $\sigma\in\Sigma_F$.


We use the notation $\wtl\cW$, with various indices, to denote full weight spaces, and $\cW$ to denote parallel weight spaces. 
Consider the following abelian $p$-adic analytic groups and associated rigid analytic spaces: 
\begin{itemize}
\item $\Z_p^\times$, of rank 1, and $\cW_F\coloneqq(\Spf\cO_K[[\Z_p^\times]])^\rig$,
\item $\cO_{F,p}^\times\times\Z_p^\times$, of rank $1+d$, and $\wtl\cW_F\coloneqq(\Spf\cO_K[[(\cO_{F,p}^\times\times\Z_p^\times)]])^\rig$.
\item $(\cO_{F,p}^\times\times\cO_{F,p}^\times)/\cO_F^\times$, with $\cO_F^\times$ embedded diagonally in $\cO_{F,p}^\times\times\cO_{F,p}^\times$, of rank $2[F\colon\Q]-(r_1(F)+r_2(F)-1-\delta_p(F))=r_1(F)+3r_2(F)+1+\delta_p(F)$, and $\wtl\cW_{2,F}\coloneqq(\Spf\cO_K[[(\cO_{F,p}^\times\times\cO_{F,p}^\times)/\cO_F^\times]])^\rig$. If $F$ is totally real, the rank is simply $d+1+\delta_p(F)$. 
\item $\cO_{F,p}^\times/\cO_F^\times$, of rank $[F\colon\Q]-(r_1(F)+r_2(F)-1-\delta_p(F))=r_2(F)+1+\delta_p(F)$, and $\cW_{2,F}\coloneqq(\Spf\cO_K[[\cO_{F,p}^\times/\cO_F^\times]])^\rig$. If $F$ is totally real, the rank is $1+\delta_p(F)$.
\end{itemize}
There is a diagram of rigid analytic spaces over $K$
\begin{equation*}\begin{tikzcd}
\cW_{2,F}\arrow{d}{\omega_{2,F}}\arrow{r}{\iota_{2,F}}&\wtl\cW_{2,F}\arrow{d}{\wtl\omega_{2,F}} \\
\cW_F\arrow{r}{\iota_F}&\wtl\cW_{F},
\end{tikzcd}\end{equation*}
induced by the group homomorphisms
\begin{equation*}\begin{tikzcd}
\cO_{F,p}^\times/\cO_F^\times&(\cO_{F,p}^\times\times\cO_{F,p}^\times)/\cO_F^\times\arrow{l}{\det} \\
\Z_p^\times\arrow{u}{\diag}&\cO_{F,p}^\times\times\Z_p^\times\arrow{u}{\id}\arrow{l}{\Norm},
\end{tikzcd}\end{equation*}
defined as follows:
\begin{equation*}\begin{aligned} \det\colon(\cO_{F,p}^\times\times\cO_{F,p}^\times)/\cO_F^\times&\to\cO_{F,p}^\times/\cO_F^\times \\
		(a,b)&\mapsto ab,
	\end{aligned}
\end{equation*}
\begin{equation*}\begin{aligned} \Norm=\Norm_{\cO_F/\Z}\otimes\Id\colon(\cO_{F}\otimes_\Z\Z_p)^\times\times\Z_p^\times&\to\Z_p^\times \\
		(a\otimes b,c)&\mapsto\Norm_{\cO_F/\Z}(a)bc,
	\end{aligned}
\end{equation*}
\begin{equation*}\begin{aligned} \id\colon\cO_{F,p}^\times\times\Z_p^\times&\to(\cO_{F,p}^\times\times\cO_{F,p}^\times)/\cO_F^\times \\
		(a,b)&\mapsto (a,b),
	\end{aligned}
\end{equation*}
and $\diag\colon\Z_p^\times\to\cO_{F,p}^\times/\cO_F^\times$ is induced from the embedding $\Z\to\cO_{F}$ after tensoring with $\Z_p$.


We use the letter $\Lambda$ to denote power-bounded rigid analytic functions on each of the above weight spaces, i.e. we set
\begin{align*} \Lambda_{2,F}=\cO^\circ_{\cW_{2,F}}(\cW_{2,F})=\cO_K[[\cO_{F,p}^\times/\cO_F^\times]], \quad \wtl\Lambda_{2,F}=\cO^\circ_{\wtl\cW_{2,F}}(\wtl\cW_{2,F})=\cO_K[[(\cO_{F,p}^\times\times\cO_{F,p}^\times)/\cO_F^\times]], \\
\Lambda_{F}=\cO^\circ_{\cW_{F}}(\cW_{F})=\cO_K[[\Z_p^\times]], \quad \wtl\Lambda_{F}=\cO^\circ_{\wtl\cW_{F}}(\wtl\cW_{F})=\cO_K[[\cO_{F,p}^\times\times\Z_p^\times]].
\end{align*}

We introduce one more weight space, the rigid generic fiber $\wtl\cW_F^\ast$ of $\cO_K[[\cO_{F,p}^\times]]$ (that is denoted by $\cW^{G^\ast}$ in \cite{aiphilbI}). We set $\wtl\Lambda_F^\ast=\cO_{\wtl\cW^\ast}^\circ(\wtl\cW^\ast)=\cO_K[[\cO_{F,p}^\times]]$. Since we do not go into the details of the construction of the Hilbert eigenvariety, we only need $\wtl\cW^\ast$ because it carries the universal character $\cO_{F,p}^\times\to\wtl\Lambda^{\ast,\times}_F$.

\subsubsection{Classical weights}\label{secwtcl}

If $G$ is an abelian $p$-adic analytic group, then it follows from \cite[Proposition 1.1 and Theorem 2.1]{loedens} that there are natural bijections between
\begin{enumerate}[label=(\arabic*)]
\item $K$-points of the generic fiber of $\Spf\cO_K[[G]]$;
\item height 1 prime ideals $P$ of $\cO_K[[G]]$ such that $\cO_K[[G]]/P$ is of characteristic 0;
\item $G_K$-orbits of continuous characters $G\to\C_p^\times$.
\end{enumerate}
We recall that the bijection from (3) to (2) associates with a character its kernel.


%

Let $\bk=(k_\sigma)_\sigma\in\Z^{\Sigma_F}$ and $w\in\Z$ such that $k_\sigma$ and $w$ have the same parity for every $\sigma$. Let $\chi_1\colon\cO_{F,p}^\times\to\C_p^\times$ and $\chi_2\colon\Z_p^\times\to\C_p^\times$ be finite order characters such that $\chi_1(-1)\chi_2(-1)=1$. We denote by $P_{\bk,w,\chi_1,\chi_2}$ the height 1 prime ideal of $\cO_K[[\cO_{F,p}^\times\times\Z_p^\times]]$ given by the kernel of the character
\begin{equation}\begin{aligned} \cO_{F,p}^\times\times\Z_p^\times&\to\C_p^\times \\
		(x,y)&\mapsto x^\bk\chi_1(x)y^w\chi_2(y),
	\end{aligned}
\end{equation}
and by $\kappa_{\bk,w,\chi_1,\chi_2}$ the corresponding point of $\wtl\cW_F$. 
We say that a character $\cO_{F,p}^\times\times\Z_p^\times\to\Q_p^\times$, equivalently a $K$-point of $\wtl\cW_F$ or a height 1 prime ideal of $\cO_K[[\cO_{F,p}^\times\times\Z_p^\times]]$, is \emph{classical}, if it is of the above form for some $\bk,w,\chi_1$ and $\chi_2$.


We call \emph{classical} the following characters of the other $p$-adic analytic groups under consideration (or, equivalently, height 1 primes of the corresponding formal $\cO_K$-algebras, or points of the corresponding rigid $K$-spaces):
\begin{itemize}
	\item characters $\Z_p^\times\to\C_p^\times$ of the form $x\mapsto x^k\chi(x)$ for an integer $k$ and a finite order character $\Z_p^\times\to\C_p^\times$;
	\item characters $\cO_{F,p}^\times/\cO_F^\times\to\C_p^\times$ induced by
	\begin{equation}\begin{aligned} \cO_{F,p}^\times&\to\C_p^\times \\
		x&\mapsto x^\bk\chi(x),
	\end{aligned}
\end{equation}
for some $\bk\in\Z(1,\ldots,1)$ and a finite order character $\chi\colon\cO_{F,p}^\times\to\C_p^\times$ trivial on $\cO_F^\times$;
	\item characters $(\cO_{F,p}^\times\times\cO_{F,p}^\times)/\cO_F^\times\to\C_p^\times$ induced by
	\begin{equation}\begin{aligned} \cO_{F,p}^\times\times\cO_{F,p}^\times&\to\C_p^\times \\
			(x,y)&\mapsto x^{\bk_1}\chi_1(x)y^{\bk_2}\chi_2(y),
		\end{aligned}
	\end{equation}
for some $\bk_1,\bk_2\in\Z^{\Sigma_F}$ such that $\bk_1+\bk_2\in\Z(1,\ldots,1)$ $(i.e., x^{\bk_1+\bk_2}$ is trivial for $x\in\cO_F^\times)$, and finite order characters $\chi_1,\chi_2\colon\cO_{F,p}^\times\to\C_p^\times$ such that $\chi_1\chi_2$ is trivial on $\cO_F^\times$. 
\end{itemize}
We use a notation analogous to that introduced above for the prime ideals or the points of rigid spaces attached to classical characters (e.g. $P_{k,\chi}, \kappa_{k,\chi}$).

\begin{rem}\label{classdens}
A standard check shows that classical primes (respectively, classical points) are dense in each of the formal schemes (respectively, rigid spaces) considered above. One can even show that classical points are dense when either the weight(s) or character(s) are fixed: in particular, if $k\in\Z_{\ge 1}$ is fixed, the weights of the form $\kappa_{(k)_\sigma,0,\chi_1,\chi_2}$ (respectively, $\kappa_{k,\chi}$) are dense in $\wtl\cW_F$ (respectively, $\cW_F$) when $\chi_1,\chi_2$ (respectively, $\chi$) vary.
\end{rem}

\begin{rem}\label{condbound}
Let $U$ be an affinoid subdomain of $\wtl\cW_F$. The order of the characters $\chi_1$ and $\chi_2$ is bounded by a constant as $(\bk,\chi_1,\chi_2)$ varies among the classical weights of $U$: indeed, the intersection of $U$ with each connected component of $\wtl\cW_F$ is contained in an affinoid disc of center 0 and radius $p^r<1$, and a standard calculation shows that the order $n$ of $\chi_1$ and $\chi_2$ is bounded by the condition $v_p(\zeta_{n}-1)\ge r$, $\zeta_{n}$ an $n$-th root of 1. If $\wtl\chi_1$ and $\wtl\chi_2$ are the Dirichlet characters of $F$ attached to $\chi_1$ and $\chi_2$, then the prime-to-$p$ parts of the conductors of $\wtl\chi_1$ and $\wtl\chi_2$ are divisors of $N$, so that the bound on the orders of $\chi_1$ and $\chi_2$ provides us with a bound on the conductors of $\wtl\chi_1$ and $\wtl\chi_2$. The analogous statements hold for the subspace $\cW_F$ of parallel weights.

Via local class field theory, we can identify each connected component of $\wtl\cW_F$ with the rigid universal deformation space of a continuous character $G_{F_\fp}\to\Fp^\times$. Therefore, if $U$ is an affinoid and $\delta\colon G_{F_\fp}\to\cO_U(U)^\times$ is a continuous character, the discussion of the previous paragraph together with the universal property of the deformation space gives us the following: if $S$ is the set of points of $U(\Qp)$ such that the evaluation $\ev_x\ccirc\delta_x$ is a locally algebraic character, i.e., attached to a triple $(\bk_x,\chi_{1,x},\chi_{2,x})$ as above, then the orders of $\chi_{1,x}$ and $\chi_{2,x}$ (as well as the orders and conductors of the associated Dirichlet characters) are bounded by a constant as $x$ varies over $S$. 
\end{rem}

%

\subsection{The Hilbert eigenvariety} 

We recall the properties of the Hilbert eigenvariety parameterizing cuspidal overconvergent $\GL_{2/F}$-eigenforms. We refer to the construction of Andreatta, Iovita and Pilloni \cite{aiphilbI}, though some earlier partial constructions exist, as summarized in the introduction of \emph{loc. cit.}. For most of the paper we will only need the subspace of the eigenvariety parameterizing parallel weight cuspidal overconvergent eigenforms. The first construction of a parallel weight eigenvariety is due to Kisin and Lai \cite{kislaihilb}; however, they only work over the ``central disc'' of weight space consisting of analytic characters, and this is not enough for our applications.

Let $k_0$ be a natural number. We say that a point of $\cE_F$ (respectively, $\wtl\cE_F$) has \emph{weight $k_0$} if it has weight $\kappa_{k_0,\chi}$ (respectively, $\kappa_{(k_0)_\sigma,w,\chi_1,\chi_2}$) for a character $\chi$ (respectively, characters $\chi_1,\chi_2$). We say that a point of $\wtl\cE_F$ has weight \emph{greater than $k_0$} if it has weight $\kappa_{(k_\sigma)_\sigma,w,\chi_1,\chi_2}$, with $k_\sigma\ge k$ for every $\sigma$.

We say that a morphism of rigid analytic spaces $f\colon X\to Y$ is \emph{partially proper} if it satisfies the valuative criterion for properness: if $D$ is the unit affinoid disc of center 0 and $D^\ast=D\setminus\{0\}$ is the punctured unit affinoid disc, then for every commutative square
\begin{center}
	\begin{tikzcd}
		D^\ast\arrow{r}{g^\ast}\arrow{d}{\iota}& X\arrow{d}{f}\\
		D\arrow{r}{g}& Y
	\end{tikzcd}
\end{center}
where $g^\ast,h$ are injections and $\iota$ is the natural inclusion, the map $g^\ast$ can be extended to a morphism $g\colon D\to X$ compatible with the other arrows.

We summarize the properties of Hilbert eigenvarieties that we will need. In the following, $N$ is an integer $\ge 4$ and prime to $p$. 
Let $N$ be a positive integer, prime to $p$. We define a compact open subgroup $K_N$ of $\GL_2(\A_F^{p\infty})$ as
\[ K_N=\{g\in\GL_2(\A_F^{p\infty})\,\vert\,g\equiv\begin{pmatrix}1 & \ast \\ 0 & 1\end{pmatrix}\hspace{-0.2cm} \pmod{N} \}. \]
Let $\calH_F=\calH_F^{Np}\otimes_\Z\calH_{F,p}$ be the (integral version of the) Hecke algebra introduced in \cite[Section 4.3]{aiphilbI}, where $\calH_F^{Np}$ is a spherical Hecke algebra of level prime to $Np$, generated by operators $S_\fl$ and $T_\fl$ for every prime $\fl$ of $\cO_F$, and $\calH_{F,p}$ is the ring generated by Hecke operators at the $p$-adic primes of $\cO_F$. In particular, $\calH_{F,p}$ contains a distinguished operator $U_p$, canonically defined, that coincides up to a $p$-adic unit with the product of the Hecke operators attached to the $p$-adic places of $\cO_F$. 

By an \emph{(overconvergent) $\GL_{2/F}$-modular form of tame level $N$} we mean a Hilbert modular form for $\GL_{2/F}$ of level $K_NK_p$ for a compact open subgroup $K_p$ of $\GL_2(\cO_{F,p})$, and by \emph{eigenforms} we refer to eigenforms with respect to the action of $\calH_F$ as above. We say that an eigenform of level $K_NK_p$, with non-trivial $K_p$, is \emph{of finite slope} if its $U_p$-eigenvalue is nonzero.

\begin{thm}\label{hilbE} 
There exists a rigid analytic space $\wtl\cE_F$ over $K$, equipped with
\begin{itemize}
\item a ring homomorphism $\Theta^\full_{\wtl\cE_F}=\Theta_{\wtl\cE_F}\otimes\Theta_{\wtl\cE_F,p}\colon\calH_F^{Np}\otimes_\Z\calH_{F,p}\to\cO_{\wtl\cE_F}(\wtl\cE_F)$,
\item a locally-on-the-domain finite, partially proper rigid analytic morphism $\omega_{\wtl\cE_F}\colon\wtl\cE_F\to\wtl\cW_F$,
\end{itemize}
with the following properties:
\begin{enumerate}
\item $\wtl\cE_F$ is equidimensional of dimension $[F\colon\Q]+1$;
\item for every classical weight $\kappa=\kappa_{\bk,w,\chi_1,\chi_2}\in\wtl\cW_F(\C_p)$, the specializations of $\Theta_{\wtl\cE_F}$ at the points of the fiber $\omega_{\wtl\cE_F}^{-1}(\kappa)$ are the eigensystems of the cuspidal overconvergent Hilbert eigenforms of tame level $N$, weight-character $\kappa$, and finite slope.
\end{enumerate}
\end{thm}

\begin{proof}
	For $F=\Q$, the result follows from the construction of the eigencurve by Coleman--Mazur, Buzzard, and Chenevier \cite{colmaz,buzzard,chenfam}, and the partial properness statement was proved by Diao and Liu \cite{diaoliu}.
	
	When $F\ne\Q$, the theorem is a rephrasing of \cite[Theorem 5.1]{aiphilbI}. The statement about partial properness is \cite[Theorem 8.8(2)]{aiphalo}, after passing to the rigid points of the adic eigenvariety that is constructed there. 
\end{proof}

We call a $\Qp$-point $x$ of $\cE_F$ \emph{classical} if $\ev_x\ccirc\Theta^\full_{\wtl\cE_F}\colon\calH_F\to\Qp$ is the Hecke eigensystem of a finite slope Hilbert eigenform for $\GL_{2/F}$. With the following well-known statement we recall how to attach a pseudorepresentation to the eigenvariety $\wtl\cE_F$.
We refer to Definition \ref{defst} when speaking of \emph{trianguline} pseudorepresentations. Recall that $d_t$ denotes the determinant of a pseudorepresentation $t$, as in Definition \ref{defps}.


\begin{prop}\label{hilbEps}\mbox{ }
The set of classical points is (analytic Zariski-)dense in $\wtl\cE_F$. 
Moreover, there exists a 2-dimensional trianguline pseudorepresentation $t_{\wtl\cE_F}\colon G_{F,Np}\to\cO_{\wtl\cE_F}(\wtl\cE_F)$ such that
\begin{enumerate}
\item for every prime $\fl$ of $\cO_F$ not dividing $Np$, $t_{\wtl\cE_F}$ is unramified at $\fl$ and satisfies
\[ t_{\wtl\cE_F}(\Frob_\fl)=\Theta_{\wtl\cE_F}(T_\fl),\quad d_{t_{\wtl\cE_F}}(\Frob_\fl)=\Theta_{\wtl\cE_F}(S_\fl)\Norm_{F/\Q}(\fl) \]
for every lift $\Frob_\fl$ of a Frobenius at $\fl$;
\item for every classical point $x$ of $\wtl\cE_F$, with corresponding $\GL_{2/F}$-eigenform $f_x$, the specialization $\ev_x\ccirc t_{\wtl\cE_F}$ is the pseudorepresentation attached to $f_x$.
\end{enumerate}
\end{prop}

We give a sketch of the proof.

\begin{proof}
The statement about density follows in a standard way from property (ii) of Theorem \ref{hilbE}, combined with the classicality criterion due to Coleman for $F=\Q$ and given in \cite[Theorem 5.1]{aiphilbI} for general $F$. 

The density of classical points allows one to construct a 2-dimensional pseudorepresentation $t_{\cE_F}\colon G_{F,Np}\to\cO_{\cE_F}(\cE_F)$, by interpolating the pseudorepresentations attached to their associated $\GL_{2/F}$-eigenforms via Chenevier's argument (see Proposition \ref{chenarg}). This construction gives directly Properties (i) and (ii).

The Galois representation attached to a $\GL_{2/F}$-eigenform of finite slope is semistable, hence trianguline, at the $p$-adic places, since it is constructed from the cohomology of a Hilbert modular variety. Moreover, the parameters of a triangulation can be explicitly interpolated in terms of the weight and Hecke eigensystems of the eigenforms at $p$, so the fact that $t_{\cE_F}$ is trianguline follows from \cite[Theorem 6.3.13]{kedpotxia}.
\end{proof}

\begin{rem}\label{hilbref}
Let $\fp$ be a $p$-adic place of $F$. Consider the set $S_{\pst,\fp}$ of points $x\in\wtl\cE_F(\Qp)$ with the property that $\ev_x\ccirc t_{\wtl\cE_F}\vert_{G_{F_\fp}}\colon G_{\Q_p}\to\Qp$ is a potentially semistable pseudorepresentation, that we lift to a potentially semistable representation $\rho_x\colon G_{F_\fp}\to\GL_2(\Qp)$. The image $\Theta^\full_{\wtl\cE_F}(U_\fp)$ of the Hecke operator at $\fp$ is an admissible Frobenius eigenvalue of $\rho_x$, hence defines a refinement of $\rho_x$. When restricted to the set $S_\cl$ of classical specializations of $\wtl\cE_F$, which belong to $S_{\pst,\fp}$ for every $\fp$, the above discussion provides us with a bijection between $S_\cl$ and the set of representations of $G_F$ attached to Hilbert eigenforms of tame level $N$ and equipped with a refinement at each $p$-adic place of $F$.
\end{rem}

We will mostly work with the \emph{parallel weight} eigenvariety, of which we state the properties in the following.

\begin{thm}\label{hilbEpar} 
	There exists a rigid analytic subspace $\cE_F$ of $\wtl\cE_F$, equipped with
	\begin{itemize}
		\item a ring homomorphism $\Theta^\full_{\cE_F}=\Theta_{\cE_F}\otimes\Theta_{\cE_F,p}\colon\calH_F^{Np}\otimes_\Z\calH_{F,p}\to\cO_{\cE_F}(\cE_F)$,
		\item a locally-on-the-domain finite, partially proper rigid analytic morphism $\omega_{\cE_F}\colon\cE_F\to\cW_F$,
		\item a 2-dimensional trianguline pseudorepresentation $t_{\cE_F}\colon G_{F,Np}\to\cO_{\cE_F}(\cE_F)$,
	\end{itemize}
obtained by restricting the objects with the same name from Theorem \ref{hilbE}, and satisfying
	\begin{itemize}
		\item $\wtl\cE_F$ is equidimensional of dimension 1;
		\item for every parallel classical weight $\kappa=\kappa_{\bk,\chi}\in\cW_F(\C_p)$, the specializations of $\Theta^\full_{\cE_F}$ at the points of the fiber $\omega_{\cE_F}^{-1}(\kappa)$ are the eigensystems of the cuspidal overconvergent Hilbert eigenforms of tame level $N$, weight-character $\kappa$, and finite slope;
		\item for every prime $\fl$ of $\cO_F$ not dividing $Np$, $t_{\cE_F}$ is unramified at $\fl$ and
		\[ t_{\cE_F}(\Frob_\fl)=\Theta^\full_{\cE_F}(T_\fl),\quad d_{t_{\cE_F}}(\Frob_\fl)=\Theta^\full_{\cE_F}(S_\fl)\Norm_{F/\Q}(\fl) \]
		for every lift $\Frob_\fl$ of a Frobenius at $\fl$;
		\item for every classical point $x$ of $\cE_F$, with corresponding $\GL_{2/F}$-eigenform $f_x$, the specialization $\ev_x\ccirc t_{\cE_F}$ is the pseudorepresentation attached to $f_x$.
	\end{itemize}
\end{thm}

\begin{proof}
We simply define $\cE_F$ as the fiber product
\begin{center}
\begin{tikzcd}
\cE_F\arrow{r}\arrow{d}{\omega} & \wtl\cE_F\arrow{d}{\omega} \\
\cW_F\arrow{r}{\iota} & \wtl\cW_F
\end{tikzcd}
\end{center}
and its properties can be deduced from the properties of the full eigenvariety $\wtl\cE_F$ stated in Theorem \ref{hilbE} and Proposition \ref{hilbEps}.
\end{proof}

Let $\cX$ be an irreducible component of $\wtl\cE_F$, and $t_\cX\colon G_{F,Np}\to\cO_{\wtl\cE_F}^\circ(\cX)$ the restriction of $t_{\wtl\cE_F}$ to $\cX$. By \cite[Corollary 7.2.12]{bellchen}, $\wtl\cE_F$ is wide open (``nested'' in the terminology of \emph{loc.cit.}; the result can be proved in the same way for $\wtl\cE_F$ as for the eigenvarieties considered there). In particular, $\cX$ is wide open and irreducible, so that $\cO_{\wtl\cE_F}^\circ(\cX)$ is a local pro-$p$ ring by Lemma \ref{bccomp}. We reduce $t_\cX$ modulo the maximal ideal of $\cO_{\wtl\cE_F}^\circ(\cX)$ to obtain a pseudorepresentation $\ovl t_\cX$ with values in a finite extension of $\F_p$, that we refer to as the \emph{residual pseudorepresentation attached to $\cX$}. We denote by $\ovl\rho_\cX$ a semisimple representation $G_{F,Np}\to\GL_2(\Fp)$ with trace $\ovl t_\cX$. If $\F$ is a finite field of characteristic $p$ and $\ovl t\colon G_{F,Np}\to\F$ a continuous pseudorepresentation, we denote by $\wtl\cE_{F,\ovl t}$ the union of the irreducible components of $\cX$ whose residual pseudorepresentation is $\ovl t$. 

If $U$ is any irreducible rigid analytic subspace of $\wtl\cE_F$, we denote by $t_U$ the restriction of $t_{\wtl\cE_F}$ to a pseudorepresentation $G_{F,Np}\to\cO_{\wtl\cE_F}(U)$. We let $\ovl t_U$ and $\ovl\rho_U$ the pseudorepresentations produced by embedding $U$ in an irreducible component of $\wtl\cE_F$ and applying the above construction; we call it the \emph{residual pseudorepresentation attached to $U$}.

We can rewrite the two previous paragraphs with $\cE_F$ in place of $\wtl\cE_F$, and we use the analogous notation and terminology.

\subsubsection{The eigencurve and weight space for $\GL_{2/\Q}$}

In the special case $F=\Q$, we simply write $\cW_\Q$ for the usual weight space for $\GL_{2/\Q}$, whose points parameterize continuous characters $\Z_p^\times\to\C_p^\times$. Such a character is classical if it is of the form $x\mapsto x^k\chi(x)$ for an integer $k\ge 1$ and a finite order character $\chi\colon\Z_p^\times\to\C_p^\times$; as usual, we abuse notation and denote by $(k,\chi)$ the corresponding point of weight space. 

We denote by $\cE_\Q$ the cuspidal Coleman--Mazur eigencurve, equipped with extra data for which we use the same notation as in the totally real, $F\ne\Q$, case. In particular, for a classical weight $(k,\chi)$, the specializations of $\Theta^\full_{\cE_\Q}$ at the points of the fiber $\omega_{\wtl\cE_\Q}^{-1}(k,\chi)$ are the eigensystems of the cuspidal overconvergent eigenforms of tame level $N$, weight $k$, and nebentypus $\chi$ at $p$.

\begin{rem}\label{cloclass}
	The finite order part of the weight is locally constant on $\cW_\Q$, in the following sense: if $(k,\chi)$ is a classical weight, then over a sufficiently small neighborhood of $k$, all classical weights are of the form $(k^\prime,\chi)$ for some $k^\prime\in\Z$. As a consequence, the $p$-adic closure of the set of classical weights consists of weights of the form $x\mapsto x^k\chi(x)$ for $k\in\Z_p$ and a finite order character $\chi\colon\Z_p^\times\to\C_p^\times$.
\end{rem}


\subsection{Eigenvarieties as closures of modular points}

Let $k_K$ be the residue field of $K$ and $\ovl t\colon G_{F,Np}\to k_K$ a continuous pseudorepresentation (we can always enlarge $K$ to allow for an arbitrary modulo $p$ pseudorepresentation). Up to enlarging $K$, we can assume that $\F$ is contained in the residue field of $K$. 
By the universal property of the rigid analytic pseudodeformation space $\fR_{\ovl t}$ given in Section \ref{psrig}, there exists a morphism
\[ \pi^\ps\colon\cE_{F,\ovl t}\to\fR_{\ovl t} \]
of rigid analytic spaces over $K$. 
On the other hand, since $\Theta^\full{\wtl\cE_F}(U_p)$ is a nowhere vanishing analytic function on $\cE_{F,\ovl t}$, it defines a rigid analytic morphism $\cE_{F,\ovl t}\to\G_m$. We define a map
\[ \pi_{\wtl\cE_{F,\ovl t}}=\pi_{\wtl\cE_{F,\ovl t}}^\ps\times\Theta^\full_{\wtl\cE_{F,\ovl t}}(U_p)\colon\wtl\cE_{F,\ovl t}\to\fR_{F,\ovl t}\times_K\G_m \]
of rigid analytic spaces over $K$. We denote by $\pi_{\cE_{F,\ovl t}}$ the restriction of $\pi_{\wtl\cE_{F,\ovl t}}$ to $\cE_{F,\ovl t}$.

We denote by $\wtl\calC_{F,\ovl t}$ (respectively, $\calC_{F,\ovl t}$) the (analytic Zariski-)closure in $\fR_{\ovl t}\times_K\G_m$ of the set
\[ \wtl\cS^\cl=(t_x,\Theta^\full_{\wtl\cE_{F,\ovl t}}(U_p)(x)), \]
where $x$ varies among the classical points of $\wtl\cE_{F,\ovl t}$ (respectively, $\cE_{F,\ovl t}$), equipped with its reduced structure of closed subspace of $\fR_{\ovl t}\times_K\G_m$. 
The following result can be proved in the same way as \cite[Theorem 7.5.1]{colmaz}, which deals with the case $F=\Q$. 



\begin{prop}\label{EFclosure}
The map $\pi_{\wtl\cE_{F,\ovl t}}\colon\wtl\cE_{F,\ovl t}\to\fR_{F,\ovl t}\times_K\G_m$ 
factors through an isomorphism
\[ \wtl\cE_{F,\ovl t}\cong\wtl\calC_{F,\ovl t} \]
of rigid analytic spaces over $K$,
The same statements hold if we replace $\wtl\cE_{F,\ovl t}$ with $\cE_{F,\ovl t}$ and $\wtl\calC_{F,\ovl t}$ with $\calC_{F,\ovl t}$.
\end{prop}

\begin{rem}
Let $\wtl\calC^\prime_{F,\ovl t}$ be the closure of the set $\wtl\cS^\cl$ inside of $\fR_{F,\ovl t}\times_K\A^1$. The intersection $\wtl\calC^\prime_{F,\ovl t}\cap\fR_{F,\ovl t}\times_K\G_m$ can be strictly larger than $\wtl\calC_{F,\ovl t}$: sequences of points of $\wtl\calC_{F,\ovl t}$ accumulating at a point of $\fR_{F,\ovl t}\times_K\{0\}$ will give rise to rigid analytic arcs in $\wtl\calC_{F,\ovl t}^\prime$. In particular, the Zariski closures of the Coleman--Stein families from Example \ref{excolste} will appear in $\wtl\calC_{F,\ovl t}^\prime$. 
\end{rem}


\subsection{Adapted neighborhoods on the parallel weight eigenvariety} 
Keep notation as above. For later use, we state a few results about the local structure of the $\ovl t$-part of the parallel weight eigenvariety $\cE_{F,\ovl t}$. 

Given an affinoid subspace $W$ of $\cW_F$ and a positive real number $h$, consider the set 
\[ U_{W,h}=\{x\in\cE_F(\C_p)\,\vert\,\omega_{\cE_{F,\ovl t}}(x)\in W\text{ and }v_p(\pi^{U_p}(x))\le h\}. \]
The following lemma is proved in exactly the same way as \cite[Proposition II.3.9]{bellaiche}.

\begin{lemma}\label{adapted}
	For every point $y$ of $\cW_F$ and positive real number $h$, there exists an affinoid neighborhood $W$ of $y$ such that $U_{W,h}$ admits a structure of affinoid subspace of $\cE_F$ and the restriction of $\omega_{\cE_{F,\ovl t}}$ to a map $U_{W,h}\to W$ is finite. 
\end{lemma}

We say that an affinoid neighborhood $W$ of $y\in\cW_F(\C_p)$ satisfying the condition in Lemma \ref{adapted} is \emph{adapted} to $h$. Note that $W$ is adapted to $h$ if and only if $(W,h)$ is a \emph{slope datum} in the terminology of \cite{johnewfun}.

\subsection{Nearly ordinary components}



Hida constructed nearly ordinary families of Hilbert modular forms in \cite{hidanord}. 
One would expect Hida's nearly ordinary families to correspond, after rigidification, to nearly ordinary components of the eigencurve $\wtl\cE_F$. However, we do not know if the two objects live over the same weight space unless we know or assume the Leopoldt conjecture for $F$. However, we can still use Hida's theory to describe the nearly ordinary part of the eigencurve, as we do below. This will be sufficient to our purposes.


We refer to \cite{hidaseveral} for the definition of nearly ordinary Hilbert eigenforms for $\GL_{2/F}$, and of the universal nearly ordinary Hecke algebra $\fh_F^{\nord}$ of tame level $N$ for $\GL_{2/F}$. 
The ring $\fh_F^{\nord}$ is a flat $\wtl\Lambda_{2,F}$-algebra, finite as a $\wtl\Lambda_{2,F}$-module. Let $\fh_F^{\nord,\rig}$ be the rigid analytic generic fiber of the $\cO_K$-formal scheme $\Spf\fh_F^\nord$. It is a rigid analytic space over $K$, and the structure map $\wtl\Lambda_{2,F}\to\fh^{\nord}$ induces a finite morphism of rigid analytic spaces $\fh_F^{\nord,\rig}\to\wtl\cW_{2,F}$. 

\begin{defin}
We call an irreducible component $\wtl\cX$ of $\wtl\cE_F$ \emph{nearly ordinary} if every classical specialization of $\wtl\cX$ is a Hilbert eigenform that is nearly ordinary at every $p$-adic place of $F$. We denote by $\wtl\cE_F^\nord$ the union of the nearly ordinary irreducible components of $\wtl\cE_F$.

We call an irreducible component $\cX$ of $\cE_F$ \emph{ordinary} if every classical specialization of $\cX$ is a Hilbert eigenform that is ordinary at every $p$-adic place of $F$.
\end{defin}



Recall that $\wtl\cW_F$ is the rigid generic fiber of $\Spf\cO_K[[\cO_{F,p}^\times\times\Z_p^\times]]$, and as such is equipped with a universal character
\[ \kappa^{\univ}\colon\cO_{F,p}^\times\times\Z_p^\times\to\wtl\Lambda_F^\times. \]
Let $\Norm\colon\cO_{F,p}^\times\to\Z_p^\times$ be the map obtained by tensoring $\Norm_{\cO_F/\Z}\colon\cO_F\to\Z$ with $\Z_p$. We introduce two group homomorphisms
\begin{align*} \iota_1\colon\cO_{F,p}^\times&\to\cO_{F,p}^\times\times\Z_p^\times \\
a&\mapsto(a,\Norm(a))
\end{align*}
and
\begin{align*} \iota_2\colon\cO_{F,p}^\times&\to\cO_{F,p}^\times\times\Z_p^\times \\
a&\mapsto(a^{-1},\Norm(a)).
\end{align*}
For a $p$-adic place $\fp$ of $F$, we denote by $I_{F,\fp}$ the inertia subgroup of $G_{F_\fp}$ and by $\rec_\fp\colon I_{F_\fp}\to\cO_{F_\fp}^\times$ the local reciprocity map. We use the same notation for the composition of $\rec_\fp$ with the natural inclusion $\cO_{F_\fp}^\times\to\cO_{F,p}^\times$.


\begin{prop}\mbox{ }\label{Enordprops}
\begin{enumerate}[label=(\roman*)]
\item There exists an isomorphism of rigid analytic spaces
\[ \iota^\nord\colon\fh_F^{\nord,\rig}\times_{\wtl\cW_{2,F}}\wtl\cW_F\to\wtl\cE_F^\nord. \]
If Leopoldt's conjecture holds for $F$, then $\wtl\cW_{2,F}\cong\wtl\cW_F$ and $\iota^\nord$ is an isomorphism $\fh_F^{\nord,\rig}\to\wtl\cE_F^\nord$. 
\item For every nearly ordinary Hecke eigenform $f$ of tame level $N$, there exists a $\Qp$-point $x$ of $\wtl\cE_F^\nord$ such that $\ev_x\ccirc\alpha_{\wtl\cE_F}$ is the Hecke eigensystem of $f$. 
\item Every nearly ordinary component of $\wtl\cE_F$ intersects $\cE_F$ in a union of ordinary irreducible components, and every ordinary irreducible component of $\cE_F$ is contained in a nearly ordinary irreducible component of $\wtl\cE_F$. In other words, there is a cartesian diagram
\begin{center}
\begin{tikzcd}
\cE_F^{\ord}\arrow{d}{\omega_{\cE_F^\ord}}\arrow{r}&\wtl\cE_F^{\nord}\arrow{d}{\omega_{\wtl\cE_F^\nord}} \\
\cW_F\arrow{r}{\iota}&\wtl\cW_F,
\end{tikzcd}
\end{center}
where the vertical maps are the restrictions of the usual weight maps, and the top horizontal arrow is induced by the inclusion $\cE_F\into\wtl\cE_F$.
\item A component $\cX$ of $\wtl\cE_F$ is nearly ordinary if and only if, for every $p$-adic place $\fp$ of $F$, the pseudorepresentation $t_\cX\vert_{G_{F_\fp}}$ is decomposable.
\item For a nearly ordinary irreducible component $\cX$ of $\wtl\cE_F$, the pseudorepresentation $t_\cX$ is the trace of an irreducible continuous representation 
\[ \rho_\cX\colon G_{F,Np}\to\GL_2(\Frac(\cO_X(\cX))) \]
such that, for every $p$-adic place $\fp$ of $F$, $\rho_\cX\vert_{G_{F_\fp}}$ is reducible and of the form 
\begin{equation}
	\begin{pmatrix}
			\vareps & \ast \\
			0 & \delta,
		\end{pmatrix}
\end{equation}
where $\chi_1,\chi_2$ are two characters satisfying
\begin{align*}
\vareps^2\vert_{I_{F,\fp}}&=(\kappa^{\univ}\ccirc\iota_1\cdot\Norm^{-2})\ccirc\rec_\fp \\
\delta^2\vert_{I_{F,\fp}}&=\kappa^{\univ}\ccirc\iota_2\ccirc\rec_\fp.
\end{align*}
\end{enumerate}
\end{prop}


\begin{proof}
As in the proof of Proposition \ref{EFclosure}, one can show that both $\fh_F^{\nord,\rig}\times_{\wtl\cW_{2,F}}\wtl\cW_F$ and $\wtl\cE_F^\nord$ are eigenvarieties for the datum $\cZ^\nord\subset\Hom_{\mathrm{ring}}(\calH,\Qp)\times\wtl\cW_{F}(\Qp)$ of pairs of Hecke eigensystems of nearly ordinary Hilbert eigenforms and their weights.

The other statements in the proposition follow by combining point (i) with the properties of $\fh^\nord$ stated in \cite[Introduction]{hidanord}, and the properties of its associated big Galois representation stated in \cite[Theorem I]{hidaseveral}. 
\end{proof}


%
%


\subsection{CM components}

Let $K$ be a totally imaginary quadratic extension of $F$ and $\Sigma_K$ a CM type for $K$, so that if $c$ is the complex conjugation in $\Gal(K/F)$, then $\Sigma_K\amalg\Sigma_Kc$ is the set of field embeddings $K\into\C$. Let $\bk=(k_\sigma)_\sigma\in\Z_{\ge 2}^{\Sigma_K}$, and let $\psi$ an algebraic Gr\"ossencharacter of $K$ that has type $(k_\sigma-1,0)$ for every pair $(\sigma,\sigma c)$ with $\sigma\in\Sigma_K$. Let $\fm$ be the conductor of $\psi$, $\fd$ the discriminant of the extension $K/F$, and $\Nm_{K/F}$ the ideal norm relative to the extension $K/F$. 
Then by the work of Yoshida there exists a Hilbert eigenform $f_\psi$ of level $\Nm_{K/F}(\fm)\fd$ and weight $(\bk,1)$ for $\GL_{2/F}$ such that, for every ideal $\fa$ of $\cO_F$ prime to $\Nm_{K/F}(\fm)\fd p$, the eigenvalue of the operator $T(\fa)$ acting on $f_\psi$ is
\[ a_\fa(\psi)\coloneqq\sum_{\substack{\fb\textrm{ ideal of }\cO_F, \\ \Nm_{K/F}(\fb)=\fa,\,(\fb,\fm)=1}}\psi(\fb). \]


More generally, we give the following definition, where $K$ is not assumed to be CM.

\begin{defin}\label{defprojdi}
We will say that a $\GL_{2/F}$-eigenform $f$ is \emph{projectively dihedral} if there exists a quadratic extension $K$ of $F$, a Gr\"ossencharacter $\psi$ of $K$ of conductor $\fm$, and a finite set $S$ of finite places of $\cO_F$, such that, for every ideal $\fa$ of $\cO_F$ coprime with the places in $S$, the eigenvalue of the operator $T(\fa)$ acting on $f$ is $a_\fa(\psi)$.
In this case, we also say that $f$ is \emph{attached} to $\psi$.
If we can choose $K$ to be a CM field, we say that $f$ has \emph{complex multiplication}, or \emph{is CM} in short.
\end{defin}


For later use, we state two results, essentially group-theoretic. The first one is proved exactly as \cite[Lemma 2.7]{bushen}; one simply observes that in the proof of \emph{loc. cit.} it is possible to replace the coefficient field $\C$ with any field of characteristic 0, and the smoothness of $(\pi,V)$ is not needed.

\begin{lemma}\label{Hss}
	Let $L$ be any field of characteristic 0, $V$ a finite dimensional $L$-vector space, $G$ a group, $H$ a finite index subgroup of $G$ and $\rho\colon G\to\GL(V)$ a representation. Then $\rho$ is semisimple if and only if $\rho\vert_H$ is semisimple.
\end{lemma}

In the following lemma we recall a result of Ribet and two immediate consequences of it. We identify Dirichlet characters with characters of $G_\Q$ in the usual way. 

\begin{lemma}\label{ribetind}
	Let $C$ be an algebraically closed field.
	\begin{enumerate}[label=(\roman*)]
		\item Let $H$ be a compact subgroup of $\GL_2(C)$. If $H$ admits an abelian subgroup $H_0$ of finite index, then it admits an abelian subgroup of index 2 containing $H_0$. 
		\item Let $G$ be a compact group, $\rho\colon G\to\GL_n(C)$ a semisimple representation and $t$ its trace. If $H$ is a finite index subgroup of $G$ such that $t\vert_H$ is decomposable, then $\rho\vert_H$ is abelian. 
		\item Let $G$ be a compact group and $t\colon G\to C$ be a continuous 2-dimensional pseudorepresentation. If $t$ is decomposable up to some finite index, then it is decomposable up to index 2.
	\end{enumerate}
\end{lemma}

\begin{proof}
	Part (i) is \cite[Proposition 4.4]{ribetneben}. For part (ii), let $H$ be a finite index subgroup of $G$ such that $t\vert_H$ is decomposable. By Lemma \ref{Hss}, $\rho\vert_H$ is semisimple. Write $t$ as a direct sum of characters $\chi_i\colon H\to C^\times$. Since $\rho\vert_H$ and $\chi_i$ are both semisimple and have the same trace, they are isomorphic by the last statement of Theorem \ref{liftings}. Therefore $\rho\vert_H$ is abelian.
	
	For part (iii), pick any continuous, semisimple representation $\rho\colon G\to\GL_2(C)$ with trace $t$, and let $H$ be a finite index subgroup of $G$ such that $t\vert_H$ is decomposable. By part (ii), $\rho\vert_H$ is abelian, so that by part (i) $\rho$ is abelian up to index 2, hence this also holds for its trace $t$.
	
	For part (iv), consider the Galois representation $\rho\colon G_\Q\to\GL_2(\Qp)$ attached to $f$. Since $f$ is cuspidal, $\rho$ is irreducible. 
\end{proof}

We recall a standard characterization of Hilbert eigenforms associated with Gr\"ossencharacters.


\begin{lemma}\label{indCM}
	Let $f$ be a $\GL_{2/F}$-eigenform. Let $\pi$ be the automorphic representation of $\GL_2(\A_F)$, and $\rho_{f,p}\colon G_F\to\GL_2(\C_p)$ the $p$-adic Galois representation, attached to $f$. The following are equivalent:
	\begin{enumerate}[label=(\roman*)]
		\item $\rho_{f,p}$ is induced by a character of $G_K$ for a quadratic extension $K$ of $F$; 
		\item the image of $\rho_{f,p}$ is projectively dihedral;
		\item[(iii.a)] there exists a nontrivial quadratic character $\eta$ of $G_F$ such that $\rho_{f,p}\otimes\eta\cong\rho_{f,p}$;
		\item[(iii.b)] there exists a nontrivial finite order character $\eta$ of $G_F$ such that $\rho_{f,p}\otimes\eta\cong\rho_{f,p}$;
		\item[(iv.a)] there exists a nontrivial quadratic character $\omega$ of $\A_F^\times$ such that $\pi\cong\pi\otimes(\omega\ccirc\det)$;
		\item[(iv.b)] there exists a nontrivial finite order character $\omega$ of $\A_F^\times$ such that $\pi\cong\pi\otimes(\omega\ccirc\det)$;
		\item[(v)] $\pi$ is monomial, in the sense that it is the base change to $\GL_2(\A_F)$ of a Gr\"ossencharacter of a quadratic extension $K$ of $F$;
		\item[(vi)] $f$ is attached to a Gr\"ossencharacter of a quadratic extension $K$ of $F$. 
	\end{enumerate}
	Moreover, if $f$ is of weight $((k_\sigma)_{\sigma\in\Sigma_F},w)$ with $k_\sigma\ge 2$ for every $\sigma$ and any of (i-vi) holds, then the extension $K$ appearing there is uniquely determined and totally imaginary, so that $f$ is a CM form.
\end{lemma}

\begin{proof}
	The equivalences (i)$\iff$(ii)$\iff$(iii.a)$\iff$(iii.b) follow from Lemma \ref{ribetind}, while (iii.a)$\iff$(iv.a) and (iii.b)$\iff$(iv.b) come from the compatibility of the global Langlands correspondence with base change. The equivalence (v)$\iff$(vi) is a matter of unraveling the definitions. The implication (v)$\implies$(iv.a) also comes from the properties of the global Langlands correspondence. The implication (iv.a)$\implies$(v) is \cite[Proposition 6.5]{lablan} (or \cite[Proposition 4.5]{ribetneben} in the special case $F=\Q$ and weight $\ge 2$).
	The last statement is contained in \cite[Proposition 4.5]{ribetneben}.
\end{proof}

\begin{defin}
	We say that an irreducible component $\cX$ of $\wtl\cE_F$ (or $\cE_F$) is CM if all of its classical specializations of weight $\ge 2$ are CM. 
\end{defin}

\subsection{Base change from the rationals to a real quadratic field}

We fix some notation. We write $\Art_\Q\colon\A_\Q^\times\to G_\Q$ for the Artin map of global class field theory, and $\det\colon\GL_2(\A_\Q)\to\A_\Q^\times$ for the determinant. Given a number field $L$, an automorphic representation $\pi$ of $\GL_2(\A_L)$, and a place $v$ of $L$, we denote by $\pi_v$ the admissible representation of $\GL_2(L_v)$ appearing in the tensor product decomposition $\pi=\otimes_v\pi_v$. We denote by $\rec_v$ the reciprocity map from the local Langlands correspondence, associating with an admissible representation of $\GL_2(L_v)$ a Weil--Deligne representation of $W_{L_v}$. We do not specify the base field in the notation $\rec_v$, since it is implicit in the choice of $v$.

In the following $F$ is a totally real field, Galois and cyclic over $\Q$. 
We recall a result about the base change of automorphic representations from $\GL_{2}(\A_\Q)$ to $\GL_2(\A_F)$, 
essentially due to Arthur and Clozel \cite{artclobc}. We refer the reader to \cite[Lemma 5.1.1]{satotatehilb} and the references there for the form of the statement that we use. 

\begin{thm}\label{artclo}
Let $\pi$ be a cuspidal automorphic representation of $\GL_{2}(\A_\Q)$ such that
\begin{equation*}\label{notwists} \pi\ncong\pi\otimes(\chi\ccirc\Art_\Q\ccirc\det)
\end{equation*}
for every character $\chi\colon G_\Q\to\C^\times$ factoring through $\Gal(F/\Q)$. 
Then there exists a cuspidal automorphic representation $\BC_{F/\Q}(\pi)$ of $\GL_2(\A_F)$ such that, for every place $v$ of $F$ lying over a prime $q$,
\[ \rec_v(\Pi_v)=\rec_q(\pi_q)\vert_{W_{F_v}}. \]
\end{thm}






%
%



\begin{cor}\label{bcform}
Let $f$ be a cuspidal $\GL_{2/\Q}$-eigenform of weight $k\ge 2$, tame level $N$ and nebentypus $\chi$, and let $\pi$ be the automorphic representation of $\GL_2(\A_\Q)$ attached to $f$. Then the cuspidal representation $\pi$ satisfies the assumption of Theorem \ref{artclo}, so that a base change $\BC_{F/\Q}(\pi)$ exists, and is associated with a Hilbert modular form $\BC_{F/\Q}(f)$ of level $N$ and (parallel) weight $(2k,2k,0,\chi^2\ccirc\Norm_{\cO_F/\Z},\triv)$ (in the notation of Section \ref{secwtcl}).
\end{cor}

\begin{proof}
If $\pi\cong\pi\otimes(\chi\ccirc\Art_\Q\ccirc\det)^i$ for some $\chi\colon\Gal(F/\Q)\to\C^\times$ and $i$, $1\le i\le[F\colon\Q]-1$, then by the equivalence (iv.b)$\implies$(iv.a) of Lemma \ref{indCM}, $\chi$ and $i$ can be chosen so that $\chi^i$ is quadratic, which means that $f$ has multiplication by a quadratic subfield of $F$. 
Since $f$ has weight at least 2, it cannot have real multiplication because of the last statement of Lemma \ref{indCM}; in particular, since $F$ is a totally real field, $f$ cannot have multiplication by $F$. Therefore $\pi$ satisfies condition \eqref{notwists}. 
\end{proof}

\begin{rem}\label{bcmap}
As before, let $\calH_F^{Np}$ be the spherical $\GL_{2/F}$-Hecke algebra away from $Np$, and let $\BC_{F/\Q}\colon\calH_F^{Np}\to\calH_\Q^{Np}$ be the homomorphism defined in \cite[Section 4.3]{johnewfun}. 
For $f$ as in Corollary \ref{bcform}, with Hecke eigensystem $\alpha\colon\calH_\Q^{Np}\to\Qp$ away from $Np$, the Hecke eigensystem of $\BC_{F/\Q}(f)$ is given by the composition $\BC_{F/\Q}\ccirc\alpha$. The weight of $\BC_{F/\Q}(f)$ is obtained from the weight of $f$ after composition with the morphism of weight spaces described in \emph{loc. cit.}.
\end{rem}

\medskip

\section{Locally decomposable families of Hilbert modular forms}\label{secbgv}

Let $F$ be a totally real field. We denote by $\zeta_p$ a primitive $p$-th root of unity. Recall that $p$ is assumed to be odd throughout the paper. We say that a 2-dimensional representation is \emph{decomposable} if it is the direct sum of two characters. As in the previous section, $\GL_{2/F}$ denotes the group $\Res_{F/\Q}\GL_{2/F}$, $N$ is an integer at least 4, and $\wtl\cE_F$ (respectively, $\cE_F$) denotes the full (respectively, parallel weight) eigenvariety of tame level $N$ for $\GL_{2/F}$

The goal of this section is to generalize the following result of Balasubramanyam, Ghate and Vatsal to the case where $p$ has arbitrary ramification in $F$.

\begin{thm}[{\cite{balghavathilb}}]\label{bgv}
Assume the $p$ splits completely in $F$.
Let $\cF$ be an ordinary (respectively, nearly ordinary) irreducible component of $\cE_F$ (respectively, $\wtl\cE_F$) such that:
\begin{enumerate}
\item $\cF$ is $p$-distinguished,
\item $\ovl\rho_\cF\vert_{G_{F(\zeta_p)}}$ is absolutely irreducible,
\item $\rho_\cF\vert_{G_{F,\fp}}$ is decomposable for every $p$-adic place $\fp$ of $\cO_F$.
\end{enumerate}
Then $\cF$ is a CM component.
\end{thm}

The statement for parallel weight families is proved in \cite[Theorem 3]{balghavathilb}, while \cite[Section 3]{balghavathilb} explains how to adapt the proof to the full eigenvariety. 
We stated the result in terms of nearly ordinary components of eigenvarieties, rather than nearly ordinary Hecke algebras as in \cite{balghavathilb}. The result in \emph{loc. cit.} implies Theorem \ref{bgv} because of Proposition \ref{Enordprops}(i), and the two are equivalent if Leopoldt's conjecture holds for $F$. 

\begin{rem}
Our eigenvarieties are only defined for tame level $N\ge 4$, so all the results are stated in that setting. This is not a problem regarding our intended applications, since we can always embed finite slope families of lower tame level into an eigenvariety of tame level $N\ge 4$, in the obvious way.
\end{rem}

Results of the shape of Theorem \ref{bgv} can be interpreted as a version for $p$-adic families of a conjecture of Greenberg \cite[Question 1]{ghavatord} on how CM eigenforms can be characterized by the fact that their associated Galois representation is decomposable locally at $p$ (in the form of a question, the same problem was posed independently by Coleman in \cite[Remark 2, Section 7]{colclass}).

The proof relies in a crucial way on a modularity result of Sasaki for representations of $G_{F}$ that decompose as a sum of two potentially unramified characters at every $p$-adic place (see \cite[Theorem 2]{balghavathilb}).

Note that \cite[Theorem 3]{balghavathilb} is actually stated for ordinary $\Lambda$-adic forms; however, the authors comment in Section 3 that the same arguments as in the ordinary case give the above result. In any case, our result will not depend on Theorem \ref{bgv}; rather, we will use the same strategy to prove what we need. We also remark that the proof of \cite[Theorem 3]{balghavathilb} is not entirely contained in \emph{loc. cit.}; for some parts of it the authors refer to earlier work of Ghate and Vatsal in the case $F=\Q$ \cite{ghavatord}, from which some arguments can be easily adapted. In this section we try to give all details as to how the arguments work in our setting, since at various points they need to be adapted in a non-trivial way.

The main assumption that we would like to remove is that regarding the splitting behavior of $p$ in the quadratic field $F$. As foreseen by the authors in \cite{balghavathilb} (see their Remark just after the statement of Theorem 3), this is now possible thanks to the modularity results of Pilloni and Stroh \cite{pilstrmod} that were still not available at the time when they were writing. We include this new input below. 

As in \cite{balghavathilb}, we prove our result both for parallel and general weight families, though in our application the parallel weight case will be sufficient.

\begin{thm}\label{bgvaff}
Let $U$ be an irreducible affinoid subdomain of $\cE_F$ (respectively, $\wtl\cE_F$). 
Assume that:
\begin{enumerate}[label=(\roman*)]
\item\label{zetapirr} $\ovl\rho_U\vert_{G_{F(\zeta_p)}}$ is irreducible;
\item\label{5hyp} if $p=5$ and $\Proj(\ovl\rho_U(G_{F,Np}))\cong\PGL_2(\F_5)$, then $F(\zeta_5)\colon F]=4$;
\item\label{Sab} there exists a dense subset $S^\ab$ of $U(\Qp)$ such that $\rho_{U,x}\vert_{G_{F_\fp}}$ is decomposable for every $x\in S^\ab$ and every $p$-adic place $\fp$ of $F$.
\end{enumerate}
Then $U$ is contained in a CM component of $\cE_F$ (respectively, $\wtl\cE_F$). 
\end{thm}



Note that we ask for actual representations to be locally decomposable, not just pseudorepresentations; all pseudorepresentations attached to nearly ordinary families are decomposable by Proposition \ref{Enordprops}.


\begin{lemma}\label{psplit}
Let $\cX$ be a CM component of $\cE_F$ (respectively, $\wtl\cE_F$). Then $\cX$ is ordinary (respectively, nearly ordinary) and has multiplication by a unique totally imaginary quadratic extension $K$ of $F$, in which every $p$-adic place of $F$ splits.
\end{lemma}

\begin{proof}
This follows from a direct calculation analogous to the one that proves \cite[Corollary 3.6]{cit}: one computes the slopes of Hilbert CM modular forms (a calculation analogue to that of Remark \ref{CMclass} below in the case $F=\Q$), and relies on the fact that the slope is locally constant along a $p$-adic family to deduces that the only eigenforms of finite slope that can fit in a family are the ordinary ones, that only appear when all of the $p$-adic places of $F$ split in $K$. 
\end{proof}





Let $U$ be an integral subdomain of $\wtl\cE_F$ (respectively, $\cE_F$) such that $\ovl\rho_U\vert_{G_{F(\zeta_p)}}$ is irreducible, and let $t_U\colon G_F\to\cO_U(U)$ be the pseudorepresentation attached to $U$. 

\begin{lemma}\label{Udec}
Let $\fp$ be a $p$-adic place of $F$. Assume that there exists a Zariski dense subset $S^\ab_\fp$ of $\ovl U(\Qp)$ such that $\rho_{\ovl U,x}\vert_{G_{F_\fp}}$ is decomposable for every $x\in S^\ab_\fp$. Then $U$ is contained in a nearly ordinary component of $\wtl\cE_F$ (respectively, an ordinary component of $\cE_F$), and $t_{U}$ is the trace of a continuous representation $\rho_{U}\colon G_{F,Np}\to\GL_2(\cO_U(U))$ such that $\rho_{U}\vert_{G_{F_\fp}}$ is decomposable for every $p$-adic place $\fp$ of $F$. 
\end{lemma}

\begin{proof}
A simple check shows that the locus of $U$ where $\rho_{\ovl U,x}\vert_{G_{F_\fp}}$ is decomposable is Zariski closed in $U$, hence coincides with the whole $U$. The Lemma then follows from the nearly ordinary $R=T$ theorem of Fujiwara \cite[Theorem 11.1]{fujihilb}.
\end{proof}

We will replace Sasaki's modularity result \cite[Theorem 2]{balghavathilb} with the following theorem of Pilloni and Stroh. For an arbitrary field embedding $\tau\colon F\to\Qp$, we denote by $c_\tau\in G_{F,Np}$ the associated complex conjugation. A representation $\rho\colon G_{F,Np}\to\GL_2(\Qp)$ is \emph{totally odd} if $\rho(c_\tau)=-\Id$ for for every $\tau$. 

\begin{thm}[{\cite[Théorème 0.2]{pilstrmod}}]\label{psmod}
Let $\rho\colon G_{F,Np}\to\GL_2(\Qp)$ be a continuous, totally odd, geometric representation such that
\begin{enumerate}
	\item $\ovl\rho\vert_{G_{F(\zeta_p)}}$ is irreducible;
	\item if $p=5$ and $\Proj(\ovl\rho(G_{F,Np}))\cong\PGL_2(\F_5)$, then $[F(\zeta_5)\colon F]=4$;
	\item\label{HT0} for every $p$-adic place $\fp$ of $F$, $\rho\vert_{G_{F_\fp}}$ is potentially semistable of Hodge--Tate weights all 0.
\end{enumerate}
Then $\rho$ is attached to a cuspidal Hilbert modular form for $\GL_{2/F}$ of parallel weight 1. In particular, $\rho$ has finite image.
\end{thm}


Recall that, for a $p$-adic field $E$, we denote by $I_E$ the inertia subgroup of $G_E$.

\begin{rem}\label{senHT0}
Let $\fp$ be a $p$-adic place of $F$. By a well-known theorem of Sen, $\rho\vert_{G_\fp}$ is potentially semi-stable of Hodge--Tate weights all 0 if and only if $\rho(I_{F_\fp})$ is finite.
\end{rem}


We will prove in Proposition \ref{wt1CM} that, if $\cF$ is an irreducible component of $\cE_F$ for which $\rho_\cF$ is locally the sum of two characters, then $\cF$ contains a dense subset of points whose associated representations satisfy the conditions of Theorem \ref{psmod}. We mimic the proofs of \cite{ghavatord,balghavathilb} in our more general setting where $p$ is not assumed to split in $F$.

%

In the following, let $\cX$ be a nearly ordinary component of $\wtl\cE_F$, and assume that: 

\begin{itemize}
\item $\ovl\rho_\cX\vert_{G_{F(\zeta_p)}}$ is irreducible;
\item if $p=5$ and $\Proj(\ovl\rho_\cX(G_{F,Np}))\cong\PGL_2(\F_5)$, then $F(\zeta_5)\colon F]=4$;
\end{itemize}
This will allow us to apply Theorem \ref{psmod} to the weight 1 specializations of $\cX$.

\begin{lemma}\label{CMdense}
If the set of classical projectively dihedral specializations of $\cX$ is dense in $\cX$, then $\cX$ is a CM component.
\end{lemma}

\begin{proof}
By \ref{indCM}, each point in $x\in X_2$ corresponds to an eigenform $f_x$ with multiplication by a quadratic extension $K_x$ of $F$ of some discriminant $D_x$, that will then divide the level of $f_x$. By Proposition \ref{psplit}, $p$ has to split in $K_x$, hence does not divide $D_x$, which means that $D_x$ divides the tame level $N$ of $f_x$. Therefore, there are only a finite number of choices for the extension $K_x$. As in \cite[Proof of Theorem 3]{balghavathilb}, we deduce that for one such choice, say $K$, the set of points of $\cX$ that have multiplication by $K$ is dense, and conclude that the whole $\cX$ has multiplication by $K$. Since the weight map is finite, $\cX$ always admits a specialization of parallel weight 2, that will be classical by Hida's control theorem. Such a specialization will have multiplication by $K$, so by Lemma \ref{indCM} $K$ has to be a totally imaginary extension of $F$. Note that this argument from \emph{loc. cit.} allows one to avoid having to prove a result of the type of \cite[Lemma 16]{ghavatord}. 
\end{proof}

The assumption on the local decomposability of the Galois representation attached to a form or a component only makes an appearance in the following lemma, but it is a crucial one to the argument.

\begin{lemma}\label{wt1}\mbox{ }
\begin{enumerate}[label=(\roman*)]
\item A parallel weight 1 specialization $x$ of $\cX$ is classical if and only if, for every $p$-adic place $\fp$ of $F$, $\rho_x\vert_{I_{F_\fp}}$ is the direct sum of two characters.
\item Assume that, for every $p$-adic place $\fp$ of $F$, $\rho_\cX\vert_{G_{F_\fp}}$ is the direct sum of two characters. Then the set of classical weight 1 specializations of $\cX$ is dense in $\cX$.
\end{enumerate}
\end{lemma}

\begin{proof}
We prove (i). 
Let $x$ be a specialization of $\cX$ and $\fp$ a $p$-adic place of $F$. Since $\cX$ is nearly ordinary, $\rho_\cX\vert_{I_{F_\fp}}$ is reducible. In particular, $\rho_x\vert_{I_{F_\fp}}$ is an extension of a character $\chi_1$ of $I_{F_\fp}$ by a second character $\chi_2$. If $x$ is of parallel weight 1, all of the Hodge--Tate weights of $\rho_x\vert_{G_{F_\fp}}$ are 0, so each of $\chi_1$ and $\chi_2$ is a Hodge--Tate character. A character of $G_{F_\fp}$ is Hodge--Tate if and only if it is potentially semi-stable, and it is Hodge--Tate of weight 0 if and only its restriction to $I_{F_\fp}$ factors through a finite quotient. Therefore both $\chi_1(I_{F_\fp})$ and $\chi_2(I_{F_\fp})$ are finite. 

If $x$ is classical, then $\rho_x\vert_{G_{F_\fp}}$ is potentially semi-stable of Hodge--Tate weights 0, hence by Remark \ref{HT0} $\rho_x(I_{F_\fp})$ is finite, which is only possible if $\rho_x\vert_{I_{F_\fp}}$ is the direct sum of $\chi_1$ and $\chi_2$.

Conversely, if $\rho_x\vert_{I_{F_\fp}}$ is the direct sum of $\chi_1$ and $\chi_2$, then $\rho_x\vert_{G_{F_\fp}}$ is potentially semistable since its restriction to $I_{F_\fp}$ factors through a finite quotient. If this is true for every $\fp$, then $\rho_x$ is geometric, given that we already know that it is unramified outside of $Np$. We are now in a position to apply Theorem \ref{psmod} to $\rho_x$ and deduce that $\rho_x$ is attached to a Hilbert eigenform of parallel weight 1, in other words that $x$ is classical.

As we recalled above, the weights of the form $\kappa_{1,0,\chi_1,\chi_2}$ are dense in $\wtl\cW_F$. Since $\cX$ is nearly ordinary, the restriction of the weight map to $\cX$ is a finite map, so the set of points of $\cX$ of weight 1 is dense in $\cX$. Part (ii) follows from this fact together with part (i).
\end{proof}

\begin{prop}\label{wt1CM}
Let $\cX$ be either a nearly ordinary cuspidal component of $\wtl\cE_F$, or an ordinary component of $\cE_F$, such that, for every $p$-adic place $\fp$ of $F$, $\rho_\cX\vert_{G_{F_\fp}}$ is the direct sum of two characters. Then the set of classical CM specializations of $\cX$ of parallel weight 1 is dense in $\cX$.
\end{prop}

\begin{proof}
We adapt to our setting the proof of the implication (ii)$\implies$(iii) of \cite[Proposition 14]{ghavatord}. 
Let $x$ be a point of $\cX$ corresponding to a Hilbert eigenform $f_x$ of parallel weight 1. The specialization $\rho_{\cX,x}$ is isomorphic to the $p$-adic Galois representation attached to $f_x$ by the work of Deligne--Serre (if $F=\Q$), Ohta and Rogawski--Tunnell. In particular, it is obtained from a continuous (hence finite image) complex representation $G_{F,Np}\to\GL_2(\C)$ via an isomorphism $\Qp\cong\C$; we still denote this complex representation by $\rho_{\cX,x}$. By the classification of finite subgroups of $\GL_2(\C)$, the image of the projective representation $\Proj\rho_{\cX,x}$ is isomorphic to one of the following:
\begin{enumerate}[label=(\arabic*)]
\item a cyclic group;
\item a dihedral group;
\item one of the permutation groups $A_4$, $A_5$ or $S_5$.
\end{enumerate}
For $i=1,2,3$, let $X_i$ be the set of weight 1, classical specializations $x$ of $\cX$ for which the image of $\Proj\rho_{\cX,x}$ is of type (i) above. 
By Lemma \ref{wt1}(ii), $X_1\cup X_2\cup X_3$ is dense in $\cX$. 
We show that $X_1$ and $X_3$ cannot be dense in $\cX$, so that $X_2$ has to be.

If $x\in S_1$, then $\rho_{\cX,x}$ is decomposable, so $x$ corresponds to an Eisenstein series for $\GL_{2/F}$. Since $\cX$ is a cuspidal component of $\cE_F$, the set of its Eisenstein specializations cannot be dense in $\cX$.

Now assume that $x\in X_3$. By assumption, the restriction $\rho_{\cX}\vert_{G_{F_\fp}}$ is isomorphic to 
\begin{equation}\label{xloc}
\begin{pmatrix} \vareps_\fp & 0 \\
	0 & \delta_\fp \end{pmatrix}
\end{equation}
for two characters $\vareps_\fp,\delta_\fp\colon G_{F_\fp}\to\wtl\Lambda_F^\times$. 
By assumption, the weight of $x$ is associated with an arithmetic ideal of $\wtl\Lambda_F$ of the form $P_{(0)_\sigma,0,\chi_1,\chi_2}$ (recalling that in the notation of Section \ref{secwt}, weight $k$ corresponds to Hodge--Tate weight $k-1$), i.e. with the character $\cO_{F,p}^\times\times\Z_p^\times\to\C_p^\times$ mapping $(x,y)$ to $\chi_1(x)\chi_2(y)$.

By Proposition \ref{Enordprops}(v), $\vareps_\fp^2\vert_{I_{F_\fp}}\ccirc\rec_\fp^{-1}\colon\cO_{F,p}^\times\to\C_p^\times$ is the character $x\mapsto\chi_1(x)\chi_2^{-1}(\Norm(x))$, and $\delta_\fp^2\vert_{I_\fp}\ccirc\rec_\fp^{-1}$ is the character $x\mapsto\chi_1^{-1}(x)\chi_2(\Norm(x))$. 
In particular, the order of both $\chi_1$ and $\chi_2$ is bounded by the size of the image of $\rho_{\cX,x}$, which is at most 120 since $x\in X_3$. There is only a finite number of characters $\chi_1$ and $\chi_2$ of every fixed finite order, and since the structure map $\Lambda_F\to\Lambda_\cF$ is finite, for every choice of $\chi_1$ and $\chi_2$ there exists only a finite number of points of weight $P_{-1,0,\chi_1,\chi_2}$. We conclude that the set $X_3$ is finite.

The case when $\cX$ is an ordinary component of $\cE_F$ is proved in exactly the same way, or by embedding $\cX$ into a nearly ordinary component of $\wtl\cE_F$.

As desired, we deduce that $X_2$ is dense in $\cX$, and Lemma \ref{CMdense} allows us to conclude.
\end{proof}


Combining Lemmas \ref{CMdense} and \ref{wt1CM} gives the following.

\begin{lemma}\label{XCM}
If $\cX$ is a nearly ordinary (respectively, ordinary) component of $\wtl\cE_F$ (respectively, $\cE_F$) such that $\rho_\cX\vert_{G_{F_\fp}}$ is decomposable for every $p$-adic place $\fp$ of $F$, then $\cX$ is a CM component.
\end{lemma}


%

\begin{proof}[Proof of Theorem \ref{bgvaff}]
Let $U$ be an affinoid of $\wtl\cE_F$ satisfying the assumptions of Theorem \ref{bgvaff}. Then the Zariski closure $\ovl U$ of $U$ in $\wtl\cE_F$ is an irreducible component $\cX$ of $\wtl\cE_F$. By Lemma \ref{Udec}, $\rho_{\cX}\vert_{G_{F_\fp}}$ is decomposable for every $p$-adic place $\fp$ of $F$ and $\cX$ is a nearly ordinary component of $\wtl\cE_F$. 
By Corollary \ref{XCM}, $\cX$ is a CM component.
\end{proof}

\medskip

\section{Families of eigenforms of infinite slope}\label{secfam}

Since this creates no ambiguity, we simply write $\GL_{2/F}$ in the rest of the paper for the $\Q$-group scheme $\Res_{F/\Q}\GL_{2/F}$.

\subsection{Eigenforms of infinite slope}

We start by recalling some standard facts about elliptic eigenforms of infinite slope. 
Let $f$ be a $\GL_{2/\Q}$-eigenform of level $\Gamma_1(Np^r)$ for some $r\ge 1$. Let $\sum_{n\in\N}a_nq^n$ be the $q$-expansion of $f$, and assume that $f$ is normalized so that $a_1=1$. Then the eigenvalue of the $U_p$-operator acting on $f$ is equal to $a_p$.

\begin{defin}
	We call \emph{slope} of $f$ the $p$-adic valuation of $a_p$. If $a_p=0$, we say that $f$ is a $\GL_{2/\Q}$-eigenform \emph{of infinite slope}.
\end{defin}

In the following, let $f$ be a newform of weight $k$, level $\Gamma_1(Np^r)$, with $p\nmid N$ and $r\in\N$, and character $\vareps$. Write $\vareps$ as a product of Dirichlet characters $\vareps^p\vareps_p$, where the conductor of $\vareps^p$ is prime to $p$ and that of $\vareps_p$ is a power of $p$. Let $\pi_{f,p}$ be the admissible representation of $\GL_2(\Q_p)$ attached to $f$, and let $\omega_p$ be the $p$-component of the character of $\A_\Q^\times/\Q^\times$ attached to $\vareps$. 

\begin{defin}[{\cite[Definition 2.7]{loewei}}]
The form $f$ is said to be \emph{$p$-primitive} if the $p$-part of its level is minimal among all its twists by Dirichlet characters.
\end{defin}

When we speak of the \emph{twists} of an eigenform, we are always implicitly referring to its twists by Dirichlet characters. We implicitly identify Dirichlet characters of number fields with finite order Gr\"ossencharacters, and with finite order Galois characters.


Given two admissible characters $\chi_1,\chi_2$ of $\Q_p^\times$ satisfying $\chi_1\chi_2^{-1}\ne\lvert\cdot\rvert^{\pm 1}$, we denote by $\pi(\chi_1,\chi_2)$ the principal series representation of $\GL_2(\Q_p)$ attached to $\chi_1$ and $\chi_2$. We denote by $\mathrm{St}$ the Steinberg representation of $\GL_2(\Q_p)$. We recall the following description of $\pi_{f,p}$.

\begin{prop}[{\cite[Proposition 2.8]{loewei}}]\label{lwclass}
	Assume that $f$ is $p$-new and $p$-primitive. Then:
	\begin{enumerate}
		\item If $r\ge 1$ and the conductor of $\omega_p$ is $p^r$, then $\pi_{f,p}\cong\pi(\chi_1,\chi_2)$ for two characters $\chi_1,\chi_2$ of $\Q_p^\times$ with $\chi_1$ unramified, $\chi_1(p)=a_p(f)/p^{(k-1)/2}$ and $\chi_1\chi_2=\omega_p$.
		\item If $r=1$ and $\omega_p$ is unramified, then $\pi_{f,p}\cong\mathrm{St}\otimes\chi$, where $\chi$ is the unramified character satisfying $\chi(p)=a_p(f)/p^{(k-2)/2}$.
		\item If neither of the above conditions hold, then $\pi_{f,p}$ is supercuspidal and the conductor of $\omega_p$ is at most $\lfloor r/2\rfloor$.
	\end{enumerate}
\end{prop}

\begin{rem}\label{lwslope}
If $f$ satisfies either (i) or (ii) of Proposition \ref{lwclass} then $a_p(f)\ne 0$, while in case (iii) $a_p(f)=0$: see for instance \cite[Table 1]{loewei} and use the fact that $\lambda(\pi_{f,p})=a_p(f)/p^{(k-1)/2}$ in the notation there.
\end{rem}

We deduce the following result for eigenforms that are not necessarily $p$-primitive.

\begin{cor}\label{infclass}
Let $f$ be an eigenform of weight $k$, level $\Gamma_1(Np^r)$, with $p\nmid N$ and $r\ge 1$ (we do not assume $f$ to be $p$-new, nor $p$-primitive). Then $f$ is of infinite slope if and only if it satisfies one of the following mutually exclusive conditions:
\begin{enumerate}
\item $\pi_{f,p}$ is a principal series representation attached to two ramified characters;
\item $\pi_{f,p}$ is a twist of the Steinberg representation with a ramified character;
\item $\pi_{f,p}$ is supercuspidal.
\end{enumerate}
In cases (i) and (ii), $f$ is a twist of an eigenform of finite slope by a Dirichlet character of $p$-power conductor.  
\end{cor}

\begin{proof}
Let $\chi$ be a Dirichlet character of $p$-power conductor such that $\chi f$ is the $p$-depletion of a $p$-primitive form $g$. 
The corollary follows by applying Proposition \ref{lwclass} and Remark \ref{lwslope} to $g$. 
The eigenform of finite slope appearing in the last statement can be taken to be $g$ itself. 
\end{proof}


\begin{defin}
	We say that $f$ is \emph{$p$-supercuspidal} if the representation $\pi_{f,p}$ is supercuspidal.
\end{defin}

The following theorem is an easy consequence of various results from the literature. In the case $F=\Q$, it allows us to reinterpret $p$-supercuspidality in terms of the \emph{$p$-adic} Galois representation, local at $p$, attached to an eigenform.

By combining Theorem \ref{finslopetri}, Corollary \ref{triquad} and Lemma \ref{pstdih}, we obtain:

\begin{cor}\label{infpottri}
	For a $p$-supercuspidal $\GL_{2/\Q}$-eigenform $f$, the representation $\rho_{f,p}\vert_{G_{\Q_p}}$ is not trianguline, but becomes so over a quadratic extension of $\Q_p$. It is not semistabelian, and becomes crystalline over a dihedral extension of $G_{\Q_p}$.
\end{cor}

Note that combining Corollary \ref{infpottri} with the classification of potentially trianguline, non-trianguline representations from Theorem \ref{berche} does not give anything interesting, since $p$-adic local at $p$ Galois representations attached to eigenforms always fall in case (ii) of Theorem \ref{berche}.


\subsection{$p$-adic families of eigenforms}

In this section we introduce two notions of ``family'' of eigenforms that are intended to be as weak as possible while remaining meaningful. They are modeled on the idea that a $p$-adic family should not only interpolate Hecke eigenvalues at good primes, which amounts to interpolating the associated Galois (pseudo)representations, but also the datum of a ``refinement'' at $p$.

As in Section \ref{triparsec}, $\calH_\Q$ denotes the Hecke algebra $\calH_\Q^{Np}\otimes_\Z\calH_{\Q,p}$, where $\calH_\Q^{Np}$ is generated over $\Z$ by the operators $T_\ell$ at the primes $\ell$ not dividing $Np$, and $\calH_p=\Z[U_p]$. 



Let $\{k_{i}\}_{i\in\N}$ and $\{r_{i}\}_{i\in\N}$ be two sequences of positive integers. 
For every $i\in\N$, let $f_i$ be a classical modular eigenform $f_i$ of weight $k_i$ and level $\Gamma_1(Np^{r_i})$. 

By the standard constructions of Eichler--Shimura, Deligne, and Deligne--Serre, each $f_i$ carries an associated semisimple continuous representation $\rho_i\colon G_{\Q,N_ip}\to\GL_2(\Qp)$, that is potentially semistable at $p$ by the results of Faltings and Tsuji. As in Section \ref{ptypessec}, we attach to $\rho_i$ a $(\varphi,N,G_{\Q_p})$-module $D_\pst(\rho_i)$. The trace of $\rho_i$ is a continuous pseudorepresentation $t_i\colon G_{\Q,N_ip}\to\Qp$. 
Let $\theta_i\colon\calH_\Q^{Np}\to\Qp$ be the Hecke eigensystem of $f_i$ outside of $Np$.



\begin{defin}
A \emph{family of Frobenius eigenvalues} for $(f_i)_{i\in\N}$ is a sequence $(\varphi_{i})_{i\in\N}$ where, for every $i\in\N$, $\varphi_i$ is an admissible Frobenius eigenvalue of $\rho_i\vert_{G_{\Q_p}}$. 
\end{defin}

By Remark \ref{admfrob}, the choice of an admissible Frobenius eigenvalue of $\rho_i\vert_{G_{\Q_p}}$ amounts to that of an $N$-stable refinement. Therefore, choosing a family of Frobenius eigenvalues for $(f_i)_{i\in\N}$ amounts to choosing an $N$-stable refinement for each of the eigenforms $f_i$. 

For every $i$, let $\theta_i\colon\calH_\Q^{Np}\to\C_p$ be the Hecke eigensystem of $f_i$ away from $Np$. Let $(\varphi_{i})_{i\in\N}$ be a family of Frobenius eigenvalues for $(f_i)_{i\in\N}$.

\begin{defin}\label{famdef}
	We say that the sequence $(f_i,\varphi_{i})_{i\in\N}$ is a \emph{weak $p$-adic family of eigenforms} if the forms $f_i$ are pairwise distinct and:
	\begin{itemize}
		\item[(i)] the sequence $(\theta_i)_i$ converges to a homomorphism $\theta_\infty\colon\calH_\Q^{Np}\to\C_p$, in the distance of Definition \ref{psdist}. 
	\end{itemize}
	We say that $(f_i,\varphi_{i})_{i\in\N}$ is a \emph{$p$-adic family of eigenforms} if, in addition to (i), the following holds:
	\begin{itemize}
		\item[(ii)]  the sequence $(\varphi_{i})_i$ converges to a limit $\varphi_\infty\in\C_p^\times$. 
	\end{itemize}
We identify two families $(f_i,\varphi_{i})_{i\in\N}$ and $(\wtl f_i,\wtl\varphi_{i})_{i\in\N}$ if, for every $i$, either $f_i=\wtl f_i$, or one among $f_i,\wtl f_i$ is the $p$-depletion of the other. In all that follows, we implicitly consider families up to this identification.
\end{defin}

The reason for the last sentence in Definition \ref{famdef} is that $p$-depleting an eigenform $f$, i.e. twisting $f$ with the trivial Dirichlet character of conductor $p$, does not affect either the Hecke eigensystem of $f$ away from $Np$ or its Frobenius eigenvalues.

Given that $p$ is fixed, we often omit the word ``$p$-adic'' and simply speak of a ``(weak) family of eigenforms''. For convenience, we sometimes index families by subsets of $\N$, in which case we specify it. When we write $(f_i,\varphi_{i})_i$, we are implicitly taking the index set to be $\N$. One could even remove the datum of the Frobenius eigenvalues from the notation of a weak family, but keeping it makes the notation more uniform.


Since the property of being a (weak) family is defined by what happens for $i\to\infty$, one can always obtain a new family by replacing a finite number of elements of a given family by randomly chosen eigenforms. For this reason, the only relevant properties of a family are those that affect all but a finite number of its elements: we will say that a family $\cF$ has \emph{eventually} property $\bP$ (or ``is eventually $\bP$'') if all but a finite number of its elements have property $\bP$ (respectively, if it is $\bP$ after removing a finite number of its elements). For instance, a family is eventually of finite slope if all except a finite number of its elements are eigenforms of finite slope.

\begin{defin}
We say that a (weak) family of eigenforms $(f_i,\varphi_{i})_i$ is \emph{of finite slope} if $f_i$ is of finite slope for every $i$.
\end{defin}

\begin{ex}\label{exEQ}
Let $\cE_\Q$ be the Coleman--Mazur eigencurve of tame level $N$, and let $(x_i)_{i\in\N}$ be a sequence of points of $\cE_\Q$ converging to a point $x_\infty$ of $\cE_\Q$. For every $i$, let $(f_i,\varphi_i)$ be the refined eigenform attached to $x_i$. Then the sequence $(f_i,\varphi_i)_i$ is obviously a family of eigenforms of finite slope, with the Hecke eigensystems of the $x_i$ away from $Np$ and the Frobenius eigenvalues $\varphi_i$ converging to the corresponding data for $x_\infty$.

Conversely, if $(f_i,\varphi_i)_i$ is a family of eigenforms of finite slope, the corresponding sequence of points $x_i\subset\cE_\Q(\Qp)$ converges to a limit point $x_\infty$. Indeed, the associated pseudorepresentations converge to a limit pseudorepresentation $t_\infty$, so that the pairs $(t_i,\varphi_i)_i$ converge to $(t_\infty,\varphi_\infty)$ in $\fR_{\ovl t}\times_\Q\G_m$, and $(t_\infty,\varphi_\infty)$ belongs to $\cE_\Q$ since $\cE_\Q$ is closed in $\fR_{\ovl t}\times_\Q\G_m$.
\end{ex}


%

\begin{ex}\label{excolste}
Let $(f_i,\varphi_i)_i$ be a family of $p$-old eigenforms of finite slope, e.g. a family extracted from the eigencurve as in Example \ref{exEQ}. Since every $f_i$ is of finite slope, the $q$-expansions of the eigenforms $f_i$ converge to the $q$-expansion of a Serre $p$-adic (not necessarily classical) eigenform $f_\infty$, whose eigensystem away from $Np$ is $\theta_\infty$, and $U_p$-eigenvalue is $\varphi_\infty$ (we prove below in Proposition \ref{famlim} that every family of finite slope can be extracted from the eigencurve, so that $f_\infty$ is always an overconvergent eigenform). Note that the sequence of weights $(k_i)_i$ has to diverge to $\infty$, since there is only a finite number of choices for eigenforms of fixed weight and level $\Gamma_1(N)\cap\Gamma_0(p)$. For every $i$, let $\wtl f_i$ be the twin of $f$; it is a $p$-old eigenform with the same Hecke eigensystem as $f_i$ away from $Np$, and it corresponds to the second choice $\varphi_i$ of admissible Frobenius eigenvalue for the crystalline pseudorepresentation $t_i$. The sequence of valuations $(v_p(\varphi_i))_i$ is eventually constant, equal to $v_p(\varphi_\infty)$, so that $v_p(\wtl\varphi_i)=k_i-1-v_p(\varphi_i)=k_i-1-v_p(\varphi_\infty)$ for large enough $i$. Therefore, the sequence $(\wtl\varphi_i)_i$ converges to 0, so that $(\wtl f_i,\wtl\varphi_i)_i$ is not a family of eigenforms, only a weak one. However, the $q$-expansions of the eigenforms $\wtl f_i$ still converge to a limit $q$-expansions, that of the $p$-depletion of $f_\infty$.

The weak family $(\wtl f_i,\wtl\varphi_i)_i$ appears as a special case of a construction by Coleman and Stein, who in \cite[Theorem 2.1]{colste} construct a sequence of eigenforms $(g_i)_i$, of finite slope and fixed tame level $N$, whose $q$-expansions converge to that of an eigenform $g_\infty$ of infinite slope, obtained by twisting an eigenform of finite slope with a power of the Teichm\"uller character $(\Z/p\Z)^\times\to\Z_p^\times$. If $\varphi_i$ is the $U_p$-eigenvalue of $g_i$ (that is, its unique admissible Frobenius eigenvalue), then the sequence $(g_i,\varphi_i)_i$ is a weak family of eigenforms, since the Hecke eigensystems of the forms $g_i$ away from $Np$ converge to that of $g_\infty$, but it is not a family, since the sequence $(\varphi_i)_i$ converges to 0.
\end{ex}




\begin{rem}\label{weakps}
	The following is an immediate consequence of Lemma \ref{hglim}. 
	Let $(f_i,\varphi_{i})_i$ be any sequence of eigenforms of tame level $\Gamma_1(N)$. Then $(f_i,\varphi_{i})_i$ is a weak family of eigenforms if and only if the associated sequence of pseudorepresentations $(t_i)_i$ converges to a pseudorepresentation $t_\infty$ in the metric of Definition \ref{psdist}.
\end{rem}

For every $i\in\N$, let $\kappa_i=(k_i,\chi_i)$ be the weight-nebentypus of $f_i$, seen as a point of the weight space $\cW_\Q$. Recall from Remark \ref{cloclass} that the weights in the $p$-adic closure of the classical weights $(k,\chi)$ are still of the form $(k,\chi)$, with $k$ now in $\Z_p$ and $\chi$ a finite order character.

\begin{lemma}
The sequence $(k_i,\chi_i)_i$ converges to a limit $(k_\infty,\chi_\infty)$ in $\cW_\Q$, with $k_\infty\in\Z_p$ and $\chi_\infty\colon\Z_p^\times\to\C_p^\times$ a finite order character. In particular, the conductors of the characters $(\chi_i)_i$ are all bounded by a common constant.
\end{lemma}

\begin{proof}
By Remark \ref{weakps}, the sequence $(t_i)_i$ converges to a pseudorepresentation $t_\infty$ in the metric of Definition \ref{psdist}. Clearly the sequence of determinants $(d_{t_i})_i$ converges to $d_{t_\infty}$. If $\chi_i^{(p)}$ is the prime-to-$p$ part of the nebentypus of $f_i$, then $d_{t_i}=\kappa_i\chi_i^{(p)}$, with the usual identification of the weight-nebentypus $\kappa_i$ with a Galois character via local class field theory. The characters $\chi_i^{(p)}$ are bounded in conductor by $N$, hence they coincide with a same character $\chi_\infty^{(p)}$ for $i$ large enough. Therefore, the sequence $(\kappa_i)_i$ converges to $\kappa_\infty\coloneqq d_{t_\infty}(\chi_\infty^{(p)})^{-1}$.

By Remark \ref{cloclass}, the limit weight $\kappa_\infty$ has to be of the form $(k_\infty,\chi_\infty)$, and since the nebentypus part of the weight is locally constant on weight space, $\chi_i=\chi_\infty$ for large enough $i$. The second statement of the lemma follows.
\end{proof}


\subsubsection{Base changing families of eigenforms}

Even though we do not need to introduce a notion of families of Hilbert eigenforms analogue to that of Definition \ref{famdef}, we will need the following lemma when base changing families of $\GL_{2/\Q}$-eigenforms to real quadratic fields. 

Let $F$ be a real quadratic field in which $p$ is either ramified or inert, and let $\fp$ be the unique $p$-adic place of $F$. 
Let $(f_i)_i$ be a sequence of $\GL_{2/F}$-eigenforms of tame level $N$ and parallel weight. To every $f_i$, we attach a continuous $p$-adic representation of $G_F$, potentially semistable at $\fp$, following Carayol, Blasius, Rogawski and Taylor. For every $i$, let $\varphi_i$ be an admissible Frobenius eigenvalue of $f_i$. 

Even if we only make very limited use of it, we introduce a notion of family of Hilbert eigenforms, completely analogous to that given in Definition \ref{famdef}.

\begin{defin}
	We say that the sequence $(f_i,\varphi_{i})_{i\in\N}$ is a \emph{$p$-adic family of $\GL_{2/F}$-eigenforms} if the forms $f_i$ are pairwise distinct and:
	\begin{itemize}
		\item[(i)] the Hecke eigensystems $\theta_i\colon\calH_F^{Np}\to\C_p$ of the eigenforms $f_i$ away from $Np$ converge to a limit homomorphism $\theta_\infty\colon\calH_F^{Np}\to\C_p$. 
		\item[(ii)] the sequence $(\varphi_{i})_i$ converges to a limit $\varphi_\infty\in\C_p^\times$. 
	\end{itemize}
We identify two families $(f_i,\varphi_{i})_{i\in\N}$ and $(\wtl f_i,\wtl\varphi_{i})_{i\in\N}$ if, for every $i$, either $f_i=\wtl f_i$, or one among $f_i,\wtl f_i$ is a twist of the other with the trivial Dirichlet character of conductor $\fp$.
\end{defin}

\begin{ex}\label{exEF}
	Let $\cE_F$ be the parallel weight $\GL_{2/F}$-eigenvariety of tame level $N$, and let $(x_i)_{i\in\N}$ be a sequence of points of $\cE_F$ converging to a point $x_\infty$ of $\cE_F$. For every $i$, let $(f_i,\varphi_i)$ be the refined eigenform attached to $x_i$. Then the sequence $(f_i,\varphi_i)_i$ is obviously a family of eigenforms (of finite slope), with the Hecke eigensystems of the $x_i$ away from $Np$ and the Frobenius eigenvalues $\varphi_i$ converging to the corresponding data for $x_\infty$.
	
	Conversely, in exactly the same way as in Example \ref{exEQ}, we show that the sequence of points $x_i\in\cE_F$ attached to a family $(f_i,\varphi_i)_i$ of eigenforms of finite slope converges to a point $x_\infty\in\cE_F$.
\end{ex}

Let $(f_i,\varphi_i)_i$ be a family of $\GL_{2/\Q}$ eigenforms, with associated family of Hecke eigensystems $(\theta_i)_i$ away from $Np$. For every $i$, let $f_{i,F}$ be the base change of $f_i$ to $\GL_{2/F}$, and let $\varphi_{i,F}=\varphi_i$ if $p$ is ramified in $F$, and $\varphi_{i,F}=\varphi_i^2$ if $p$ is inert in $F$.

\begin{lemma}\label{bcfam}
For every $i$, $(f_{i,F},\varphi_{i,F})$ is a refined $\GL_{2/F}$-eigenform, and the sequence $(f_i,\varphi_{i,F})_{i\in I}$ is a family of $\GL_{2/F}$-eigenforms. 
\end{lemma}

\begin{proof}
The first statement follows from the fact that the $(\varphi,N)$-module $D_{\pst}(\rho_{f_i,\fp})$ is obtained by base changing $D_\pst(\rho_{f,p})$.

For every $i$, the Hecke eigensystem $\theta_{i,F}$ of $f_{i,F}$ away from $Np$ is obtained from $\theta_i$ via composition with the map $\BC_{F/\Q}\colon\calH_F^{Np}\to\calH_\Q^{Np}$ attached to the base change in Remark \ref{bcmap}. In particular, if the sequence $\theta_i$ converges to $\theta_\infty$, then $\theta_{i,F}=\BC_{F/\Q}\ccirc\theta_i$ converges to $\BC_{F/\Q}\ccirc\theta_\infty$. The convergence of $(\varphi_{i,F})_i$, to either $\varphi_\infty$ or $\varphi_\infty^2$, is obvious.
\end{proof}

\subsubsection{$p$-primitive families}


Let $F$, $\fp$ be as above, and let $(f_i,\varphi_i)_i$ be a family of $\GL_{2/F}$-eigenforms. 
Assume that every $f_i$ is the twist of a $\GL_{2/F}$-eigenform $f_i^0$ of finite slope with a finite order Hecke character $\delta_i$ of $F$ whose conductor divides $\fp^\infty$. Write $t_i$, respectively $t_i^0$, for the continuous pseudorepresentation of $G_F$ attached to $f_i$, respectively $f_i^0$. Recall that the eigencurve $\cE_F$ is defined over an arbitrary $p$-adic field $K$ containing the images of all the embeddings of $F$ into $\Qp$; in particular, because of our choice of $F$, one can take $K=F_\fp$ in the following.

\begin{lemma}\label{untwistfam}
	The sequence $(f_i^0,\varphi_i)_{i}$ is a ($p$-primitive) family, and the characters $\delta_i$ all coincide for large enough $i$. 
\end{lemma}

\begin{proof}
	Fix a continuous pseudorepresentation $\ovl t_F\colon G_F\to\GL_2(\Fp)$. Let $X$ be the set of continuous characters $\ovl\chi\colon G_F\to\Fp^\times$ factoring through $G_F\to G_{F(\zeta_p)}$. For every $\chi\in X$, embed $\cE_{F,\ovl{\chi t}}$ as a closed subspace of $\G_{m,F}\times_K\fR_{\ovl{\chi t}}$ as in Proposition \ref{EFclosure}, hence the disjoint union $\coprod_{\ovl\chi\in X}\cE_{F,\ovl{\chi t}}$ as a closed subspace of $\G_{m,F}\times_K\coprod_{\chi\in X}\fR_{\ovl\chi\ovl t}$. 
	
	The disjoint union of the universal deformation spaces of the characters $\ovl\chi\in X$ is isomorphic to $\wtl\cW_F^\ast=\Spf\wtl\Lambda_F^\ast$, with $\wtl\Lambda_F^\ast=\cO_K[[(\cO_{F,p}^\times)]]$. We attach to the universal character on $\wtl\cW_F^\ast$ a character $F^\times\to\wtl\Lambda_F^{\ast,\times}$ by mapping a uniformizer to 1, and a character $G_{F_\fp}\to\wtl\Lambda_F^{\ast,\times}$ via the local reciprocity map. Then $\wtl\cW_F^\ast$ acts on $\coprod_{\ovl\chi\in X}\fR_{\ovl{\chi t}}$ via character multiplication, giving a rigid analytic map
	\[ h\colon\wtl\cW_F^\ast\times_K\coprod_{\ovl\chi\in X}\fR_{\ovl{\chi t}}\to\coprod_{\ovl\chi\in X}\fR_{\ovl{\chi t}}. \]
	Let $\cE_{F,\ovl t}^\tw$ be the image of $\wtl\cW_F^\ast\times_K\coprod_{\ovl\chi\in X}\cE_{F,\ovl{\chi t}}$ under the morphism 
	\[ \id_{\G_{m,F}}\times h\colon\G_{m,F}\times_K\coprod_{\ovl\chi\in X}\fR_{\ovl{\chi t}}\times_F\wtl\cW_F\to\G_{m,F}\times_K\coprod_{\ovl\chi\in X}\fR_{\ovl{\chi t}}; \]
	it is a closed subspace of $\G_{m,F}\times_K\coprod_{\ovl\chi\in X}\fR_{\ovl{\chi t}}$. 
	The points of the ``eigensurface'' $\cE_{F,\ovl t}^\tw$ parameterize the ``wild twists'' of the eigencurve: the natural embedding $\cE_{F,\ovl t}^\tw\subset\G_{m,F}\times_K\coprod_{\ovl\chi\in X}\fR_{\ovl{\chi t}}$ identifies points of $\cE_{F,\ovl t}^\tw$ with pairs $(\varphi,t)$ where $t$ can be written as 
	\begin{equation}\label{chit0}
		t=\delta t^0
	\end{equation}
	for a continuous character $\delta\colon G_F\to\C_p^\times$, a continuous representations $t^0$ attached to an overconvergent $\GL_{2/F}$-eigenform $f^0$ of finite slope, and $\varphi$ is an admissible Frobenius eigenvalue of $f^0$. The ``classical points'' of $\cE^\tw_{F,\ovl t}$ correspond to twists of overconvergent eigenforms of finite slope with Dirichlet characters of finite order and conductor dividing $p^\infty$, and are dense in $\cE_{F,\ovl t}^\tw$.
	
	For every $i$, the pair $x_i\coloneqq(\varphi_i,t_{i})$ defines a point of $\cE_{F,\ovl t}^\tw$, that by definition converges to the limit $x_\infty\coloneqq(\varphi_\infty,t_\infty)$ in the $p$-adic topology of $\G_{m,F}\times_K\coprod_{\ovl\chi\in X}\fR_{\ovl{\chi t}}$. Since $\cE_{F,\ovl t}^\tw$ is Zariski closed in $\G_{m,F}\times_K\coprod_{\ovl\chi\in X}\fR_{\ovl{\chi t}}$, it is in particular $p$-adically closed, so the point $x_\infty$ belongs to $\cE_{F,\ovl t}^\tw$. Therefore, we can write $t_\infty$ as $\delta t_\infty^0$ as in \eqref{chit0}. In particular, $t_\infty^0$ is attached to an overconvergent eigenform $f_\infty^0$ of finite slope, so that $x_\infty^0\coloneqq(\varphi_\infty,t_\infty^0)$ is a point of $\cE_{F,\ovl t}$.
	
	Since $h$ induces an isomorphism $\wtl\cW_F^\ast\times_K\cE_{F,\ovl t}\cong\cE_{F,\ovl t}^\tw$, the universal character on $\wtl\cW_F^\ast$ induces a continuous character $\delta\colon \Z_p^\times\to\cO_{\cE_{F,\ovl t}^\tw}(\cE_{F,\ovl t}^\tw)^\times$, with the property that $\delta(x_i)=\delta_i$ for every $i\in\N\cup\{\infty\}$. By continuity of $\delta$, $\delta_i$ converges to $\delta_\infty$ in $\cW_F$, and since $t_i$ converges to $t_\infty$, we conclude that $t_i^0$ converges to $t_\infty^0$ in $\coprod_{\ovl\chi\in X}\fR_{\ovl{\chi t}}$. Therefore, $(\varphi_i,t_i^0)$ converges to a point $x_\infty^0$ of $\cE_{F,\ovl t}$, which implies that $(f_i^0,\varphi_i)$ is a family of eigenforms.

Since the characters $\delta_i$ converge to $\delta$ in $\wtl\cW_F^\ast$, their order is bounded in terms of the distance of $\delta$ from the trivial character. In particular, there is a finite number of choices for the characters $\delta_i$, so that they must all coincide for large enough $i$.
\end{proof}

\begin{rem}
In Lemma \ref{untwistfam}, we are using the fact that $(f_i,\varphi_i)_i$ is a family: the analogous statement for weak families is false. One can build a counterexample from Example \ref{excolste}: with the notation of the example, consider the family $(\wtl f_i,\wtl\varphi_i)_i$ of finite slope eigenforms converging to the $p$-depletion of $f_\infty$, i.e. the twist of $f_\infty$ by the 0-th power of the Teichm\"uller character $\omega\colon(\Z/p\Z)^\times\to\Z_p^\times$. Then for any $j\in\Z, (\omega^j \wtl f_i,\wtl\varphi_i)_i$ is again a weak family of eigenforms, but not a family (since the Frobenius eigenvalues are unchanged by the twist, hence still converge to 0). Any ``mix'' of $(\wtl f_i,\wtl\varphi_i)_i$ and its twists, for instance $(g_i,\varphi_i)_i=(\omega^i\wtl f_i,\wtl\varphi_i)_i$, is still a weak family, but does not satisfy the conclusion of Proposition \ref{untwistfam}.
\end{rem}

\subsection{Affinoid families of eigenforms}
We give a more geometric definition of family, inspired by the definitions of ``weakly refined'' and ``refined'' families in \cite[Definitions 4.2.3 and 4.2.7]{bellchen}. 

We denote by $\ev_x\colon\cO_U(U)\to L_x$ the evaluation map from the ring of regular functions on a rigid analytic space $U$ to the residue field $L_x$ at a point $x$. Given a $\GL_{2/\Q}$-eigenform $f$, we denote by $\rho_f\colon G_\Q\to\GL_2(V_f)$ the associated Galois representation, with $V_f$ a 2-dimensional $\Qp$-vector space. Recall that we gave in Definition \ref{dim2ref} a notion of admissibility for Frobenius eigenvalues.

\begin{defin}\label{afffamdef}
	A \emph{weak $p$-adic affinoid family of eigenforms} of tame level $N$, defined over a $p$-adic field $L$, is the datum $(U,\Theta,\Phi,S)$ of
	\begin{enumerate}[label=(\roman*)]
		\item an integral affinoid space $U$ over $L$,
		\item a homomorphism $\Theta\colon\calH_\Q^{Np}\to\cO_U(U)$, 
		\item a nowhere-vanishing element $\Phi\in\cO_U(U)$, and
		\item a dense subset $S\subset U(\C_p)$
	\end{enumerate}
	such that, for every $x\in S$, $\ev_x\ccirc\Theta$ is the Hecke eigensystem away from $Np$ attached to an eigenform $f_x$ of level $Np^r$, for some $r\ge 1$, and $\Phi(x)$ is an admissible Frobenius eigenvalue of $V_{f_x}\vert_{G_{\Q_p}}$. 

An \emph{affinoid $p$-adic family of eigenforms} is a weak $p$-adic family of eigenforms $(U,\Theta,\Phi,S)$ that satisfies:
	\begin{enumerate}
	\item[\optionaldesc{(Aff1)}{aff1}] there exists a continuous pseudorepresentation $t_U\colon G_{\Q,Np}\to\cO_U^\circ(U)$ such that $t_U(\Frob_\ell)=\Theta(T_\ell)$ for every prime $\ell\nmid Np$, 
	\item[\optionaldesc{(Aff2)}{aff2}] the map $\pi^\ps_\cF\colon U\to\fR_{\ovl t_U}$ deduced from \eqref{aff1} and the universal property of $\fR_{\ovl t}$ is finite onto its image. 
	\end{enumerate}

The \emph{dimension} of a (weak) affinoid family $(U,\Theta,\Phi,S)$ is the dimension of $U$.
\end{defin}

\noindent The residual pseudorepresentation $\ovl t_U$ appearing in condition \ref{aff2} is the one given by Definition \ref{affres}.
As before, we typically omit the words ``$p$-adic'' and simply speak of a ``(weak) affinoid family of eigenforms''. 

\begin{rem}\label{weakref}
Remark \ref{admfrob} gives a bijection between admissible Frobenius eigenvalues of $V_f\vert_{G_{\Q_p}}$ and $N$-stable refinements of $V_f\vert_{G_{\Q_p}}$. Hence, a weak affinoid family of eigenforms interpolates not just a set of eigenforms, but also a chosen refinement for each of them. 
In this sense, the definition of a (weak) affinoid family of eigenforms resembles the definition of (weakly) refined family of $p$-adic representations given in \cite[Definitions 4.2.3 and 4.2.7]{bellchen}, with the important difference that we want to interpolate potentially semistable representations, not just crystalline ones. 
\end{rem}



\begin{rem}
By \cite{diaoliu}, the eigencurve satisfies the valuative criterion of properness, i.e. if one extracts from the eigencurve a punctured affinoid equipped with data as in (i-iv) of Definition \ref{afffamdef}, one can always extend it to a weak family defined over the whole affinoid. It might be possible to show, along the same lines, that the condition that $\Phi$ is nonvanishing as in Definition \ref{afffamdef}(iii) follows from the others.
\end{rem}

\begin{rem}\label{frobUp}
If all of the forms $f_x$, $x\in S$, are of finite slope, then in light of Example \ref{eigenpar}, interpolating their refinements amounts to interpolating their Hecke eigenvalues at $p$, as is customary when constructing $p$-adic families of eigenforms of finite slope. 
\end{rem}

In the case when all of the $f_x$ are of infinite slope, the Hecke eigenvalues at $p$ are trivially interpolated by the zero function. In the following remark, we stress why it is important in this case to include a nontrivial interpolation condition at $p$ in order to obtain a meaningful notion of family of eigenforms of infinite slope.

\begin{rem}\label{emepas}
	If one does not require that the refinements be interpolated by an analytic function, then the datum of a family only amounts to that of a family of Galois representations together with a dense set of modular points. One can produce such families via the known results on the density of modular points in Galois deformation spaces, and show that our classification results do not hold for such families. One can even interpolate $p$-supercuspidal eigenforms in this Galois-theoretic sense: as far as local representations are concerned, Emerton and Pa\v{s}k\={u}nas \cite{emepassc} show that, for a residual representation $\ovl\rho$ of $G_{\Q_p}$ admitting a crystalline, potentially diagonalisable lift of regular weight (an assumption one can now remove by the work of Emerton--Gee), the closure of the set of points corresponding to $p$-supercuspidal eigenforms of any fixed regular weight contains some irreducible components of the universal framed deformation space $\fR_{\ovl\rho}^\square$ of $\ovl\rho$. If one had a global analogue of such a statement, one could use it to produce a family of Galois representations containing a dense set of modular, $p$-supercuspidal points, but not necessarily to complete this datum with an interpolation of the corresponding Frobenius eigenvalues to obtain a weak family of eigenforms in our sense. 
\end{rem}

In the following Remarks, we comment conditions \ref{aff1},\ref{aff2} of Definition \ref{afffamdef}. As we recall below in Example \ref{hidacol}, they are satisfied by the ``usual'' affinoid families of eigenforms of finite slope. 

\begin{rem}
	Any $t_U$ satisfying condition \ref{aff1} in Definition \ref{afffamdef} will have the property that, for every $x\in S$, $\ev_x\ccirc t_U$ is the pseudorepresentation $t_x$ attached to the eigenform $f_x$.
	
	Moreover, if a continuous pseudorepresentation $t_U\colon G_{\Q,Np}\to\cO_U(U)$ is given, such that $\ev_x\ccirc t_U=t_x$ for all $x$ in a subset $S^\prime$ of $S$ that is still dense in $U$, then a simple density argument allows us to conclude that $t_U$ has property \ref{aff1}.
\end{rem}

\begin{rem}
Condition \ref{aff2} of Definition \ref{afffamdef} rules out pathological families such as:
\begin{itemize}
\item constant affinoid families, i.e. families where $\Theta$ factors via $L\subset\cO_U(U)$ and the Hecke eigensystem $\calH_\Q^{Np}\to L$ of an eigenform $f$, so that the specialization $\ev_x\ccirc\Theta$ is the eigensystem of $f$ independently of $x\in U(\C_p)$; 
\item families obtained by translating a single one, as follows: let $\cF$ be any (weak) affinoid family of eigenforms, supported on an affinoid $U$, and let $V$ be any integral $L$-affinoid; then we can define a family $\cF\times_LV$ by composing $\Theta$ with the map $\cO_U(U)\to\cO_{U\times_LV}(U\times_LV)$ induced by the projection $U\times_LV\to U$. Constant families can be seen (abusing definitions) as the special case when $U=\Spm L$. 
\end{itemize}
\end{rem}

Definition \ref{afffamdef} is modeled on the existing notions of families for eigenforms of finite slope. In particular, we have the following.

\begin{ex}\label{hidacol}\mbox{ }
	\begin{itemize}
		\item If $\cE_\Q$ is the eigencurve of tame level $\Gamma_1(N)$ constructed by Coleman--Mazur, Buzzard and Chenevier, and $x_0$ is any point of $\cE_\Q$ of weight $\psi x^{k_0}$ with $k_0\in\Z_p$, then we can construct an affinoid family starting with any irreducible affinoid neighborhood of $x_0$. Indeed, the slope is constant, say equal to $h$, on a sufficiently small affinoid neighborhood $U_0$ of $x_0$, and since $k_0\in\Z_p$, $U_0$ contains an infinite set of points of weight $\psi x^{k_0}$ with $k>h+1$. By Coleman's classicality criterion, all such points are classical. Recall that $\cE_\Q$ is 1-dimensional and comes equipped with a Hecke eigensystem $\Theta^\full_{\cE_\Q}\colon\calH_\Q\to\cO_\cE(\cE)$, that (by definition) specializes to the Hecke eigensystem of a classical eigenform of level $\Gamma_1(Np^r)$, for some $r\ge 1$, at every classical point. 
		Therefore, we can define an affinoid family of eigenforms as the datum of:
		\begin{itemize}
			\item an integral affinoid neighborhood $U$ of $x_0$,
			\item the homomorphism $\Theta\colon\calH_\Q^{Np}\to\cO_U(U)$ induced by the restriction of $\Theta^\full_{\cE_\Q}$ to $\calH_\Q^{Np}$,
			\item as $\Phi$, the image of the operator $U_p$ under $\Theta^\full_\cE$, restricted to $U$ (using Example \ref{eigenpar}),
			\item any infinite set of classical points of $U$, which will be dense since $U$ is 1-dimensional.
		\end{itemize}
		For a given residual pseudorepresentation $\ovl t$, Coleman and Mazur construct a morphism from the $\ovl t$-part $\cE_{\ovl t}$ of the eigencurve to $\fR_{\ovl t}\times_L\G_m$. The composition of such a morphism with the projection to $\fR_{\ovl t}$ is not well-behaved (its image is Mazur's ``infinite fern''), but its restriction to an integral affinoid subdomain $U$ of $\cE_{\ovl t}$ gives a closed embedding $U\into\fR_{\ovl t}$, that equips $U$ with a pseudorepresentation $t_U$. Therefore, conditions \ref{aff1} and \ref{aff2} of Definition \ref{afffamdef} hold for the affinoid families supported by the eigencurve. 
		\item With the notations of the previous point, when the slope of $x_0$ is 0 we can mimic the construction above starting with the rigid analytic generic fiber of $\Spf\T^\ord$, where $\T^\ord$ is the big ordinary Hecke algebra of tame level $\Gamma_1(N)$ constructed by Hida; we can apply Hida's control theorem for ordinary forms of weight at least 2, instead of Coleman's classicality criterion. Since the rigid analytic generic fiber of $\T^\ord$ coincides with the union of the slope 0 connected components of the Coleman--Mazur eigencurve, this construction provides us with the families of slope 0 already constructed in the previous point.
	\end{itemize}
	Conditions \ref{aff1} and \ref{aff2} of Definition \ref{afffamdef} follow either from the properties of the Coleman--Mazur eigencurve recalled above, or directly from Hida's original constructions. 
\end{ex}

Even if it is convenient to keep it in the notation, that data of the family $S$ is somewhat auxiliary. For this reason, we give the following definition.

\begin{defin}
We say that two (weak) families of eigenforms $(U,\Theta,\Phi,S)$ and $(\wtl U,\wtl\Theta,\wtl\Phi,\wtl S)$ are \emph{equivalent} if there exists an isomorphism of $L$-affinoids $f\colon U\cong\wtl U$ such that $\wtl\Theta=f^\ast\Theta$ and $\wtl \Phi=f^\ast \Phi$.
\end{defin}

If $\cF=(U,\Theta,\Phi,S)$ is a weak family of eigenforms and $U_0$ is an affinoid subdomain of $U$ such that $S\cap U_0(\C_p)$ is dense in $U_0$, we write
\[ \cF\vert_{U_0}=(U_0,\Theta\vert_{U_0},\Phi\vert_{U_0},S\cap U_0), \]
where $\Theta\vert_{U_0}$ is the composition of $\Theta$ with the restriction map $\cO_U(U)\to\cO_{U_0}(U_0)$. 
Clearly $\cF\vert_{U_0}$ is a weak affinoid family of eigenforms, that we call the \emph{restriction of $\cF$ to $U_0$}. If $\cF$ is an affinoid family, then $\cF\vert_{U_0}$ is an affinoid family.


Let $\cF=(U,\Theta,\Phi,S)$ be an affinoid family of eigenforms. Let $t_U$ be the associated pseudorepresentation and $\pi_\cF^\ps\colon U\to\fR_{\ovl t}$ the map given by the universal property of $\fR_{\ovl t}$. For every $x\in S$, the specialization $t_x=\ev_x\ccirc t_U$ is attached to an eigenform, hence the image of $S$ under $\pi_\cF^\ps$ lands in the Sen-null subspace $\fR_{\ovl t}^0$. Since $S$ is dense in $U$, the image of $\pi_\cF^\ps$ is also contained in $\fR_{\ovl t}^0$. 

\begin{rem}\label{kappa}
If $s_\sigma$ is the analytic function on $\fR_{\ovl t}$ produced by Lemma \ref{weightan}, the pair $(0,(\pi_\cF^\ps)^\ast s_\sigma)$ describes the Hodge--Tate--Sen weights of $t_U$: for every $x\in U(\Qp)$, the Hodge--Tate--Sen weights of $t_x$ are 0 and $s_\sigma(\pi_\cF^\ps(x))$. 
In the following, we denote by $\kappa_\cF$ the rigid analytic function $(\pi_\cF^\ps)^\ast s_\sigma$ on $U$.
\end{rem}

\subsection{Pseudorepresentations along weak affinoid families}
Let $\cF=(U,\Theta,\Phi,S)$ be a weak affinoid family of eigenforms. We show that under a mild assumption on $S$, and up to replacing $U$ with a smaller affinoid, we can equip $\cF$ with a continuous pseudorepresentation with the interpolation property required in condition \ref{aff1} of Definition \ref{afffamdef}. 

We start by recalling an argument of Chenevier allowing one to interpolate pseudorepresentations along points of a wide open rigid space. 
Let $X$ be a rigid analytic space over $L$, and let $S$ be a Zariski-dense set of $\C_p$-points of $X$. For every $x\in S$, let $t_x\colon G_\Q\to\C_p$ be a continuous pseudorepresentation. Assume that, for all except a finite number of primes $\ell$, $t_x$ is unramified at $\ell$ for every $x\in S$, and that there exists a rigid analytic function $F_\ell\in\cO_X^\circ(X)$ such that $F_\ell(x)=t_x(\Frob_\ell)$ for every $x\in S$ and an arbitrary lift $\Frob_\ell\in G_\Q$ of the Frobenius at $\ell$.

\begin{prop}[{cf. \cite[Proposition 7.1.1]{chenfam}}]\label{chenarg}
	If $X$ is wide open, there exists a continuous pseudorepresentation $t_X\colon G_\Q\to\cO_X^\circ(X)$ such that $\ev_x\ccirc t_X=t_x$ for every $x\in S$. 
\end{prop}

\begin{proof}
	Consider the map
	\[ \prod_{x\in S}t_x\colon G_\Q\to\prod_{x\in S}\C_p. \]
	An easy check shows that $\prod_{x\in S}t_x$ is a pseudorepresentation. Since $S$ is Zariski-dense in $X$, the product of the evaluation maps $\ev_x\colon\cO_{X}^\circ(X)\to\C_p$ at $x\in S$ gives an injective ring homomorphism $\prod_{x\in S}\ev_x\colon\cO_{X}^\circ(X)\to\prod_{x\in S}\C_p$. The ring $\cO_{X}^\circ(X)$ is compact by Lemma \ref{nestedlemma}, hence its image under $\prod_{x\in X^\cl}\ev_x$ is closed since $\prod_{x\in S}\C_p$ is separated. On the other hand, $t_x(\Frob_\ell)=F_\ell(x)$ for every prime $\ell\nmid Np$ and lift $\Frob_\ell$ of the Frobenius at $\ell$, hence
	\[ \Big(\prod_{x\in S}t_x\Big)\big(\Frob_\ell\big)=\Big(\prod_{x\in S}\ev_x\Big)\big(F_\ell\big), \]
	so that by Chebotarev's theorem 
	the image of $G_{\Q}$ under $\prod_{x\in S}t_x$ is contained in the image of $\prod_{x\in S_0}\ev_x$. Therefore $\prod_{x\in S}t_x$ factors through a pseudorepresentation $t_{X}\colon G_\Q\to\cO^\circ_{X}(X)$, that clearly has the desired interpolation property.
\end{proof}

\begin{cor}\label{chenargaff}
	If $X$ is affinoid and $X_0$ is an affinoid subdomain of $X$, nested in $X$ and such that $S\cap X_0(\C_p)$ is dense in $X_0$, then there exists a continuous pseudorepresentation $t_{X_0}\colon G_\Q\to\cO_X^\circ(X_0)$ such that $\ev_x\ccirc t_{X_0}=t_x$ for every $x\in S\cap X_0(\C_p)$. 
\end{cor}

\begin{proof}
	It is enough to apply Proposition \ref{chenarg} to any wide open subspace $X_1\subset X$ such that $X_0\subset X_1$, and compose the resulting pseudorepresentation with the restriction map $\cO_X(X_1)\to\cO_X(X_0)$.
\end{proof}



%

Finally, we construct a pseudorepresentation along a suitable restriction of a weak affinoid family.

\begin{prop}\label{galfam}\mbox{ }
	Let $\cF=(U,\alpha,\varphi_1,S)$ be a weak affinoid family of eigenforms, and let $U_0$ be a nested affinoid subdomain of $U_1$ such that $S\cap U_0(\C_p)$ is dense in $U_0$. Then there exists a continuous, 2-dimensional pseudorepresentation $t_{U_0}\colon G_{\Q,Np}\to\cO_{U_0}^\circ(U_0)$ such that, for every $x\in S\cap U_0$, $\ev_x\ccirc t_{U_0}$ is the pseudorepresentation attached to $f_x$. 
\end{prop}



\begin{proof}
	%
	%
	Since $S\cap U_0(\C_p)$ is dense in $U_0$, $S\cap U(\C_p)$ is dense in the wide open space $U$, so we are in a position to apply Proposition \ref{chenarg}, which provides us with a pseudorepresentation $t_U\colon G_\Q\to\cO_U^\circ(U)$. The restriction of $t_U$ to a pseudorepresentation $t_{U_0}\colon G_{\Q,Np}\to\cO_{U_0}^\circ(U_0)$ satisfies the required property.
\end{proof}

\begin{rem}
	Keep the notations of Proposition \ref{galfam}.
	\begin{enumerate}[label=(\roman*)]
		\item In general, the pseudorepresentation $t_{U_0}$ does not make $\cF\vert_{U_0}$ into an affinoid family of eigenforms: one still has to check condition \ref{aff2} of Definition \ref{afffamdef}. 
		\item The proof of Proposition \ref{galfam} fails if we work with the (irreducible) affinoid $U_0=U$, since then we are not able to find a wide open space $U_1$ lying between $U_0$ and $U$, and the ring of regular functions of norm bounded by 1 over $U$ is a Tate $\Z_p$-algebra, therefore not compact unless $U$ is a point.
		\item If $U$ is 1-dimensional and $x_0$ is an accumulation point of $S$ in $U$, then one can choose as $U_0$ an arbitrary affinoid neighborhood of $x_0$ strictly contained in $U$. This is the situation we will work with in most applications.
	\end{enumerate}
\end{rem}



\subsection{Specializations of affinoid families}

Let $\cF=(U,\Theta,\Phi,S)$ be a weak affinoid family of eigenforms. 

\begin{defin}
	If $I$ is any set, and $(f_i,\varphi_i)_{i\in I}$ a collection of refined eigenforms, then we say that $(f_i,\varphi_i)_{i\in I}$ is a \emph{specialization} of $\cF$ if there exists a subset $\{x_i\}_{i\in I}$ of $U(\Qp)$ such that, for every $i$, $\ev_{x_i}\ccirc\Theta$ is the Hecke eigensystem away from $Np$ attached to $f_i$, and $\Phi(x_i)=\varphi_i$.
	
	If $I$ is a singleton, we also refer to either $(f_i,\varphi_i)$ or $x_i$ as a \emph{classical specialization} of $\cF$. If $I$ is an infinite subset of $\N$, we say that $(f_i,\varphi_i)_{i\in I}$ is \emph{eventually} a specialization of $\cF$ if there exists $i_0\in\N$ such that $(f_i,\varphi_i)_{i\in I\cap\N_{\ge i_0}}$ is a specialization of $\cF$.
	
	We say that $\cF$ is \emph{of finite slope} if, for every classical specialization $x$ of $\cF$, $\ev_x\ccirc\Theta$ is the Hecke eigensystem outside of $Np$ of a classical eigenform of finite slope.
\end{defin}

For later use, we prove a lemma. Let $\cF=(U,\Theta,\Phi,S)$ be an affinoid family of eigenforms. For every classical specialization $x\in U(\Qp)$ of $\cF$, let $p^{n_x}$ be the conductor of the $p$-part of the nebentypus of $f_x$.

\begin{lemma}\label{nebenbound}
	The set $\{n_x\}_x$ is bounded as $x$ runs among the classical specializations of $\cF$.
\end{lemma}

\begin{proof}
	Consider the map $\det\colon\fR_{\ovl t}\to\cW_\Q$ introduced in Section \ref{detsec}. The image of $U$ under $\det\ccirc\pi_\cF^\ps$ is an affinoid subspace of a connected component of $\cW_\Q$. After identifying such a connected component with a wide open disc of center 0 and radius 1, $U$ is contained in a disc of center 0 and radius $p^{-r}$ for some $r\in\Q_{>0}$. Then a standard calculation shows that $n_x$ is bounded by the condition $v_p(\zeta_{n_x}-1)\ge r$, $\zeta_{n_x}$ an $n_x$-th root of 1.
\end{proof}

\subsubsection{Affinoid families of finite slope}

The classification of affinoid families of finite slope and their specializations comes down to the classical theory of Hida and Coleman families.

\begin{defin}\label{hidacolfam}
	An \emph{affinoid Hida (respectively, Coleman) family of eigenforms} is an affinoid family of eigenforms constructed from a Hida (respectively, Coleman) family as in Example \ref{hidacol}.
	
	A \emph{Hida (respectively, Coleman) family of eigenforms} is a family of eigenforms arising as a specialization of an affinoid Hida (respectively, Coleman) family. 
\end{defin}

\begin{prop}\label{fsequiv}\mbox{ }
	\begin{enumerate}
		\item Let $\cF=(U,\Theta,\Phi,S)$ be a 1-dimensional affinoid family of eigenforms. The following are equivalent:
		\begin{enumerate}
			\item $\cF$ is of finite slope,
			\item there exists a dense subset $S^\prime$ of $U$ such that $\ev_x\ccirc\Theta$ is attached to an eigenform of finite slope for every $x\in S^\prime$,
			\item $\cF$ is an affinoid Hida or Coleman family.
		\end{enumerate}
		\item Let $(f_i,\varphi_i)_i$ be a family of eigenforms. Then $(f_i)_i$ is of finite slope if and only if it is eventually a union of specializations of Hida or Coleman families.
	\end{enumerate}
\end{prop}

Statement (ii) is false if $(f_i,\varphi_i)_i$ is only a weak family of eigenforms. See Example \ref{excolste}.

\begin{proof}
	The implications (i)$\implies$(ii) and (iii)$\implies$(i) are trivial. We prove that (ii)$\implies$(iii). By Proposition \ref{extract}, there exists an extension $F/\Q$, either trivial or real quadratic, such that the base change of $\cF$ to $F$ is supported by an affinoid subdomain of the parallel weight eigenvariety $\cE_F$. The choice of the field in the proof of Proposition \ref{extract} shows that we can take $F=\Q$, so that $\cF$ is the family carried by an affinoid domain of the $\GL_{2/\Q}$-eigencurve, i.e. an affinoid Hida or Coleman family.
	
	The ``if'' of (ii) is obvious. We prove the ``only if''. By Proposition \ref{famlim}, the family of eigenforms $(f_i,\varphi_i)_i$ is eventually a union of specializations of affinoid families $\cF_1,\ldots,\cF_m$. The set $S_j^\prime$ of specializations of $\cF_j$ corresponding to forms in the family $(f_i,\varphi_i)_i$ satisfies (i.c), so that each $\cF_j$ is a Hida or Coleman family by part (i).
\end{proof}

\subsubsection{Families arising as specializations}

We show that every family of eigenforms can be partitioned into subfamilies that arise as specializations of affinoid families. This will allow us to rely on deformation arguments when studying $p$-supercuspidal families in Section \ref{secpsc}.

\begin{prop}\label{famlim}
Let $(f_i,\varphi_{i})_i$ be a family of eigenforms. 
Then there exists a finite partition $\N=\coprod_{j=0}^mI_i$ such that $I_0$ is finite and, for every $j\in\{1,\ldots,m\}$, $(f_i,\varphi_{i})_{i\in I_j}$ is the specialization of a 1-dimensional affinoid family of eigenforms $\cF_j=(U_j,\Theta_j,\Phi_j,S_j)$ at some subset of $S_j$. 

Moreover, for every $j\in\{1,\ldots,m\}$:
\begin{enumerate}[label=(\roman*)]
\item There exists a real quadratic field $F$ in which $p$ does not split, a rigid analytic map $\iota\colon U\to\cE_F$ mapping $S$ to a set of classical points of $\iota(U)$ and satisfying 
\begin{align*} \Theta&=(\iota^\ast\ccirc\Theta_{\cE_F})\otimes\eta, \\ 
	\varphi_1&=\iota^\ast(\Theta^\full_{\cE_F}(U_p)), \end{align*}
where $\iota^\ast\colon\cO_{\cE_F}(\cE_F)\to\cO_U(U)$ is the homomorphism induced by $\iota$, and $\eta$ is a Dirichlet character of $F$ of finite order and conductor dividing $p^\infty$.
\item The specializations $f_i$, $i\in I_j$, converge to a (not necessarily classical) specialization $x_{j,\infty}$ of $\cF_j$, and the pseudorepresentation $t_\infty\colon G_{\Q,Np}\to\C_p$ attached to $x_{j,\infty}$ corresponds to the Hecke eigensystem $\theta_\infty$ (via Definition \ref{hgcorr}). In particular, it is independent of $j$. 
\end{enumerate}
\end{prop} 

\begin{proof}
Recall that we denote by $t_i\colon G_\Q\to\Qp$ the pseudorepresentation attached to $f_i$. We write $\rho_i\colon G_\Q\to\GL_2(\Qp)$ for any representation of trace $t_i$. 

We prove part (i). By Lemma \ref{resttri}, for every $i$ there exists an inert or ramified quadratic extension $E_i$ of $\Q_p$ such that $t_i\vert_{G_{E_i}}$ is trianguline. 
Since $\Q_p$ only has a finite number of quadratic extensions, we can find a finite partition $\N=\coprod_{j=0}^kI^j$ such that $I^0$ is finite and, for every j, there exists a quadratic extension $E_j/\Q_p$ such that $t_i\vert_{G_{E_j}}$ is trianguline for every $i\in I^j$. Clearly, it is enough to prove the proposition in the case when $\N=I^j$ for some $j\ge 0$, which we assume from now on. Let $E=E_j$, and $F$ be a real quadratic field such that $F\otimes_\Q\Q_p=E$. 
We denote by $\fp$ the only $p$-adic place of $F$. Write $t_{i,F}=t_i\vert_{G_F}$ for every $i\in\N\cup\{\infty\}$, and $\ovl t_F=\ovl t\vert_{G_F}$.

For every $i\in\N$, let $f_{i,F}$ be the Hilbert eigenform attached to the base change of $f_i$ to $F$. Let $\varphi_{i,F}=\varphi_i$ if $p$ ramifies in $F$, and $\varphi_{i,F}=\varphi_i^2$ if $p$ is inert in $F$, so that $(f_{i,F},\varphi_{i,F})_i$ is a family of $\GL_{2/F}$-eigenforms by Remark \ref{bcfam}.

By Theorem \ref{finslopetri}, $f_{i,F}$ is the twist of a representation attached to a Hilbert eigenform $f_{i,F}^0$ of finite slope with a Hecke character of $F$ of finite order and conductor dividing $p^\infty$. Recall that $f_{i,F}$ and $f_{i,F}^0$ share the same Frobenius eigenvalues. 
By Lemma \ref{untwistfam}, up to removing a finite number of $f_{i,F}$, we can find a Dirichlet character $\delta$ of $F$ such that $f_{i,F}=\delta f_{i,F}^0$ for every $i$, and the sequence $(f_{i,F}^0,\varphi_i)_i$ is a family of $\GL_{2/F}$-eigenforms. 
By Remark \ref{hilbref}, the eigenform $f_{i,F}^0$ and the refinement of $t_{i,F}$ attached to the eigenvalue $\varphi_{1,i,F}$ define a point $x_{i,F}$ of $\cE_{F,\ovl t}$, and by Example \ref{exEF} this sequence of points converges to a limit point $x_{\infty,F}$ of $\cE_{F,\ovl t}$.

By Lemma \ref{adapted}, we can choose an affinoid neighborhood $U$ of $x_{\infty,F}$ such that the weight map $\omega_{\cE_{F,\ovl t}}\vert_U$ has image a disc in weight space, and is finite over such a disc. In particular, $U$ has a finite number of irreducible components. 
Since the sequence $(x_{i,F})_i$ converges to $x_{\infty,F}$, all except a finite number of the $x_{i,F}$ belong to the 1-dimensional affinoid $U$. 

Let $X_1,\ldots,X_m$ be the irreducible components that contain an infinite subsequence of $(x_{i,F})_{i\in\N}$, and pick any partition $\N=\coprod_{j=0}^mI_j$ such that $I_0$ is finite and, for every $j\in\{1,\ldots,m\}$ and $i\in I_j$, $x_{i,F}$ belongs to $X_j$. Note that $x_{\infty,F}$ belongs to all of the $X_j$; in particular, if $\cE_{F,\ovl t_F}$ is smooth at $x_{\infty,F}$, then $m=1$.

For each $j\in\{1,\ldots,m\}$, pick an irreducible affinoid neighborhood $U_{j}$ of $x_{\infty,F}$ in $X_j$. It comes equipped with the homomorphism $\Theta_{\cE_{F}}\colon\calH_F^{Np}\to\cO_{U_j}(U_j)$ and the analytic function $\varphi=\Theta^\full_{\cE_{F}}(U_\fp)\vert_{U_j}$. 
In order to construct an affinoid family of $\GL_{2/\Q}$-eigenforms, it is enough to extend $\Theta_{\cE_{F}}$ along the map $\BC_{F/\Q}\colon\calH_F^{Np}\to\calH^{Np}_\Q$ of Remark \ref{bcmap}. Since $U_j$ is a 1-dimensional affinoid and $\{x_{i,F}\}_{i\in I_j}$ is a subset of $U_j(\C_p)$ admitting an accumulation point in $U_j$, by Corollary \ref{chenargaff}, up to implicitly replacing $U_j$ with an affinoid neighborhood of $x_{\infty,F}$ nested in $U_j$, and removing a finite number of elements from $I_j$, the pseudorepresentations $t_{i}\colon G_\Q\to\Qp$ are interpolated by a pseudorepresentation $t_{U_j}\colon G_\Q\to\cO_{U_j}^\circ(U_j)$. The correspondence of Definition \ref{hgcorr} attaches to such a $t_{U_j}$ a homomorphism $\Theta_{j}\colon\calH_F^{Np}\to\cO_{U_j}^\circ(U_j)$.

Now $\cF_j=(U_j,\Theta_{j},\Phi,\{x_{i,F}\}_{i\in I_j})$ is an affinoid family of eigenforms that admits $(f_i,\varphi_{i})_{i\in I_j}$ as a specialization.

For part (ii), simply pick as $x_{j,\infty}$ the specialization of $\cF_j$ corresponding to $x_{\infty,F}$. 
\end{proof}

The above proof essentially shows that, possibly only after base change to a real quadratic field, every family of eigenforms is the specialization of an affinoid family cut out from a Hilbert eigenvariety. 

We state two corollaries of Proposition \ref{famlim}. 

\begin{cor}\label{weightinfty}
Let $(f_i,\varphi_{i})_i$ be a family of eigenforms. 
\begin{enumerate}[label=(\roman*)]
\item The sequence of weights $k_i$ diverges to $\infty$. 
\item The family $(f_i,\varphi_{i})_i$ contains at most a finite number of numerically critical eigenforms.
\end{enumerate}
\end{cor}

\begin{proof}
We use the notation of the proof of Proposition \ref{famlim}. All except a finite number of the base changes $f_{i,F}$ correspond to points of the affinoid $U$. Since $U$ is finite over its image in weight space, it contains only a finite number of points of every given weight. Therefore, it contains only a finite number of points of integer weight smaller than every given real constant. Since the eigenforms $(f_i)_i$ are assumed to be pairwise distinct, and at most two eigenforms can have a common base change to $F$, part (i) follows.

For part (ii), it is enough to observe that the valuation of $\varphi_{i}$ is bounded along the family while the weights $k_i$ diverge to $\infty$ by part (i).
\end{proof}

\begin{cor}
Let $(f_i,\varphi_i)_i$ be a family of eigenforms of finite slope. Then $(f_i,\varphi_i)_i$ can be extracted from the eigencurve as in Example \ref{exEQ}.
\end{cor}

\begin{proof}
It is enough to observe that we can take $F=\Q$, and $\eta$ to be trivial, in the proof of Proposition \ref{famlim}.
\end{proof}

\begin{rem}\mbox{ }
\begin{itemize}
\item Note that for the weak families of Coleman and Stein from Example \ref{excolste}, the limit Hecke eigensystem is actually attached to an eigenform, which need not be the case for an arbitrary weak family.
\item Though we give a complete classification of families of eigenforms in Theorem \ref{classfam}, our techniques do not seem to provide us with a classification of weak families. Such a classification would give an answer to the question posed by Coleman and Stein on which eigenforms of infinite slope can be approximated with eigenforms of finite slope. More generally, a classification of weak families would give a description of the closure, in the $p$-adic topology, of Mazur's infinite fern in the pseudodeformation space of a modulo $p$ 2-dimensional pseudorepresentation of $G_{\Q,Np}$. Coleman and Stein were a priori only interested in the modular points of such a closure. 
\end{itemize}
\end{rem}

\subsection{$p$-twisted families}


In the following, let $\delta$ be a Dirichlet character of $p$-power conductor $c(\delta)$. When we say ``$p$-power conductor'' we mean conductor $p^m$ for some $m\ge 1$ (i.e. we implicitly exclude the case of trivial conductor).

We identify $\delta$ with a finite order character of $G_\Q$, that we also denote by $\delta$, in the standard way: if $\zeta_{c(\delta)}$ is a root of unity of order exactly $c(\delta)$, then we define $G_\Q\to\Qp^\times$ as the composite of $G_\Q\onto\Gal(\Q(\zeta_{c(\delta)})/\Q)\cong(\Z/c(\delta)\Z)^\times$ and the character $(\Z/c(\delta)\Z)^\times\to\Qp^\times$ that maps $g$ to $\delta(\wtl g)$, $\wtl g$ being any representative of $g$ in $\Z$.

Let $p,N$, and the Hecke algebra $\calH_\Q$ be as before. Let $L$ be a field extension of $\Q_p$ containing the image of $\delta$ and $A$ an $L$-algebra. 
If $\theta^\full\colon\calH_\Q^{Np}\to A^\times$ is a Hecke eigensystem away from $Np$, we denote by $\delta\theta^\full$ the unique Hecke eigensystem satisfying $\delta\theta(T_\ell)=\delta(\ell)\theta(T_\ell)$ for every $\ell\nmid Np$, and set $\delta\theta^\full=\delta\theta\otimes\theta^\full\vert_{\calH_{\Q,p}}$ (i.e. we do not modify $\theta^\full$ at $p$). 

Let $f$ be an eigenform of weight $k_f$, level $\Gamma_1(Np^r)$ for some $r\ge 0$, and nebentypus $\psi$ of conductor $c(\psi)$. 
Let $\theta_f^\full\colon\calH_\Q^{N_f}\to\Qp$ be the system of Hecke eigenvalues, and $\rho_f\colon G_\Q\to\GL_2(V_f)$ the $\Qp$-linear representation attached to $f$. Let $\varphi_{1,f}, \varphi_{2,f}$ be the Frobenius eigenvalues of $\rho_f\vert_{G_{\Q_p}}$.

Let $s$ be the maximum of the valuations of $c(\delta)^2$ and $c(\delta)c(\psi)$. 
By \cite[Proposition 3.64]{shiarith}, for every eigenform $f$, there exists an eigenform $\delta f$ of weight $k$, level $\Gamma_1(Np^s)$ and character $\psi\delta^2$ whose associated Hecke eigensystem is $\delta\theta^\full_f$, and Galois representation is $\rho_{\delta f}\cong\rho_f\otimes\delta$. The representation $\rho_f\otimes\delta\vert_{G_{\Q_p}}$ is obviously still potentially semi-stable, and $D_\pst((\rho_f\otimes\delta)\vert_{G_{\Q_p}})\cong D_\pst(\rho_f\vert_{G_{\Q_p}})\otimes D_\pst(\delta\vert_{G_{\Q_p}})$. Since $\delta$ is of finite order, the Frobenius operator on $D_\pst(\delta\vert_{G_{\Q_p}})$ is trivial, so that $\varphi_{i,\delta f}=\varphi_{i,f}$ for $i=1,2$.

Let $\cF=(f_i,\varphi_i)_i$ be a family of eigenforms. As usual, for every $i$ we write $\theta_i$ for the Hecke eigensystem of $f_i$ away from $Np$. 
Since $(\theta_{i})_i$ converges to a limit $\theta_\infty$ for $i\to\infty$, a trivial check shows that $(\delta\theta_{f_i})_i$ converges to $\delta\theta_\infty$, so that $(\delta f_i)_i$, with the family of Frobenius eigenvalues $(\varphi_{i})_i$, is a family of eigenforms. 
We call it the \emph{twist of $\cF$ by $\delta$} and denote it by $\delta\cF$. 

Let $\cF=(U,\Theta,\Phi,S)$ be an affinoid family of eigenforms, defined over a $p$-adic field $L$ containing the image of $\delta$. The tuple
\[ \delta\cF\coloneqq(U,\delta\Theta,\Phi,S) \]
is again an affinoid family of eigenforms: its specialization at $x\in S$ is the eigenform $\delta f_x$, where $f_x$ is the specialization of $\cF$ at $x$, and by the above discussion $\delta(\ev_x\ccirc\alpha)$ is the system of Hecke eigenvalues away from $Np$ and $\Phi(x)$ a Frobenius eigenvalue of $\delta f_x$. We call $\delta\cF$ the \emph{twist of $\cF$ with $\delta$}.

\begin{defin}
	We call an eigenform \emph{$p$-twisted} if it is the twist of a finite slope eigenform with a Dirichlet character of $p$-power conductor.
	
	We call an (affinoid, weak, weak affinoid) family of eigenforms \emph{$p$-twisted} if it is the twist of an (affinoid, weak, weak affinoid) family of finite slope with a Dirichlet character of $p$-power conductor.
\end{defin}

\begin{rem}
	If $\delta$ is a Dirichlet character of $p$-power conductor and $\cF=(f_i)_i$ is a family of eigenforms, then all of the forms in the family $\delta\cF$ are of infinite slope: the $U_p$-eigenvalue of $f_i$ is $\delta\theta^\full_{f_i}(U_p)=\delta(p)\theta^\full_{f_i}(U_p)=0$. The same calculation gives that, for an affinoid family of eigenforms $\cG$, the twist $\delta\cG$ is of infinite slope.
\end{rem}


\begin{prop}\label{extract}
Let $\cF=(U,\Theta,\Phi,S)$ be an $L$-affinoid family of eigenforms. Then there exists a finite partition $S=\coprod_{j=0}^mS_j$ such that $S_0$ is finite and, for every $j\ge 1$, there exists a number field $F$, either equal to $\Q$ or real quadratic, with a unique $p$-adic place $\fp$, such that: 
\begin{enumerate}[label=(\arabic*)]
\item $t_x\vert_{G_{F,Np}}$ is trianguline at $\fp$ for every $x\in S_j$;
\item if $U_j$ is the Zariski closure of $S_j$ in $U$, then there exists a rigid analytic map $\iota\colon U_j\to\cE_F$ mapping $S_j$ to a set of classical points of $\iota(U_j)$ and satisfying 
\begin{align*} \Theta&=(\iota^\ast\ccirc\Theta_{\cE_F})\otimes\delta, \\ 
	\Phi&=\iota^\ast(\Theta^\full_{\cE_F}(U_p)), \end{align*}
where $\iota^\ast\colon\cO_{\cE_F}(\cE_F)\to\cO_{U_j}(U_j)$ is the homomorphism induced by $\iota$, and $\delta$ is a (possibly trivial) Dirichlet character.
\end{enumerate}

In particular, for every $j$:
\begin{enumerate}[label=(\roman*)]
	\item $U_j$ is 1-dimensional;
	\item if $U$ is 1-dimensional, then $U=U_j$ for some $j$, so that we can choose $S_0=\varnothing$, $m=1$, $S=S_1$, and moreover the field $F$ can be chosen to be any field (equal to $\Q$ or real quadratic) that satisfies condition (1) above for every $x\in S$;
	\item if the points of $S_j$ are all defined over a common $p$-adic field, then we can extract a family of eigenforms from the set $\{f_x\}_{x\in S_j}$.
\end{enumerate}
\end{prop}

Proposition \ref{extract} tells us that the classical specializations of $\cF$ live in 1-dimensional affinoid subfamilies, whose base change of the family $\cF$ to $F$ is constructed, up to a twist, from the eigenvariety $\cE_F$.

We refer to a family $(f_x)_x$ obtained from an affinoid family $\cF$, as in Proposition \ref{extract}, as a \emph{family of specializations of $\cF$}, or simply \emph{specialization of $\cF$} when this does not create confusion.

\begin{proof}
Let $t_U\colon G_{\Q,Np}\to\cO_U^\circ(U)$ be the pseudorepresentation attached to $\cF$. 
With the same reasoning as in the proof of Proposition \ref{famlim}, we can find a finite partition $S=\coprod_{i=0}^mS^j$ with the property that $S^0$ is finite and, for every $j\in\{1,\ldots,m\}$, there exists a number field $F_j$, either equal to $\Q$ or real quadratic, with a unique $p$-adic place $\fp_j$, such that, for every $x\in S^j$, the restriction $t_x\vert_{G_{F_{j,\fp_j}}}$ is trianguline. In particular, by Theorem \ref{finslopetri}, the base change $f_{x,F_j}$ of $f$ to $\GL_{2/F_j}$ is the twist of a $\GL_{2/F_j}$-eigenform $f_{x,F_j}^0$ of finite slope with a Dirichlet character $\delta_x$ of $F$ of conductor dividing $\fp_j^\infty$. For $j\in\{1,\ldots,m\}$, let $U^j$ be the Zariski closure of $S^j$ in $U$. If $U^j_0$ is an irreducible component of $U^j$, then $\cF^j\coloneqq(U^j_0,\Theta\vert_{U^j_0},\Phi\vert_{U^j_0},S^j\cap U^j_0(\C_p))$ is an affinoid family of eigenforms. It is clearly enough to prove the statement for each family $\cF^j$. Therefore, to simplify the notation, we remove all reference to $j$ in what follows.

Up to enlarging $L$, we can assume that $L$ contains the images of all the embeddings $F\into\Qp$. Let $\cE_F^\tw$ be the (base change to $L$ of the) ``eigensurface'' introduced in the proof of Lemma \ref{untwistfam}, equipped with its natural embedding 
\[ \cE_{F,\ovl t_F}^\tw\subset\fR_{\ovl t_F}\times_L\G_m. \]
Set $t_{U,F}=t_{U}\vert_{G_{F,Np}}$, and $\ovl t_F=\ovl t_{U,F}$. Given that $t_{U,F}$ is a deformation of $\ovl t_F$ to an $L$-affinoid, the universal property of $\fR_{\ovl t_F}$ given in Section \ref{psrig} provides us with a rigid analytic map $\pi^\ps_{U}\colon U\to\fR_{\ovl t_F}$. 
Since $\Phi$ is analytic and nonvanishing on $U$, we can define a closed embedding
\[ \pi_{U}\coloneqq\pi_{U}^\ps\times \varphi_1\colon U\to\fR_{\ovl t_F}\times_L\G_m. \]
Since $S$ is dense in $U$, $\pi_U(S)$ is dense in $\pi_U(U)$, and since $\pi_U(S)$ is contained in the closed subspace $\cE_{F,\ovl t_F}^\tw\subset\fR_{\ovl t_F}\times_L\G_m$, we obtain that $\pi_U(U)\subset\cE_F^\tw$. The image of $U$ under the closed embedding $\pi_U$ is affinoid, so its image under the map $\cE_F^\tw\to\wtl\cW_F^\ast$ introduced in the proof of Proposition \ref{untwistfam}, parameterizing twist characters, is also affinoid. In particular, the norms of the conductors of the characters $\delta_x$, for $x\in S$, are all bounded by a common constant, so that there is only a finite number of choices for the characters $\delta_x$. Therefore, we can find a finite partition $S=\coprod_{i=0}^mS_j$, where $S_0$ is finite and, for every $x,y\in S\setminus S_0$, $\delta_x=\delta_y$ if and only if $x,y$ both belong to $S_j$ for the same $j\in\{1,\ldots,m\}$. For every $j\in\{1,\ldots,m\}$, let $U_j$ be the Zariski closure of $S_j$ in $U$. Then the image of the closed embedding 
\begin{align*} U&\to\fR_{\ovl t_F}\times_L\G_m. \\
x &\mapsto (\delta_x^{-1}t_x,\Phi(x)) \end{align*}
is an affinoid subdomain of the eigencurve $\cE_F$, and one easily checks that the resulting embedding $\iota\colon U_j\into\cE_F$ satisfies the properties required in part (2) of the proposition. 

%
%

Since $\cE_F$ is 1-dimensional, so is $U_j$. If $U$ is already 1-dimensional, then it must coincide with one of the Zariski-closed subspaces $U_j$ (i.e. $S_0$ is empty and $S_1=S$), since it is reduced and irreducible by the definition of an affinoid family. The statement about the choice of $F$ is clear from the proof of the first part of the proposition.

Now assume that all of the points of $S_j$ are defined over a common $p$-adic field. Since $U$ is 1-dimensional and $S_j$ is Zariski-dense in $U$, it admits an accumulation point $x_\infty\in U_j(\C_p)$ by Lemma \ref{dense}. Let $\{x_i\}_{i\in\N}$ be a sequence of elements of $S_j$ converging to $x_\infty$. The associated refined eigenforms obviously form a family.
\end{proof}


\begin{prop}\label{ptwequiv}\mbox{ }
	\begin{enumerate}[label=(\roman*)]
				\item Let $\cF=(U,\Theta,\Phi,S)$ be a 1-dimensional affinoid family of eigenforms. The following are equivalent:
		\begin{enumerate}[label=(\alph*)]
			\item $\cF$ is a $p$-twisted affinoid family;
			\item there exists a Dirichlet character $\delta$ of $p$-power conductor, and a dense subset $S^\prime$ of $U$, such that, for $x\in S^\prime$, $\ev_x\ccirc\alpha^{Np}$ is the Hecke eigensystem away from $Np$ of a twist $\delta f_x^0$, for an eigenform $f_x^0$ of finite slope;
			\item there exists a dense subset $S^\prime$ of $U$ such that $\ev_x\ccirc\Theta$ is attached to a $p$-twisted eigenform for every $x\in S^\prime$.
		\end{enumerate}
				\item Let $(f_i,\varphi_i)_i$ be a family of eigenforms. The following are equivalent:
		\begin{enumerate}[label=(\alph*)]
			\item for all except a finite number of $i$, $f_i$ is a $p$-twisted eigenform; 
			\item $(f_i,\varphi_i)_i$ is eventually a $p$-twisted family; 
			\item there exists a finite partition $\N=\coprod_jJ_j$ such that, for every $j$, $J_j$ is infinite and $(f_i,\varphi_i)_{i\in J_j}$ is eventually the specialization of a $p$-twisted affinoid family.
		\end{enumerate}
	\end{enumerate}
\end{prop}


\begin{proof}
We prove (i). 
The implications (a)$\implies$(b)$\implies$(c) are trivial. We show that (c)$\implies$(b)$\implies$(a). 
We prove that (c)$\implies$(a). For every $x\in S^\prime$, the local-at-$p$ Galois representation attached to $f_x$ is trianguline, being the twist with a character of a representation attached to a finite slope eigenform. Therefore, one can apply Proposition \ref{extract}(ii) to the family $\cF^\prime=(U,\Theta,\Phi,S^\prime)$, with $F=\Q$, to deduce that $\cF^\prime$ is the twist of a finite slope affinoid family with a Dirichlet character. Obviously, this implies that the same holds for $\cF$.

We now prove (ii). The implication (c)$\implies$(a) is obvious. If we assume (a), then by Lemma \ref{untwistfam}, there exists a Dirichlet character $\delta$ of $p$-power conductor such that $f_i=\delta f_i^0$, with $f_i^0$ an eigenform of finite slope, for every sufficiently large $i$. Then, $(f_i,\varphi_i)_i$ is eventually a twist of the finite slope family $(f_i^0,\varphi_i)_i$ with $\delta$, so that (b) holds.

Finally, assume (b). By Proposition \ref{famlim}, there exists a finite partition $\N=\coprod_{j=0}^mI_i$ such that $I_0$ is finite and, for every $j\in\{1,\ldots,m\}$, $(f_i,\varphi_{i})_{i\in I_j}$ is the specialization of a 1-dimensional affinoid family of eigenforms $\cF_j$. By part (i) of this proposition, (c)$\implies$(a), each family $\cF_j$ is a $p$-twisted family, so that (a) holds.
%
%
\end{proof}

\section{Families with complex multiplication}\label{seccm}

\subsection{CM eigenforms}
Let $K$ be a number field. Let $\fm$ be a non-zero ideal of $\cO_K$, and let $J_K(\fm)$ be the group of fractional ideals of $K$ prime to $\fm$. Given an arbitrary abelian group $A$, an \emph{$A$-valued Gr\"ossencharacter} of $K$ of modulus $\fm$ is a group homomorphism $J_K(\fm)\to A$. 

We identify in the usual way $A$-valued Gr\"ossencharacters of modulus $\fm$ with $A$-valued ``idelic'' Gr\"ossencharacters, that is, continuous characters $K^\times\backslash\A_K^\times\to A$ vanishing on $\cU(\fm)\cdot F_\infty^\circ$, where $\cU(\fm)$ is the compact open subgroup of $\A_F^{\infty}$ attached to $\fm$. 

Now let $K$ be a quadratic extension of $\Q$, $A$ an abelian group, and $\psi\colon J_K(\fm)\to A$ a Gr\"ossencharacter of modulus $\fm\subset\cO_K$. For every $n\in\N_{\ge 1}$, we set
\[ a_n(\psi)=\sum_{\substack{\fa\textrm{ ideal of }\cO_K,\\ \Norm_{K/\Q}{\fa}=n,\,(\fa,\fm)=1}}\psi(\fa). \]
Let $M=\Norm_{K/\Q}(\fm)$. 
A simple check shows that there exists a unique group homomorphism $\alpha^{M}(\psi)\colon\calH^{M}\to A$ satisfying $\alpha^{M}(T_n)=a_n(\psi)$ for every $n\ge 1$. If $M^\prime$ is a multiple of $M$, we write $\alpha^{M^\prime}$ for the restriction of $\alpha^M$ to a homomorphism $\calH^{M^\prime}\to A$.

We specialize the above definitions to the case when $K$ is imaginary quadratic of discriminant $D$, and $A=\C^\times$. Let $\{\sigma,\ovl\sigma\}$ be the pair of complex embeddings of $K$, and assume that the infinity type of $\psi$ is $(k-1,0)$. We attach a CM eigenform to $\psi$, as explained for instance in \cite[Section 3]{ribetneben}. 

Let $\varphi_K$ be the quadratic Dirichlet character modulo $D$ attached to $K$, and let $\eta$ be the Dirichlet character modulo $M$ given by
\[ x\mapsto \frac{\psi((x))}{\sigma(x^{k-1})} \]
for every $x\in\Z$. Let $\vareps=\varphi_K\eta$, a Dirichlet character modulo $DM$.

\begin{prop}\label{psiform}
The series
\[ \sum_{n\ge 1}a_n(\psi)q^n \]
is the $q$-expansion of a new eigenform $f_\psi$ of level $\Gamma_1(DM)$ and character $\vareps$. 
\end{prop}

Obviously, the Hecke eigensystem of $f_\psi$ outside of $DM$ is $\alpha^{DM}(\psi)$.

\begin{proof}
The first statement is \cite[Lemma 3]{shiarith}, together with the remark on \cite[p. 138]{shicla} about the form $f_\psi$ being new.
\end{proof}


\begin{rem}\label{CMclass}
Let the notation be as in Proposition \ref{psiform}. Write $DM=Np^s$ with $N\in\N$ and $s=v_p(D)+v_p(M)$. 

We explain how to attach to $f_\psi$ an eigenform of level $\Gamma_1(Np^r)$ for some $r\ge 1$. We distinguish various cases, depending on the $p$-adic valuations of $D$ and $M$. 
\begin{enumerate}
	\item If $p\nmid DM$, then $N=DM$ and $a_p(\psi)=0$. Since $f_{\psi}$ is old at $p$, we can $p$-stabilize it gives two eigenforms of level $\Gamma_1(N)\cap\Gamma_0(p)$, character $\vareps$ of conductor $N$, and slopes 0 and $k-1$, respectively. 
	\item If $p\nmid M$ but $p\mid D$, then since $p$ is odd $s=1$, $p$ is ramified in $F$ and $f_{\psi}$ is a newform of level $\Gamma_1(Np)$, character $\vareps$ of conductor $Np$, and $U_p$-eigenvalue $a_p(\psi)=\psi(\fp)\ne 0$. A direct calculation shows that the slope of $f_{\psi}$ is $\frac{k-1}{2}$. 
	\item If $p\mid M$, then $f_\psi$ is a newform of level $\Gamma_1(Np^s)$, character $\vareps$ of conductor dividing $Np^s$. There are two possibilities for the slope of $f_\psi$:
	\begin{itemize}
		\item if all of the $p$-adic places of $K$ divide $\fm$, then $a_p(f_\psi)=0$, so that the slope of $f_\psi$ is infinite (this is automatically the case if $p$ is inert or ramified in $K$);
		\item if $p$ splits as $\fp\ovl\fp$ in $K$, and $\fp\mid\fm$ but $\ovl\fp\nmid\fm$, then $a_p(f_\psi)=\psi(\ovl\fp)\in\Zp^\times$, so that $f_\psi$ is ordinary.
	\end{itemize}
\end{enumerate}
The above possibilities exhaust all CM eigenforms of level $\Gamma_1(Np^r)$ for $r\ge 1$.
\end{rem}



In Proposition \ref{CMclassprop}, we state a more precise version of the classification from Remark \ref{CMclass}. Given two algebraic Gr\"ossencharacters $\psi_1,\psi_2$ of $K$ of moduli $\fm_1$ and $\fm_2$, respectively, we always consider their product $\psi_1\psi_2$ as a Gr\"ossencharacter of $K$ of modulus $\fm_1\fm_2$. 

\begin{defin}\label{ptwist}\mbox{ }
We say that $\psi$ is \emph{$p$-primitive} if it cannot be written as $\psi=\psi_0\cdot(\delta\ccirc\Norm_{K/\Q})$ for a Gr\"ossencharacter $\psi_0$ of $K$ and a Dirichlet character $\delta$ of $\Q$ of $p$-power (non-trivial) conductor.
\end{defin}

\begin{prop}\label{CMclassprop}\mbox{ }
\begin{enumerate}[label=(\roman*)]
\item The eigenform $f_\psi$ is a $p$-newform of infinite slope if and only if all of the $p$-adic places of $K$ divide $\fm$. 
\item Write $\psi=\psi_0\cdot(\delta\ccirc\Norm_{K/\Q})$ for a $p$-primitive Gr\"ossencharacter $\psi_0$ of conductor $\fm_0$ and a Dirichlet character $\delta$ of $p$-power conductor $m$. Then:
\begin{enumerate}[label=(\arabic*)]
\item If $m=1$ (i.e., $\psi$ is $p$-primitive) and $p$ does not divide $\Norm_{K/\Q}(\fm_0)$, then $f_\psi$ is of finite slope.
\item If $m\ne 1$ and $p$ does not divide $\Norm_{K/\Q}(\fm_0)$, then $f_\psi$ is an infinite slope twist of the finite slope eigenform $f_{\psi_0}$.
\item If $p$ splits as $\fp\ovl\fp$ in $K$, and $\fp\mid\fm_0$ but $\ovl\fp\nmid\fm_0$, then $f_\psi$ is ordinary if $m=1$, and an infinite slope twist of an ordinary eigenform if $m>1$.
\item If all of the $p$-adic places of $K$ divide $\fm_0$, then $f_\psi$ is $p$-supercuspidal (as well as $f_{\psi_0}$). This holds in particular when $p$ divides $\Norm_{K/\Q}(\fm_0)$ and is inert or ramified in $K$. 
\end{enumerate}
\end{enumerate}
\end{prop}


\begin{proof}
The only statement that does not follow trivially from Remark \ref{CMclass} is the last one. Remark \ref{CMclass} gives that if all of the $p$-adic places of $K$ divide $\fm_0$, then $f_{\psi_0}$ is of infinite slope. If $f_{\psi_0}$ were a twist of an eigenform $g$ of finite slope by a Dirichlet character $\delta$, then $g$ would have to be of the form $f_{\psi_1}$ for a Gr\"ossencharacter $\psi_1$ of $K$ satisfying $\psi_0=\psi_1\cdot(\delta\ccirc\Norm_{K/\Q})$, which is impossible since $\psi_0$ is $p$-primitive. Therefore, $f_{\psi_0}$ is $p$-supercuspidal, and its twist $f_\psi$ is also $p$-supercuspidal.
\end{proof}


\begin{rem}
Using Proposition \ref{CMclassprop}, we can construct examples of both kinds of CM eigenforms of infinite slope. If $\psi=\psi_1\cdot(\delta\ccirc\Norm_{K/\Q})$ for a Gr\"ossencharacter $\psi_0$ of $K$ of conductor prime to $p$ and infinity type $(k-1,0)$, and a Dirichlet character $\delta$ of $p$-power conductor, then $f_\psi$ is a twist of an eigenform of finite slope with $\delta$.

To obtain instead a $p$-supercuspidal CM eigenform, start for instance with $K=\Q(\sqrt{-p})$ for a regular prime $p$. As before, write $\varphi_K$ for the quadratic Dirichlet character (of $\Z$) attached to the extension $K/\Q$. Choose $\psi$ such that the finite part of $\psi$, restricted to the group of principal fractional ideals of $K$, is the character $(\cO_K/\fp)^\times\to\C^\times$ obtained by composing the Dirichlet character $\varphi_K$ with the isomorphism $(\cO_K/\fp)^\times\to(\Z/p\Z)^\times$ induced by $\Z\into\cO_K$. 
Any such $\psi$ is clearly $p$-primitive, so that $f_\psi$ is $p$-supercuspidal by Proposition \ref{CMclassprop}(ii.4). This agrees with the fact that $f_\psi$ fits in case (iii) of Proposition \ref{lwclass}: since $\eta$ coincides with $\varphi_K$ and $p$ is regular, $\vareps_p=\varphi_K^2$ is the trivial character modulo $p^2$. The admissible character $\omega_p$ associated with $\vareps$ is trivial, hence its conductor is 1. Since $M=p$, the resulting form $f_\psi$ has level $\Gamma_1(p^2)$ if $p\equiv -1\pmod{4}$, and $\Gamma_1(4p^2)$ if $p\equiv 1$ or $2\pmod{4}$. 
\end{rem}


\subsection{CM families}

We define affinoid families of CM modular forms, with no prescription on the slope. It will turn out that their slope can only be 0 or infinite. The affinoid families we construct can be extracted from the $\Lambda$-adic families that Hida defines in \cite[Section 7.6]{hidaeis}, but we introduce them in a slightly different way.

In the following, $N$ denotes a positive integer prime to $p$.

\begin{defin}\label{CMfamdef} 
	A \emph{(weak) CM family of eigenforms} is a (weak) family of eigenforms $(f_i)_i$ such that every $f_i$ is a CM eigenform.
	
	A \emph{CM affinoid family of eigenforms} is an affinoid family of eigenforms $(U,\Theta,\Phi,S)$ such that, for every $x\in S$, the eigenform $f_x$ is CM.
\end{defin}

\begin{lemma}\label{CMfamrep}\mbox{ }
	\begin{enumerate}
		\item Let $\cF=(U,\Theta,\Phi,S)$ be an affinoid family of eigenforms, with associated pseudorepresentation $t_U$. 
		Then the following are equivalent:
		\begin{enumerate}
			\item $\cF$ is an affinoid CM family,
			\item $t_U$ is induced by a character of $G_K$, $K$ an imaginary quadratic field;
			\item there exists a dense subset $S^\prime$ of $U$ such that $\ev_x\ccirc t_U$ is induced for every $x\in S^\prime$.
		\end{enumerate}
		Moreover, under any of the above equivalent conditions, all of the classical specializations of $\cF$ are CM. 
		\item Let $(f_i,\varphi_i)_i$ be a family of eigenforms. Then $(f_i,\varphi_i)_i$ is CM if and only if it is eventually a finite union of specializations of affinoid CM families.
	\end{enumerate}
\end{lemma}

\begin{proof}
	By Proposition \ref{extract}, $U$ is 1-dimensional. 
	The implication (a)$\implies$ (c) of part (i) is obvious: simply take $S^\prime=S$. If (b), holds, then $\ev_x\ccirc t_U$ is induced for every $x\in S$, so that $f_x$ is a CM form for every $x$, giving (a); the same argument gives the ``moreover'' statement. Now assume that (c) holds. Let $\ovl t_U$ be the residual pseudorepresentation attached to $t_U$; since $S^\prime$ is non-empty, $\ovl t_U$ is induced by a character of $G_K$, $K$ a quadratic extension of $\Q$, and there is only a finite number of choices for such an $K$ (see for example \cite[Lemma 2.4.4]{calsar}). For every such $K$, the $K$-induced locus in $U$ is closed by Corollary \ref{indclo}(ii). If (c) holds, there must be an $K$ for which the $K$-induced locus coincides with $U$, so that $t_U$ is induced. If $K$ is real, the deformation space of a residual character only has a finite number of points up to cyclotomic twists, and since one of the Hodge--Tate--Sen weights of $t_U$ is constant equal to 0, $t_U$ cannot be induced from a character of $G_K$. Therefore, $K$ has to be imaginary, giving (b). 
	This completes the proof of (i).
	
	By Proposition \ref{famlim}, every family of eigenforms $(f_i,\varphi_i)_i$ is eventually a finite union of specializations of affinoid families, so (ii) follows immediately from (i).
\end{proof}

%

Let $K$ be an imaginary quadratic field. 
Let $\cW_{K,1}$ and $\cE_{K,1}$ be the weight space and eigenvariety of tame level $N$ for $\GL_{1/K}$ introduced in Section \ref{GL1}. Let $\sigma_1,\sigma_2\colon\cO_{L,p}^\times\to\C_p^\times$ be the group homomorphisms attached to the two embeddings $L\into\C_p$, and let $\cW_{K,1}^0$ be the Zariski closure of the set of classical weights corresponding to homomorphisms $\cO_{L,p}^\times\to\C_p^\times$ that coincide with $\sigma_i^n$, for some $i\in\{1,2\}$ and $n\in\Z$, over some subgroup of finite index of $\cO_{L,p}^\times$.

Let $\cE_{K,1}^0=\cE_{K,1}\times_{\cW_{K,1}}\cW_{K,1}^0$. The restriction of the weight map to $\cE_{K,1}^0$ is a finite map $\cE_{K,1}^0\to\cW_{K,1}^0$. By definition, the classical points of $\cE_{K,1}^0$ are those corresponding to algebraic Gr\"ossencharacters of infinity type $(k-1,0)$ for one of the two orderings of the embeddings $L\into\C_p$. 

Since the Leopoldt conjecture holds for the quadratic field $K$, $\cE_{K,1}^0$ is 1-dimensional, and the set of classical weights is a dense and accumulation subset of $\cW_{K,1}$. Since a point of $\cE_{K,1}$ is classical if and only if its weight is classical, and the weight map is finite, the set of classical points is dense and accumulation in $\cE_{K,1}$. The same statements hold with $\cW_{K,1}$ and $\cE_{K,1}$ replaced with $\cW_{K,1}^0$ and $\cE_{K,1}^0$. 

Let $U$ be an integral affinoid subdomain of $\cE_{K,1}^0$ containing a classical point, and let $\psi_U\colon J_K(Np^\infty)\to\cO_U(U)^\times$ be the restriction to $U$ of the universal Gr\"ossencharacter of $\cE_{K,1}$. 
If $p$ is ramified in $K$, fix a uniformizer $\pi$ of $E=K\otimes_\Q\Q_p$; otherwise, simply set $\pi=p$. 
We define an affinoid family of eigenforms $\cF_{\psi_U}=(U,\Theta,\Phi,S)$ as follows:
\begin{itemize}
\item $\Theta$ is the Hecke eigensystem $\Theta(\psi_U)$, away from $Np$, that maps the Hecke operator $T_\ell$ to
\[ a_\ell(\psi_U)=\sum_{\substack{\fa\textrm{ ideal of }\cO_K,\\ \Norm_{K/\Q}{\fa}=\ell,\,(\fa,Np)=1}}\psi_U(\fa). \] 
\item $\Phi=\psi_U(\pi)$, where we are looking at $\psi_U$ as a character on the idèles group of $K$; 
\item $S$ is the set of classical points $x$ of $U$. 
\end{itemize}
For every point $x$ of $U$, we denote by $\psi_x$ the $p$-adic Gr\"ossencharacter obtained by specializing $\psi_U$ at $x$.

Note that the value of $\Phi$ depends on the choice of uniformizer $\pi$; this is not surprising since the Frobenius eigenvalues of an eigenform also depend on such a choice.

\begin{prop}
The tuple $\cF_{\psi_U}$ is a (1-dimensional) CM affinoid family of eigenforms.
\end{prop}

\begin{proof}
The set $S$ is dense in $U$, since $S$ is a dense and accumulation subset of $\cE_{K,1}^0$. For every $x\in S$, $\ev_x\ccirc\alpha^{Np}$ is the Hecke eigensystem outside $Np$ of the CM eigenform $f_{\psi_x}$. Since the restriction to $G_K$ of the $p$-adic Galois representation $\rho_{f_{\psi_x},p}$ is the direct sum of the character of $G_K$ attached to $\psi_x$ and its conjugate, $\Phi(x)=\psi_x(\fp)$ is one of the eigenvalues of the potentially semistable Frobenius of $f_{\psi_x}$. 

We attach to $\psi_U$ a continuous character $\chi_U\colon G_K\to\cO_U(U)^\times$ via global class field theory, and the pseudorepresentation $t_U=\Ind_{G_K}^{G_\Q}\chi_U$ interpolates the traces of the representations $\rho_{f_{\psi_x},p}$ when $x$ varies over $S$. The resulting map to the universal deformation ring of $\ovl t_U$ is finite, since the map from the eigencurve $\cE_{F,1,\ovl t_U}$ to the deformation ring of $\ovl t_U$ is locally a closed embedding.
\end{proof}



\begin{prop}\label{CMfam}
All CM affinoid families of eigenforms of tame level $N$ have the form $\cF_{\psi_U}$ for some choice of $K$ and $U$ as above (up to modifying the choice of the set $S$). 
\end{prop}

\begin{proof}
Let $\cF=(U,\Theta,\Phi,S)$ be a CM family of eigenforms. By Lemma \ref{CMfamrep}, the pseudorepresentation $t_U$ is induced by a character $\chi_U\colon G_K\to\cO_U(U)^\times$ for a quadratic field $K$. Via global class field theory, we can attach to $\chi_U$ a Gr\"ossencharacter $\psi_U\colon J_K(N)\to\cO_U(U)^\times$. Then a simple check shows that $\cF=\cF_{\psi_U}$, up to modifying $S$. 
\end{proof}

\begin{prop}\label{CMinf}
	With the notation of Proposition \ref{CMfam}, let $\cF=\cF_{\psi_U}=(U,\Theta,\Phi,S)$ be an affinoid CM family. Then $\cF$ satisfies exactly one of the following:
	\begin{enumerate}
		\item $\cF$ is a CM affinoid Hida family, in which case all of its classical specializations are ordinary CM eigenforms; 
		\item $\cF$ is a twist of a CM affinoid Hida family with a Dirichlet character of conductor divisible by $p$, in which case all of its classical specializations are $p$-twists of CM eigenforms of finite slope;
		\item $\cF$ admits a $p$-supercuspidal specialization, in which case all except a finite number of classical specializations of $\cF$ are $p$-supercuspidal CM eigenforms.
	\end{enumerate}
\end{prop}

In case (iii), we say that $\cF$ is a \emph{$p$-supercuspidal CM family}. 

\begin{proof}
If $\cF$ admits an infinite number of specializations $f_{\psi_x}$ of finite slope, then it is an affinoid family of finite slope by Proposition \ref{fsequiv}. From the classification in Remark \ref{CMclass}, we read that if the slope of $f_{\psi_x}$ (or its $p$-stabilization if $f_{\psi_x}$ has level prime to $p$) is finite, then it is equal to either $0, k-1$ or $(k-1)/2$. Since the slope is locally constant along $\cF$ while the weight is not, the slope has to be 0 along the family. 

If $\cF$ admits an infinite number of specializations that are $p$-twisted eigenforms, then by Proposition \ref{ptwequiv} it is a $p$-twist of an affinoid CM family of finite slope, that is necessarily a Hida family by the previous paragraph. 

If $\cF$ admits a $p$-supercuspidal specialization, then it cannot satisfy either (i) or (ii). Therefore, it cannot have an infinite number of classical specializations that are either or finite slope, or $p$-twists of eigenforms of finite slope. 
%
\end{proof}

\subsection{A classification of families}

%
%

Let $(f_i,\varphi_i)_{i\in\N}$ be a family of eigenforms. Consider the sets:
\begin{itemize}
\item $S_\fin=\{i\in\N\,\vert\, f_i$ is of finite slope$\}$;
\item $S_{\tw}=\{i\in\N\,\vert\, f_i$ is the twist of an eigenform of finite slope with a Dirichlet character of $p$-power conductor$\}$;
\item $S_\psc=\{i\in\N\,\vert\, f_i$ is $p$-supercuspidal$\}$.
\end{itemize}
By the standard classification that follows from instance from \ref{lwclass}, $\N=S_\fin\amalg S_\tw\amalg S_\psc$. 

\begin{lemma}\label{eventual}\mbox{ }
At most one between $S_\fin$ and $S_\tw$ is infinite.
\end{lemma}

\begin{proof}
Assume that at least one between $S_\fin$ and $S_\tw$ is infinite, so that $(f_i,\varphi_i)_{i\in S_\fin\amalg S_\tw}$ is a family of eigenforms. 
By Proposition \ref{untwistfam}, either $f_i=f_i^0$ for almost every $i\in S_\fin\amalg S_\tw$, or $f_i=\delta f_i^0$ for every $i\in S_\fin\amalg S_\tw$ and a Dirichlet character of $p$-power conductor, independent of $i$. If $S_\fin$ is infinite, we must be in the first case; if $S_\tw$ is infinite, we must be in the second one.
\end{proof}

If $S_\fin$ and $S_\tw$ are both finite, then $(f_i,\varphi_i)_i$ is eventually $p$-supercuspidal. 
In this case, we will prove in Theorem \ref{infslope} that $(f_i,\varphi_i)_i$ is eventually a CM family. 
We deduce the following.

\begin{lemma}\label{pscfam}
Assume that $S_\psc$ is infinite. Then $S_\fin\amalg S_\tw$ is finite, and the family $(f_i,\varphi_i)_i$ is eventually a $p$-supercuspidal CM family.
\end{lemma}

\begin{proof}
By Theorem \ref{infslope}, the family $(f_i,\varphi_i)_{i\in S_\psc}$ is CM (and $p$-supercuspidal). In particular, the limit pseudorepresentation $t_\infty$ is induced, and the restriction $t_\infty\vert_{G_{\Q_p}}$ is irreducible. If either $S_\fin$ or $S_\tw$ is infinite, then we obtain either the finite slope family $(f_i,\varphi_i)_{i\in S_\fin}$, or a $p$-twisted family $(f_i,\varphi_i)_{i\in S_\tw}$ whose associated pseudorepresentations converge to $t_\infty$. However, by Proposition \ref{famlim}, $(f_i,\varphi_i)_{i\in S_\fin}$ would be a family extracted from the eigencurve $\cE_\Q$, and $(f_i,\varphi_i)_{i\in S_\tw}$ a $p$-twist of a family extracted from $\cE_\Q$, so that $t_\infty$ would be attached to a finite slope CM eigenform, or a $p$-twist of it. In both cases, $t_\infty$ would have CM by an imaginary quadratic field in which $p$ splits, hence $t_\infty\vert_{G_{\Q_p}}$ would be decomposable, a contradiction. 
\end{proof}

%


By combining Lemmas \ref{eventual} and \ref{pscfam}, we obtain the following complete classification of $p$-adic families of eigenforms, in particular of those of infinite slope.

\begin{thm}\label{classfam}
Let $(f_i,\varphi_i)_{i}$ be a family of eigenforms whose associated residual pseudorepresentation $\ovl t$ satisfies conditions \ref{hyprhobar1}, \ref{hyprhobar2} of Section \ref{secpsc}. Then $(f_i,\varphi_i)_{i}$ is eventually of exactly one of the following types:
\begin{enumerate}[label=(\roman*)]
\item a Hida or Coleman family;
\item the twist of a finite slope family with a Dirichlet character of conductor divisible by $p$;
\item a $p$-supercuspidal CM family. 
\end{enumerate}
\end{thm}


We also give a classification of affinoid families, which follows immediately from Proposition \ref{extract}. 

\begin{thm}
Let $\cF$ be an affinoid family of eigenforms, whose associated residual pseudorepresentation $\ovl t$ satisfies conditions \ref{hyprhobar1}, \ref{hyprhobar2} of Section \ref{secpsc}. Then $\cF$ fits in one of the following cases: 
\begin{enumerate}[label=(\roman*)]
\item $\cF$ is an affinoid Hida or Coleman family; 
\item $\cF$ is a twist of an affinoid Hida or Coleman family with a Dirichlet character of conductor divisible by $p$;
\item $\cF$ is a $p$-supercuspidal affinoid CM family. 
\end{enumerate}
\end{thm}

\medskip

\section{Families of potentially trianguline representations}\label{secloc}


This section is devoted to the proof of the main local result of the paper, Theorem \ref{locthm}. In Section \ref{secfam} we will give a global application of it, to the study of families of eigenforms of infinite slope.

Let $L$ be an algebraic extension of $\Q_p$. We recall and introduce some notation. Let $X$ be a rigid analytic space over $L$, $G$ a topological group and $t\colon G\to\cO_X(X)$ a continuous pseudorepresentation. For every extension $L^\prime$ of $L$ and point $x\in X(L^\prime)$, we denote by $t_x\colon G\to L^\prime$ the pseudorepresentation obtained by specializing $t$ at $x$, that is, the composition of $t$ with the evaluation at $x$ map $\cO_X(X)\to L^\prime$.

Given a subspace $U\subset X$, the \emph{restriction of $t$ to $U$} is the pseudorepresentation obtained by composing $t$ with the restriction map $\cO_X(X)\to\cO_X(U)$. We sometimes denote it by $t\vert_U$ (this should not create any confusion with the restriction to a subgroup $H$ of $G$, denoted in a similar way).

Now let $U$ be a 1-dimensional, integral affinoid space over $L$, equipped with a 2-dimensional pseudorepresentation $t_U\colon G_{\Q_p}\to\cO_U(U)$. 
As $x$ varies over the $\Qp$-points of $U$, we assume that one of the Hodge--Tate--Sen weights is always 0. In particular, since Hodge--Tate--Sen weights vary locally analytically, the second Hodge--Tate--Sen weight of $x$ is interpolated by a rigid analytic function $k\in\cO_U(U)$. We assume that
\begin{equation}\tag{ncst}\label{ncst} \text{\emph{$k$ is not constant.}} 
\end{equation} 


Let $S$ be an infinite subset of $U(\C_p)$ such that:
\begin{enumerate}[label=(H\arabic*)]
	\item \label{hyp1} for every $x\in S$, the pseudorepresentation $t_x$ is potentially refined trianguline (in particular, it is potentially semistable); 
	\item \label{hyp2} there exist elements $\varphi_1,\varphi_2\in\cO_U(U)^\times$ such that, for every $x\in S$, the eigenvalues of the Frobenius operator on $D_\pst(\rho_x)$ are $p^{k_1}\varphi_1(x)$ and $p^{k_2}\varphi_2(x)$. 
\end{enumerate}
For every $x\in S$, $t_x$ becomes semistable over a dihedral extension $L_x$ of $\Q_p$ by Corollary \ref{pstdih}. Moreover, by Lemma \ref{typetri} $t_x$ is potentially trianguline, but not trianguline. 

For $x\in S$ the Frobenius eigenvalues $\varphi_{1,x},\varphi_{2,x}$ are assumed to be distinct, so there exist rank 1 $(\varphi,N,G_{\Q_p})$-subquotients $D_{x,1}$ and $D_{x,2}$ of $D_x$ such that, for $i=1,2$, the semistable Frobenius acts on $D_{x,i}$ with eigenvalue $\varphi_{i,x}$ and $G_{\Q_p}$ via a finite order character $\psi_{x,i}\colon G_{\Q_p}\to L_x^\times$. We further assume that:
\begin{enumerate}[resume,label=(H\arabic*)]
\item \label{hyp3} there exists a finite extension $\wtl E$ of ${\Q_p}$ such that $\psi_{x,1}$ factors through the finite quotient $G_{\Q_p}\to\Gal(\wtl E/{\Q_p})$ for every $x\in S$.
\end{enumerate}
We will see below that \ref{hyp3} implies the same statement with $\psi_{x,1}$ replaced by $\psi_{x,2}$. 


\begin{rem}\label{replace}
	Recall that we are equipping all of our rigid analytic spaces with the analytic Zariski topology. In particular, since $U$ is irreducible of dimension 1, any infinite subset of $U(\C_p)$ (in particular, $S$) is dense in $U$. 
\end{rem}



The following is our main local Galois-theoretic result.

\begin{thm}\label{locthm}
	The pseudorepresentation $t_U$ is $E$-induced for a quadratic extension $E/\Q_p$. 
\end{thm}

The proof of Theorem \ref{locthm} relies in a crucial way on Theorem \ref{berche}. The main obstacle in generalizing Theorem \ref{locthm} to representations of $G_K$, $K$ a finite extension of $\Q_p$, is that we do not know whether an analogue of Theorem \ref{berche} holds in this generality.

We prove a simple lemma, for later use.

\begin{lemma}\label{subaff}
For an irreducible, 1-dimensional rigid analytic space $X$, the following are equivalent:
\begin{enumerate}[label=(\roman*)]
\item $t$ is induced,
\item there exists an affinoid subspace $X_0\subset X$ such that $t_X\vert_{X_0}$ is induced. 
\end{enumerate}
\end{lemma}

\begin{proof}
The implication (i)$\implies$(ii) is obvious. 
For the converse, recall that the locus 
\[ X_\ind=\{x\in X\,\vert\,t_x\textrm{ is induced}\} \]
is closed by Corollary \ref{indclo}. Since $X$ is irreducible and 1-dimensional, the points of any affinoid subdomain of $X$ are dense in $X$, hence the result.
\end{proof}

The rest of the section is devoted to the proof of Theorem \ref{locthm}. 
The first step consists in attaching to $t_U$ a $p$-adic family of Galois representations, which can only be done by passing to a suitable cover of $U$. 

\begin{lemma}\label{locfreesh}
	There exists a diagram of rigid analytic spaces
	\[ X\xrightarrow{f} Y\xrightarrow{g} U \]
	and a locally free $\cO_{X}$-module $V$ of rank $2$ equipped with a continuous $\cO_{X}$-linear action of $G_{\Q_p}$, such that:
	\begin{enumerate}[label=(\arabic*)]
		\item $Y$ is an integral, normal affinoid and $g\colon Y$ to $U$ is a finite dominant map; 
		\item $f\colon X\to Y$ is the blow-up of $Y$ along a closed subspace $Y_0$; 
		\item\label{blowup3} the trace of the action $\rho_X$ of $G_{\Q_p}$ on $V_{X}$ is $(gf)^\ast t_U$.
	\end{enumerate}
\end{lemma}

\begin{proof}
	This is the result of \cite[Lemma 7.8.11(i,ii)]{bellchen} applied to the reduced affinoid $U$ and the pseudorepresentation $t_U$.
\end{proof}

Given a field extension $L^\prime$ of $L$ and a point $x$ of $X(L^\prime)$, we denote by $V_x$ the fiber of $V$ at $x$, equipped with the $L^\prime$-linear action of $G_{\Q_p}$ induced by the action of $G_{\Q_p}$ on $V$. It is a 2-dimensional, continuous $L^\prime$-linear representation of $G_{\Q_p}$.

Consider the pseudorepresentation $t_X\coloneqq (gf)^\ast t_A\colon G_{\Q_p}\to\cO_X(X)$. 

\begin{lemma}\label{tXind}
The pseudorepresentation $t_X$ is induced if and only if $t_U$ is.
\end{lemma}

\begin{proof}
By Corollary \ref{indclo}(ii), $t_X$ is induced if and only if $t_{X,x}$ is induced for every $x\in X$. At every $x\in X$, $t_{X,x}=t_{U,gf(x)}$ by definition, so $t_X$ is induced if and only if $t_{U,y}$ is induced for every $y$ in the image of $gf$. In particular, if $t_U$ is induced then $t_X$ is induced. Conversely, assume that $t_X$ is induced, and consider the subset
\[ U_\ind=\{y\in U\,\vert\, t_{U,y}\textrm{ is induced}\}. \] 
By Corollary \ref{indclo}(i), $U_\ind$ is closed, and by the above considerations, it contains the image of $gf$. Since $g$ is surjective and $f$ is dominant, the image of $gf$ is dense in $U$, so that $U_\ind=U$.
\end{proof}

Define a subset $S_X$ of $X$ as $f^{-1}(g^{-1}(S)\setminus Y_0)$. Since $A$ is 1-dimensional, $g$ is finite and $f$ is a blow-up, both $X$ and $Y$ are 1-dimensional, and the closed subspace $Y_0$ in the 1-dimensional affinoid $Y$ consists in a finite union of points. In particular, the set $g^{-1}(S)\setminus Y_0$ is infinite, hence dense in the 1-dimensional affinoid $Y$. If $D$ is the exceptional fiber of $f$, then $f$ is an isomorphism on $X\setminus D$, and $X\setminus D$ is dense in $X$ since $X$ is irreducible and 1-dimensional, so the subset $S_X\subset (X\setminus D)(\C_p)\subset X(\C_p)$ is dense in $X$. 
In particular, $X$ admits an admissible covering by $\{U_i\}_i$ by integral affinoid subdomains such that $S_X\cap U_i$ is dense in $U_i$ for every $i$. Pick an element $U_i$ of such a covering. 

\begin{rem}\label{Xst}
For every $x\in S_X$, $gf(x)\in S$ and $t_{X,x}=t_{A,gf(x)}$. By assumption \ref{hyp1}, the pseudorepresentation $t_{A,gf(x)}$ is potentially semistable, so the same holds for $t_{X,x}$.
\end{rem}

Since $X$ is the blow-up of an integral rigid analytic space along a closed subspace, it is still integral. 
By Lemma \ref{subaff}, $t_U$ is induced if and only if $t_X$ is, and by Lemma \ref{tXind}, $t_X$ is induced if and only if $t_{U_i}$ is. Therefore, in order to deduce Theorem \ref{locthm}, it is enough to prove that $t_{U_i}$ is $E$-induced for a quadratic extension $E/\Q_p$.

In order not to burden the notation, we implicitly replace the original $U$ with $U_i$, $t_U$ with $t_X\vert_{U_i}$, $S$ with $S_X\cap U_i$, and the vector bundle $V$ with $V\vert_{U_i}$. Note that the new data obviously still satisfies assumptions \eqref{ncst} and \ref{hyp1}-\ref{hyp3}.

For every $x\in S$, $t_{U,x}$ is potentially semistable by Remark \ref{Xst}, so Corollary \ref{pstps} implies that there exists a quadratic extension $E_x$ of $\Q_p$ such that $t_{U,x}\vert_{G_{E_x}}$ is trianguline. Since $\Q_p$ only admits a finite number of quadratic extensions, there exists a quadratic extension $E$ of $\Q_p$ such that the set
\[ S_{E}=\{x\in S\,\vert\, t_{U,x}\vert_{G_E}\textrm{ is trianguline}\} \]
is infinite, hence dense in the 1-dimensional, irreducible affinoid $U$. We implicitly replace $S$ with $S_{E}$. 

By our choice of $E$, the $(\varphi,\Gamma_E)$-module $D_\rig(V_x\vert_{G_E})$ is triangulable for every $x\in S$. 
Recall that, by assumption \ref{hyp2}, we are given two functions $\varphi_1$ and $\varphi_2$ specializing at every $x\in S$ to the eigenvalues of the Frobenius operator on $D_\st(V_x\vert_{G_F})$, and by \eqref{ncst} a function $k\in\cO_U(U)$ that specializes to a Hodge--Tate--Sen weight of $V_x$ for every $x\in S$, the other weight being 0. 

We denote by $\bx$ the identity character $\Qp^\times\to\Qp^\times$, and we use the same notation for its restriction to a character $L^\times\to L^\times$ for any $p$-adic field $L$. We write $\pi_E$ for a uniformizer of $E$.

Let $\delta_{x,1},\delta_{x,2}\colon E^\times\to\C_p^\times$ be the parameters of a triangulation of $V_x\vert_{G_E}$. 
Let $k_1$ be the function on $U$ constantly equal to 0, and $k_2=k$. By Example \ref{eigenpar}, up to swapping the roles of $\delta_{x,1}$ and $\delta_{x,2}$, we have for $i=1,2$: 
\begin{enumerate}
\item $\delta_{x,i}\vert_{\cO_E^\times}=\psi_{x,i}\bx^{k_i(x)}$ for a finite order character $\psi_{x,i}\colon \cO_E^\times\to\Qp^\times$, 
\item $\delta_{x,i}(\pi_E)=\varphi_{1,i}(x)$.
\end{enumerate}

Let $\Lambda^2D_\rig(V)$ be the determinant of $D_\rig(V)$, a 1-dimensional $(\varphi,\Gamma_E)$-module over $\cR_U(\pi_E)$. By \cite[Theorem 6.2.14]{kedpotxia}, $\Lambda^2D_\rig(V)$ is isomorphic to $\cR_U(\pi_E)(\delta)\otimes_E\cL$ for a character $\delta\colon G_E\to\cO_U(U)^\times$, and a line bundle $\cL$ over $U$ with trivial actions of $\varphi$ and $\Gamma_E$. 

At every $x\in S$, 
\[ \delta=\delta_{1,x}\delta_{2,x}=\psi_{x,1}\psi_{x,2}\bx^{k_{x,1}+k_{x,2}}. \]
By Remark \ref{condbound}, the conductor of (the Dirichlet character attached to) $\psi_{x,1}\psi_{x,2}$ is bounded on $U$, so that there exists a finite extension $\doublewidetilde{E}/\Q_p$ such that $\psi_{x,1}\psi_{x,2}$ factors through $G_{\Q_p}\onto G_{\doublewidetilde{E}}$. In particular, \ref{hyp3} also holds with $\psi_{x,1}$ replaced by $\psi_{x,2}$, as long as we replace $\wtl E$ with the Galois closure of $\wtl E\doublewidetilde{E}$. We do so implicitly from now on.

Since both $\psi_{x,1},\psi_{x,2}$ factor through the finite quotient $G_{\Q_p}\onto G_{\wtl E}$, there are only a finite number of possible choices for each of $\psi_{x,1},\psi_{x,2}$. Therefore, up to replacing $S$ with a smaller infinite set, we can assume that $\psi_{x,1}$ and $\psi_{x,2}$ are both independent of $x\in S$. We denote by $\psi_1,\psi_2\colon G_{\Q_p}\to\cO_U(U)^\times$ the constant finite order characters that specialize to $\psi_{x,1},\psi_{x,2}$ at every $x\in S$. Then, for $i=1,2$, the characters $\delta_{x,i}$ are interpolated by the continuous character $\delta_i\colon E^\times\to\cO_U(U)^\times$ defined by
\begin{enumerate}
	\item $\delta_{i}\vert_{\cO_E^\times}=\psi_i\bx^{k_i}$, 
	\item $\delta_{i}(\pi_E)=\varphi_{1,i}$.
\end{enumerate}
At every $x\in S$, the characters $\delta_1,\delta_2$ specialize to the parameters $\delta_{x,1}, \delta_{x,2}$ of a triangulation of $D_\rig(V\vert_{G_E})$. 
We write $\delta_{1,x}$ and $\delta_{2,x}$ for the specializations of $\delta_1$ and $\delta_2$, respectively, at every $x\in U(\Qp)$.

\begin{lemma}\label{ptriclo}
For every $x\in U(\Qp)$, $t_x\vert_{G_E}$ is trianguline of parameters $\delta_{1,x}$ and $\delta_{2,x}$.
\end{lemma}

\begin{proof}
We simply apply \cite[Theorem 6.3.13]{kedpotxia} to, in the notations of \emph{loc. cit.}, $X=U$, $M=D_\rig(V)$, $X_\alg=S$ and the characters $\delta_1,\delta_2$ are as above.
\end{proof}

\begin{lemma}\label{triclo}
	The locus $U_\tri=\{x\in U\,\vert\, t_x\text{ is trianguline}\}$ is Zariski closed in $U$.
\end{lemma}

\begin{proof}
Let $x\in U(\Qp)$, and assume that $t_x$ is trianguline of some parameters $\wtl\delta_{1,x},\wtl\delta_{2,x}\colon\Q_p^\times\to\cO_U(U)^\times$. The corresponding triangulation on $D_\rig(V_x)$ induces a triangulation on $D_\rig(V_x\vert_{G_E})$, of parameters obtained by composing $\wtl\delta_{1,x}$ and $\wtl\delta_{2,x}$ with the norm map $\Nm_{E/\Q_p}\colon E^\times\to\Q_p^\times$. On the other hand, $t_x\vert_{G_E}$ is trianguline of parameters $\delta_{1,x}$ and $\delta_{2,x}$ by Lemma \ref{ptriclo} if $\delta_{1,x}$ and $\delta_{2,x}$ factor through $\Nm_{E/\Q_p}$, then $D_\rig(V_x)$ is triangulable. Hence, $U_\tri$ is the locus of points such that $\delta_{1,x}$ and $\delta_{2,x}$ factor through $\Nm_{E/\Q_p}$, which is closed. 
\end{proof}


\begin{cor}\label{ntridense}
	The complement of $U_\tri$ is open and dense in $U$.
\end{cor}

\begin{proof}
	The statement follows from Lemma \ref{triclo} together with the fact that the set $S$ is dense in $U$ and not contained in $U_\tri$. 
\end{proof}

\begin{cor}\label{ptriall}
	The subspace $U_\ptri$ coincides with the whole $U$.
\end{cor}

\begin{proof}
	By assumption \ref{hyp1} the pseudorepresentation $t_{U,x}$ is potentially trianguline for every $x\in S$, so that $S\subset U_\ptri$. Since $S$ is Zariski-dense in $U$, the result follows from Lemma \ref{ptriclo}.
\end{proof}




Our theorem will follow from the previous lemmas combined with the result of Berger and Chenevier that we recalled in Theorem \ref{berche}.





\begin{proof}[Proof of Theorem \ref{locthm}] 
Let $U_\tdR$ be the locus of points in $U$ satisfying condition \eqref{tdr} of Theorem \ref{berche}.
	Consider the locus
	\[ U_{\ptnt}\coloneqq U_\ptri\setminus U_\tri, \]
	of points $x$ where $t_x$ is potentially trianguline but not trianguline. By Lemma \ref{ptriall}, $U_\ptri=U$, so that by Corollary \ref{ntridense} $U_\ptnt$ is open and dense in $U$. 
	By Theorem \ref{berche}, $U_{\ptnt}$ is the union of $U_\ind$ and $U_\tdR$. 
	
	For $x\in U_\tdR$, the Hodge--Tate--Sen weight $k(x)$ is an integer. 
Since $k$ is not constant over $U$, the locus
	\[ U_{\mathrm{int}}=\{x\in U(\C_p)\,\vert\,k(x)\in\Z\} \]
	does not contain any Zariski open subset of $U(\C_p)$, and neither does $U_\tdR$ since it is a subset of $U_{\mathrm{int}}$. By Corollary \ref{indclo}, $U_\ind$ is closed. Assume by contradiction that $U_\ind\ne U$. 
	As we remarked above, $U_\ptnt=U_\ind\cup U_\tdR$ is open and dense in $U$, so that $(U_\ind\cup U_\tdR)\setminus U^\ind$ is open in $U$ and nonempty: indeed, if it were empty, then $U_\ind=U_\ind\cup U_\tdR$ would be both open and closed, hence the whole $U$ since $U$ is connected. On the other hand, $(U_\ind\cup U_\tdR)\setminus U^\ind\subset U_\tdR$ cannot contain any nonempty open set, a contradiction. We conclude that $U_\ind=U$.
\end{proof}




\medskip

\section{Families of $p$-supercuspidal eigenforms}\label{secpsc}

Given a $p$-adic family of $p$-supercuspidal eigenforms, we prove that it is a CM family in the sense of Definition \ref{CMfamdef}. 
We will deduce our result, Theorem \ref{infslope}, from the local Galois-theoretic result Theorem \ref{locthm}. 

We fix a positive integer $N$, prime to $p$, that will serve as a tame level. 
Though the domain of most (pseudo-)representations we consider below is $G_{\Q,Np}$, we will sometimes abuse and lighten notation by restricting such (pseudo-)representations to $G_K$, $K$ a finite extension of $\Q$, when what we mean is that we are restricting them to $\Gal(\Q^{Np}/K)$, $\Q^{Np}$ being the maximal extension of $\Q$ in $\ovl\Q$ unramified outside of $Np$.

Let $(f_i,\varphi_i)_i$ be a family of eigenforms of tame level $\Gamma_1(N)$. For every $i$, let $t_i\colon G_{\Q,Np}\to\Qp$ be the pseudorepresentation attached to $f_i$. Since the eigenforms $f_i$ are all pairwise congruent modulo $p$ for $i$ large enough, the residual pseudorepresentations $\ovl t_i$ all coincide for $i$ large enough. We denote by $\ovl t$ their common value and by $\ovl\rho\colon G_{\Q,Np}\to\GL_2(\ovl\F_p)$ any continuous representation with trace $\ovl t$. We assume from now on that:
\begin{enumerate}[label=(H$\ovl t$\arabic*)]
\item \label{hyprhobar1}
$\ovl t$ is not $\Q(\sqrt p)$-induced;
\item \label{hyprhobar2}
if $p=5$, $F$ is a quadratic extension of $\Q$, ramified at 5, such that $\ovl t\vert_{G_{F_\fp}}$ is decomposable for the unique 5-adic place $\fp$ of $F$, and $\Proj(\ovl\rho(G_{F,5N}))\cong\PGL_2(\F_5)$, then the degree of $F(\zeta_5)/F$ is 4.
\end{enumerate}
 
 

Our main result is the following.

\begin{thm}\label{infslope}
If $f_i$ is $p$-supercuspidal for every $i$, then the family $(f_i,\varphi_i)_i$ is eventually CM. 
\end{thm}

We devote the rest of the section to the proof of Theorem \ref{infslope}.

By Proposition \ref{famlim}, we can partition $\N$ as $\N=\coprod_{j=0}^mI_i$ such that $I_0$ is finite and, for every $j\in\{1,\ldots,m\}$, $(f_i,\varphi_i)_{i\in I_j}$ is the specialization of an affinoid family $\cF_j$. Theorem \ref{infslope} is an immediate consequence of the following lemma, together with Lemma \ref{CMfamrep}(ii).

\begin{lemma}\label{FjCM}
For every $j\in\{1,\ldots,m\}$, the family $\cF_j$ is CM, in the sense of Definition \ref{CMfamdef}.
\end{lemma}

We proceed to proving Lemma \ref{FjCM}. In what follows, we pick an arbitrary $j\in\{1,\ldots,m\}$ and simply write $(f_i,\varphi_i)_i$ for $(f_i,\varphi_i)_{i\in I_j}$. Let $U$ be the integral $L$-affinoid underlying $\cF_j$, and let $t_U\colon G_{\Q,Np}\to\cO_U^\circ(U)$ be the pseudorepresentation attached to $\cF_j$. Clearly $\ovl t_U=\ovl t$.


We use the local results of Section \ref{secloc} to deduce the following. 

\begin{lemma}
The pseudorepresentation $t_U\vert_{G_{\Q_p}}$ is $E$-induced for a ramified quadratic extension $E/\Q_p$.
\end{lemma}

\begin{proof}
Let $S$ be the infinite set of points of $U$ corresponding to the eigenforms $f_i$. The statement will follow from Theorem \ref{locthm} applied to the pseudorepresentation $t_U\vert_{G_{\Q_p}}$ (we warn the reader that the pseudorepresentation $t_U$ appearing in the statement of Theorem \ref{locthm} is a local one, contrary to that of the current section). In order to do so, it is enough check that $t_U\vert_{G_{\Q_p}}$, together with our choice of set $S$, satisfies the assumptions \eqref{ncst} and \ref{hyp1}-\ref{hyp3} of the theorem:
\begin{itemize}
	\item \eqref{ncst} is true since the eigenforms $f_i$ have weights that tend to $\infty$ as $i\to\infty$;
	\item \ref{hyp1} holds thanks to Corollary \ref{infpottri}, since $f_i$ is $p$-supercuspidal for every $i$;
	\item \ref{hyp2} is verified by the elements $\Phi,\Phi^{-1}\in\cO_U(U)^\times$; 
	\item for every $i$, the base change $f_{i,F}$ is the twist of a $\GL_{2/F}$-eigenform of finite slope by the Dirichlet character attached to either $\psi_{x,1}\vert_{G_F}$ or $\psi_{x,2}\vert_{G_F}$; the conductor of such a character is bounded as i varies by Proposition \ref{untwistfam}. Since the conductor of $\psi_{x,1}\psi_{x,2}$ is bounded for $x\in U(\Qp)$ by Remark \ref{condbound}, \ref{hyp3} follows.
\end{itemize}
\end{proof}

Let $F$ be a real quadratic extension of $\Q$ such that $F\otimes_\Q\Q_p=E$. One such $F$ always exists: simply pick any quadratic field $F_0=\Q(\sqrt d)$ such that $F_0\otimes_\Q\Q_p=E$, and if $d<0$, replace it with a negative integer prime to $p$ that is a square modulo $p$. 

Since $t_U\vert_{G_{\Q_p}}$ is $E$-induced, the restriction $t_U\vert_{G_E}$ is decomposable, hence trianguline, everywhere on $U$. In particular, Proposition \ref{extract} provides us with a closed embedding $\iota\colon U\to\cE_F$ identifying $U$ with an affinoid subdomain of the eigenvariety $\cE_F$, and the family $\cF_j$ with that supported by $\iota(U)$. We identify $U$ with $\iota(U)$ implicitly from now on. Note that since the $f_i$ are assumed to be $p$-supercuspidal, they are of infinite slope, so that one cannot choose $F=\Q$ in Proposition \ref{extract}.

\begin{lemma}\label{UCM}
The affinoid $U$ is contained in a CM irreducible component of $\cE_F$.
\end{lemma}



\begin{proof}
The conclusion follows from Theorem \ref{bgvaff}, once we check that $U$ satisfies the assumptions therein. 
Let $\fp$ be the unique $p$-adic place of $F$. Since $F_\fp=E$, the pseudorepresentation $t_U\vert_{G_{F_\fp}}$ is decomposable, and condition \ref{Sab} of the theorem is satisfied. Since $t_U\vert_{G_{\Q_p}}$ is induced, so is $\ovl t_U\vert_{G_{\Q_p}}$, hence $\ovl t_U$ is irreducible.


If $\ovl t_U$ becomes reducible after restriction to $G_K$, $K$ a finite extension of $\Q$, 
then by Lemma \ref{ribetind}(iii) it is $K$-induced for a quadratic extension $K/\Q$. By assumption \ref{hyprhobar1}, one cannot take $K=\Q(\sqrt p)$, and as a consequence $\ovl t\vert_{G_{\Q(\sqrt p)}}$ is still irreducible, so that assumption \ref{zetapirr} of Theorem \ref{bgvaff} holds.
	
Finally, assumption \ref{5hyp} is satisfied because of assumption \ref{hyprhobar2}.
\end{proof}	


We are ready to conclude the proof of Lemma \ref{FjCM}, and as a consequence that of Theorem \ref{infslope}. By Lemma \ref{UCM}, $U$ is contained in an irreducible component of $\cE_F$ that has CM by some totally imaginary quadratic extension $K$ of $F$. In particular, $t_U\vert_{G_K}$ is decomposable, meaning that $t_U$ is decomposable up to index 4. 
By Lemma \ref{ribetind}(iii), $t_U$ must be decomposable up to index $2$. As remarked in the proof of Lemma \ref{UCM}, $\ovl t_U$ is irreducible, so that $t_U$ is also irreducible, but becomes decomposable after restriction to a subgroup of index 2 of $G_{\Q,Np}$. We apply Lemma \ref{CMfamrep}(i) to conclude that $U$ is an affinoid CM family, thereby proving Lemma \ref{FjCM}.


\bigskip

\printbibliography

\bigskip

\bigskip

\noindent \textsc{Andrea Conti, University of Luxembourg, Esch-Sur-Alzette, Luxembourg,} \\
\url{andrea.conti@uni.lu}, \url{contiand@gmail.com}
\end{document}